%% file: Singular.tex
\documentclass[10pt,twoside,a4paper]{amsart}

\input{MySetting}

\def\Pol{{\textnormal{Pol}^\circ}}
\def\PolC{{\textnormal{Pol}_C^\circ}}
\def\Rect{{\textnormal{Rect}^\circ}}

\begin{document}

\title{Singular link Floer homology}
\author{Benjamin \textsc{Audoux}}
\date{\today}
\address{Section de Math\'ematiques, Unige\\ rue du li\`evre 2-4, 1211 Gen\`eve 4, Switzerland}
\email{benjamin.audoux@unige.ch}

\begin{abstract}
We define a grid presentation for singular links, \ie links with a finite number of rigid transverse double points. Then we use it to generalize link Floer homology to singular links. Besides the consistency of its definition, we prove that this homology is acyclic under some conditions which naturally make its Euler characteristic vanish.
\end{abstract}

\maketitle

\input{Introduction}

\saut
\input{Grids}

\newpage
\input{Homology}

\saut
\input{Properties}

\saut
\input{Computations}

\saut
\bibliographystyle{amsalpha}
\bibliography{Singular}
\addcontentsline{toc}{part}{Bibliography}

\end{document}

%% file: MySetting.tex
\usepackage[english]{babel}
\usepackage[latin1]{inputenc}
\usepackage{amsfonts,amsthm,amssymb,amsmath,latexsym,wasysym,stmaryrd,mathrsfs}
\usepackage{multirow,rotating,subfigure,picins,graphicx,bigstrut,lscape}
\usepackage{hyperref}
\usepackage[centerlast]{caption}

\usepackage{fancyhdr}

\usepackage{pdftricks}
\usepackage{color}
\usepackage[all,color]{xy}

\graphicspath{{Figures/}}

\newtheorem {theo}{Theorem}[section]
\newtheorem {theo*}{Theorem}[]
\newtheorem {lemme}[theo]{Lemma}
\newtheorem {lemme*}[theo*]{Lemma}
\newtheorem {prop}[theo]{Proposition}
\newtheorem {prop*}[theo*]{Proposition}

\newtheorem {cor_proof}[theo]{Corollary (of the proof)}
\newtheorem {cor*}[theo*]{Corollary}
\newtheorem {cor_proof*}[theo*]{Corollary (of the proof)}
\newtheorem {conjecture}{Conjecture}[]
\theoremstyle{definition}
\newtheorem {defi}[theo]{Definition}
\newtheorem {defi*}[theo*]{Definition}

\newtheorem {nota*}[theo*]{Notation}

\newtheorem {question*}[theo*]{Question}
\theoremstyle{remark}
\newtheorem {remarque}[theo]{Remark}
\newtheorem {remarque*}[theo*]{Remark}

\newtheorem {exemple*}[theo*]{Example}

\newtheorem {fact*}[theo*]{Fact}

\def\e{\varepsilon}

\def\I{{\mathcal I}}

\def\J{{\mathcal J}}

\def\L{{\mathcal L}}

\def\N{{\mathbb N}}
\def\O{{\mathbb O}}

\def\R{{\mathbb R}}
\def\S{{\mathbb S}}

\def\T{{\mathcal T}}
\def\X{{\mathbb X}}
\def\Z{{\mathbb Z}}
\def\2Z{{\fract{\Z}/{2\Z}}}

\def\p{\partial}

\def\Id{{\textnormal{Id}}}

\def\Int{{\textnormal{Int}}}

\def\deg{{\textnormal{deg}}}
\def\gr{{gr}}
\def\ie{{\it i.e. }}

\def\rk{{\textnormal{rk}}}

\newcommand{\noi}{\noindent}
\newcommand{\disp}{\displaystyle}
\newcommand{\saut}{\vspace{\baselineskip}}

\def\fract#1/#2{\hbox{\leavevmode
  \kern.1em \raise .25ex \hbox{\the\scriptfont0 $#1$}\kern-.1em }\big/
  {\hbox{\kern-.15em \lower .5ex \hbox{\the\scriptfont0 $#2$}} }}

\newcommand{\dessin}[2]{
  \vcenter{\hbox{\includegraphics[height=#1]{#2.pdf}}}}
\newcommand{\dessinH}[2]{
  \vcenter{\hbox{\includegraphics[width=#1]{#2.pdf}}}}

\newcommand{\func}[3]{
  #1 \colon #2 \longrightarrow #3}

\newcommand{\function}[5]{
  #1 \colon
  \begin{array}{ccc}
    #2 & \longrightarrow & #4\\[.2cm]
    #3 & \longmapsto & #5
  \end{array}}


%% file: Introduction.tex
\section*{Introduction}
\label{sec:Introduction}

Since the Jones polynomial was categorified in 1999 by Mikhail Khovanov \cite{Khovanov}, knot and link invariants of homological type have constantly been growing to become one of the most flourishing and promising fields in knot theory.
Practically, categorifying a polynomial invariant $\lambda$ means, for a presentation of a link $L$, defining a graded chain complex $\L=(\L_i^j)_{i,j\in\Z}$ such that
\begin{enumerate}
\item[i)] the graded Euler characteristic $\xi_\gr (\L) = \displaystyle{\sum_{i,j}}(-1)^i \rk (\L_i^j)q^j$ is equal to $\lambda(K)$;
\item[ii)] the homology $H_*(\L)$ depends only on the link $L$.
\end{enumerate}

Categorification is worthwhile since it defines a new link invariant which sharpens the information given by its single Euler characteristic.
Typically, $H_*(\L)$  distinguishes more links and provides equalities when $\lambda$ gives only bounds.
Moreover, it is usually endowed with good functorial properties with regard to the category of cobordisms.

\saut

In this context, the question of a Vassiliev-like theory for invariants of homological type has been raised.
Actually, any link polynomial invariant can be naturally extended to singular links, \ie links with a finite number of rigid transverse double points, using the following recursive formula
\begin{equation}
  \lambda(\dessin{.5cm}{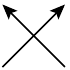}) := \lambda(\dessin{.5cm}{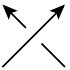}) - \lambda(\dessin{.5cm }{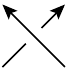}),
  \label{eq:Vassiliev}  
\end{equation}
where the three pictures should be understood as pieces of link diagrams which are identical outside the represented crossing.

\emph{Finite type invariants of order $k\in \N$} are then defined as the polynomial invariants which vanish for every knot with at least $k+1$ double points.
Finiteness defines a filtration on polynomial invariants.
Most known invariants are combinations of finite type invariants \cite{Birman}.
To date, the question is still open to know whether the union of all finite type invariants is strong enough to distinguish all knots.

A similar theory for link invariants of homological type should imply an exact triangle
$$
\xymatrix@!0@C=1.5cm@R=1.5cm{
  H_*\left(\L\left(\dessin{.5cm}{Pos}\right)\right) \ar[rr] && H_*\left(\L\left(\dessin{.5cm}{Neg}\right)\right) \ar[dl]\\
  & H_*\left(\L\left(\dessin{.5cm}{Double}\right)\right) \ar[ul]&
}.
$$
which categorifies the relation (\ref{eq:Vassiliev}).
Exact triangles arise naturally when dealing with mapping cone of chain maps.
Hence, as proposed by N. Shirokova \cite{Shiro}, a strategy can be to consider a wall-crossing map $\func{f}{\L\left(\dessin{.5cm}{Pos}\right)}{\L\left(\dessin{.5cm}{Neg}\right)}$.
Unfortunately, there is no canonical way to define such a map and, among the candidates, a selection has be done.
Polynomial invariants automatically vanish for links with a singular loop.
It is then justified to require that the homology also vanishes for such links.
$$
\lambda\left(\ \dessin{.9cm}{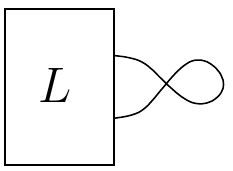}\right) = 0 \ \xrightarrow{requirement}\ H_*\left(\L\left(\ \dessin{.9cm}{SmallLoop}\right)\right) \equiv 0
$$

\saut

The purpose of this paper is to prove the following theorem:
\begin{theo*}
  There exists a generalization $\widehat{HFV}$ of the link Floer homology $\widehat{HF}$ with $\Z$ coefficients to singular links with oriented double points, which categorifies the relation (\ref{eq:Vassiliev}) and vanishes for links with a singular loop.
\end{theo*}

An orientation for a double point is a choice of orientation for the plane spanned by the two tangent vectors at this double point.
Link Floer homology is a categorification of the Alexander polynomial.
It has been defined in 2004 by Peter Ozsv\'ath \& Zoltan Szab\'o \cite{OS2} and, independantly, by Jacob Rasmussen \cite{Rasmussen}.
It appeared to be particularly rich since it detects the unknot, the trefoils, the figure eight knot, fiberedness and Seifert genus.
In 2006, link Floer homology has been given an alternative description, which is combinatorial in nature \cite{MOS}, \cite{MOST}.
In the next paragraph, we briefly review this construction which is based on the grid presentation for links.

Another generalization $HFS$ of link Floer homology to singular links is given in \cite{OSSing}; but actually, the two approaches differ in motivation and in construction.
Moreover, they satisfy different exact triangles.

\begin{figure}[b]
  \begin{gather*}
    \hspace{-.6cm}
    \begin{array}{ccc}
      \xymatrix@C=2cm{\dessin{2.1cm}{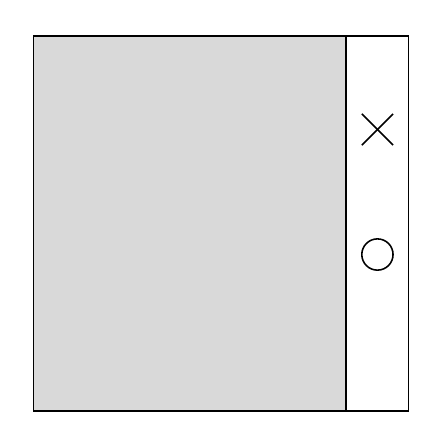} \ar@{<->}[r]^{\dessin{.9cm}{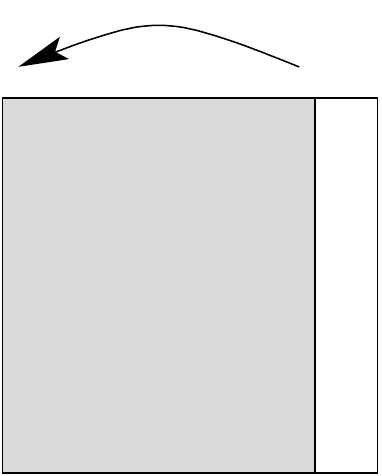}} & \dessin{2.1cm}{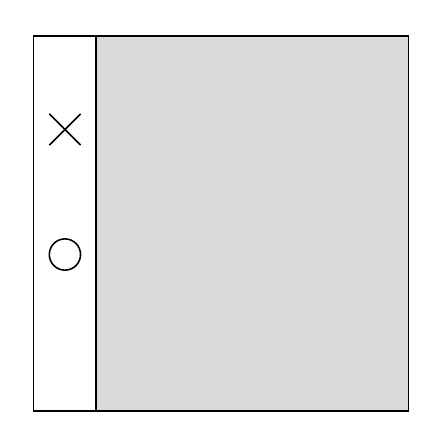}}
      &  &
      \xymatrix@C=2cm{\dessin{2.1cm}{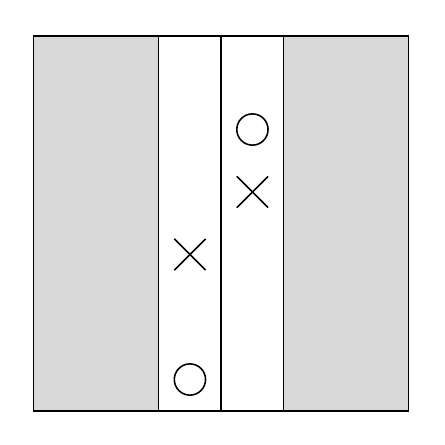} \ar@{<->}[r]^{\dessin{.9cm}{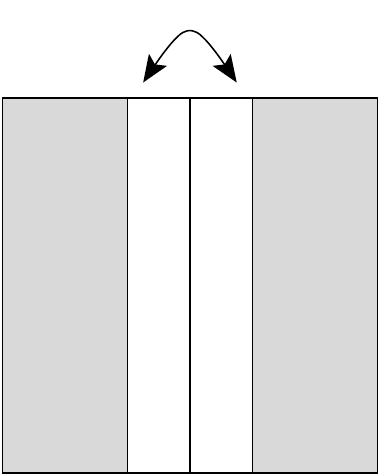}} & \dessin{2.1cm}{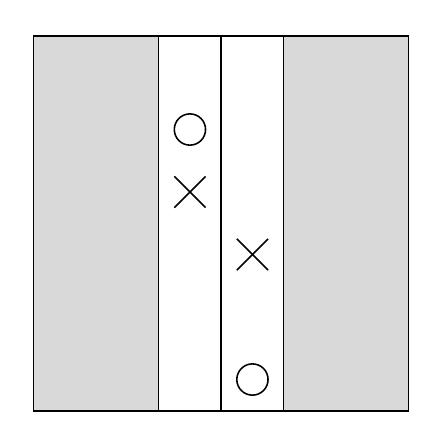}}\\[1cm]
      \textrm{Cyclic permutation} && \textrm{Commutation}
    \end{array}\\[.6cm]
    \begin{array}{c}
      \xymatrix@C=2cm{\dessin{2.1cm}{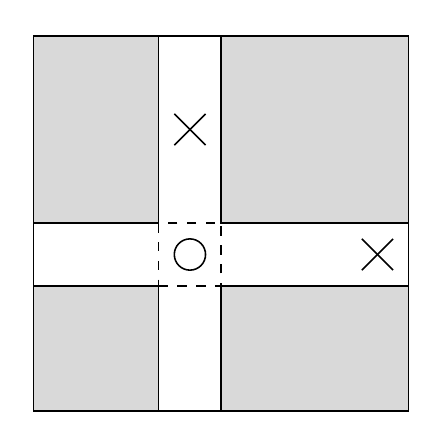} \ar@{<->}[r]^{\dessin{.9cm}{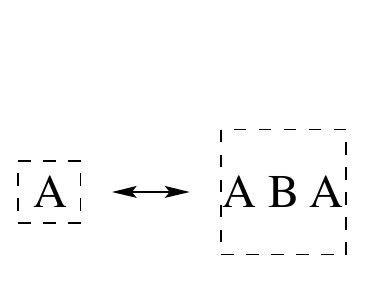}} & \dessin{2.1cm}{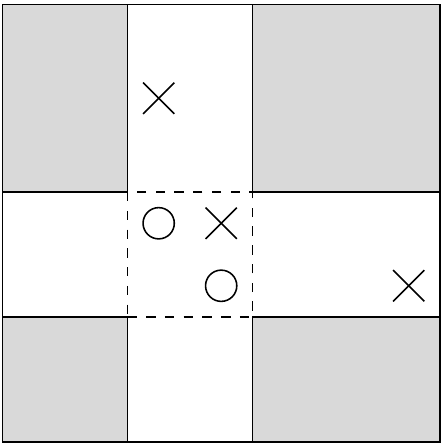}}\\[1cm]
      \textrm{Stabilization/Destabilization}
    \end{array}
  \end{gather*}
  \caption*{Elementary grid diagram moves}
  \label{fig:Regular_Elementary_Moves}
\end{figure}

\saut

A \emph{grid diagram} $G$ of size $n\in \N^*$ is a $(n\times n)$--grid  with some squares decorated by a $O$ or by an $X$ in such a way that each column and each row contains exactly one $O$ and one $X$.
We denote by $\O$ the set of $O$'s and by $\X$ the set of $X$'s.
A \emph{decoration} is an element of $\O\cup \X$.
From now on, we will write RoC as an abbreviation for ``row or column''.

An oriented link diagram is associated to any grid diagram.
To this end, we join, in each column, the two decorations by a straight line.
We do the same in each row with straight lines which underpass all the vertical ones.
Each decoration is then replaced by a right angled corner which can be smoothed.
By convention, the link is oriented by running the horizontal strands from the $\O$--decoration to the $\X$--one.

Reading Figure \ref{fig:G->K} from right to left shows that, up to isotopy, any grid diagram can be described in this way.
P. Cromwell \cite{Cromwell} and, later, I. Dynnikov \cite{Dynnikov} have proven that any two grid diagrams which describe the same link can be connected by a finite sequence of the following elementary grid moves:
\begin{description}
\item[Cyclic permutation] cyclic permutation of the RoCs;
\item[Commutation] commutation of two adjacent columns (resp. rows) under the condition that all the decorations of one of the two commuting columns (resp. rows) are strictly above (resp. strictly on the right of) the decorations of the other one;
\item[Stabilization/Destabilization] addition (resp. removal) of one column and one row by replacing (resp. substituting) locally a decorated square by (resp. to) a $(2\times 2)$--grid containing three decorations in such a way that it globally remains a grid diagram.
\end{description}

\begin{figure}[t]
  $$
  \dessin{2.3cm}{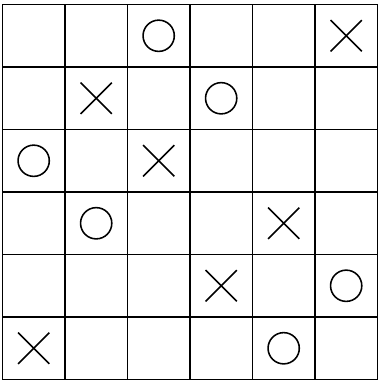} \hspace{.3cm} \leadsto \hspace{.3cm}\dessin{2.3cm}{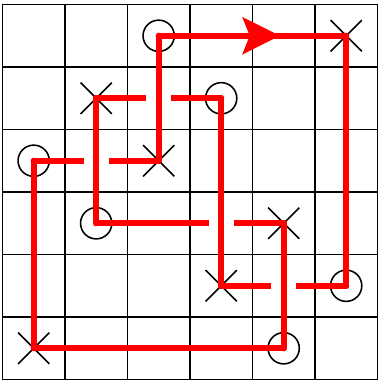} \hspace{.3cm} \leadsto \hspace{.3cm} \dessin{2.3cm}{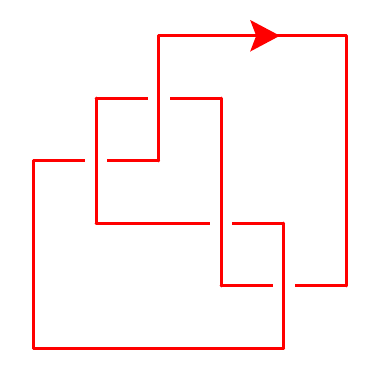} \hspace{.3cm} \leadsto \hspace{.3cm} \dessin{2.3cm}{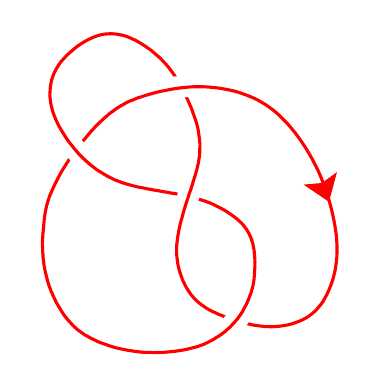}
  $$
  \caption{From grid diagrams to links}
  \label{fig:G->K}
\end{figure}

\saut

Let $G$ be a grid diagram of size $n$ for a link $L$ with $\ell$ components.
We define $C^-(G)$ as the $\Z[\{U_O\}_{O\in\O}]$--module generated by all one-to-one correspondences between the rows and the columns of $G$.
Every generator can be depicted on the grid by drawing a dot at the bottom left corner of each common square of associated row and column.
Then, generators are sets of $n$ dots arranged on the intersections of the grid lines such that every line contains exactly one point, except the rightmost and the uppermost ones which do not contain any.

Before turning $C^-(G)$ into a bigraded module, we need to introduce some definitions.
For $A$ and $B$ two finite subsets of $\R^2$, we define $\J(A,B)$ as half the number of pairs $((a_1,a_2),(b_1,b_2))\in A\times B$ satisfying $(b_1-a_1)(b_2-a_2)>0$ \ie
$$
\J(A,B):=\#\{(a,b)\in A\times B|a \textrm{ lies in the open south-west or north-east quadrants of }b\}.
$$
Then, we set $M_B(A):=\J(A,A)-2.\J(A,B)+\J(B,B)+1$.
Now, for every generator $x$ of $C^-(G)$ and every $\alpha\in \Z[\{U_O\}_{O\in\O}]$ we can set
\begin{itemize}
\item[-] $M(\alpha.x):=M_\O(x) - 2.\deg(\alpha)$;
\item[-] $A(\alpha.x):=\frac{1}{2}(M_\O(x)-M_\X(x)) - \frac{n-\ell}{2} - \deg(\alpha)$;
\end{itemize}
where $\deg(\alpha)$ is the total polynomial degree of $\alpha$ and decorations are assimilated to their centers of gravity.
The maps $M(\ .\ )$ and $A(\ .\ )$ are respectively called the \emph{Maslov} and the \emph{Alexander grading}.

A differential $\p_G^-$ which decreases $M$ by one and respects the filtration induced by $A$ can then be defined by counting rectangles.
To formalize this, we consider $\T_G$ the torus obtained by gluing together the opposite sides of $G$.
For two generators $x$ and $y$ of $C^-(G)$, a \emph{rectangle $\rho$ connecting $x$ to $y$} is an embedded rectangle in $\T_G$ which satisfies:
\begin{itemize}
\item[-] edges of $\rho$ are embedded in the grid lines;
\item[-] opposite corners of $\rho$ are respectively in $x\setminus y$ and $y\setminus x$;
\item[-] except on $\partial \rho$, the sets $x$ and $y$ coincide;
\item[-] according to the orientation of $\rho$ inherited from the one of $\T_G$, horizontal components of $\p\rho$ are oriented from points of $x$ to points of $y$.
\end{itemize}
A rectangle $\rho$ is \emph{empty} if $\Int(\rho)\cap x=\emptyset$.
We denote by $\Rect(G)$ the set of all empty rectangles on $G$ and by $\Rect(x,y)$ the set of those which connect $x$ to $y$.

\begin{figure}[ht]
  $$
  \begin{array}{ccccc}
    \dessin{3cm}{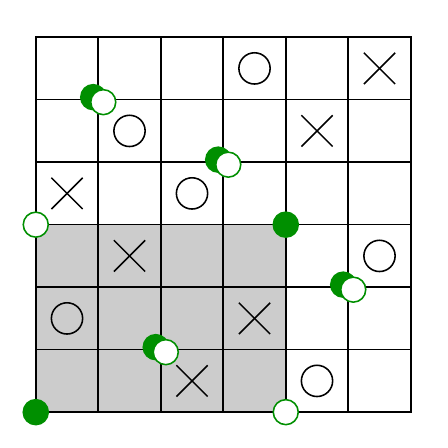} & & \dessin{3cm}{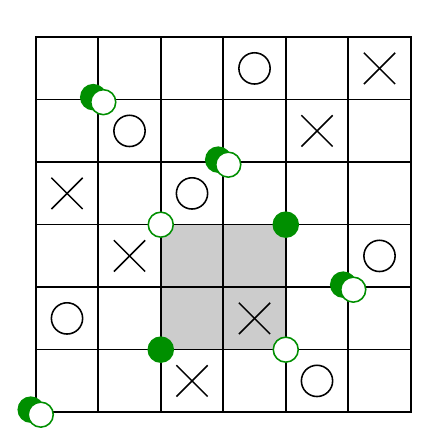} & & \dessin{3cm}{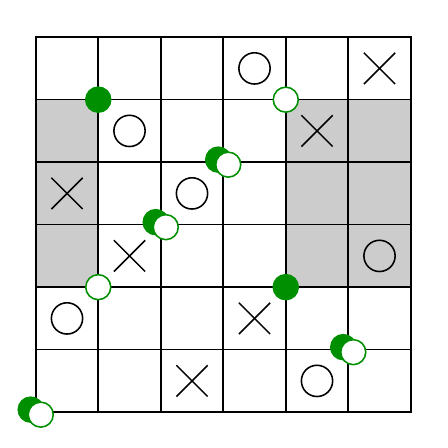}\\[1cm]
    \textrm{a non empty rectangle} && \multicolumn{3}{c}{\textrm{empty rectangles}}
  \end{array}
  $$
  \caption*{Examples of rectangles connecting $x$ to $y$: {\footnotesize dark dots describe the generator $x$ while hollow ones describe $y$. Rectangles are depicted by shading. Since a rectangle is embedded in the torus and not only in the rectangular grid, it may be ripped in several pieces as in the case on the right.}}
  \label{fig:Rect}
\end{figure}

Finally, we define the map $\func{\p_G^-}{C^-(G)}{C^-(G)}$ as the morphism of $\Z[\{U_O\}_{o\in\O}]$--modules defined on the generators of $C^-(G)$ by
$$
\p_G^-(x)=\sum_{y \textrm{ generator}\textcolor{white}{R}} \sum_{\rho\in \Rect(x,y)} \e(\rho)\Big(\prod_{O\in\O\cap\rho}U_O\Big).y,
$$
where $\func{\e}{\Rect(G)}{\{\pm 1\}}$ is a sign assigment which is defined in \cite{MOST} or, equivalently, in \cite{Gallais}.

\saut

The homology $H^-(G)$ is defined as $H_*\big(C^-(G),\p_G^-\big)$.
It is filtered by the Alexander grading.
If $\S\subset\O$ is a choice of one $\O$--decoration on each connected component of $L$, then the link Floer homology $\widehat{HF}(G)$ is defined as $H_*\left(\fract{\big(C^-(G),\p^-_G\big)}/{\{U_O\}_{O\in\S}}\right)$.
It is proven in \cite{MOST} that $\widehat{HF}(G)$ depends only on $L$ and that it categorifies the Alexander polynomial.

\saut

The paper is organized as follows.
In the first section, we generalize the grid presentation to singular links.
This is done by defining singular grids with singular RoCs, which contain four decorations instead of two.
Every singular RoC corresponds to a singular double point of the underlying link.
There are four ways to split a singular RoC into regular ones.
At the level of links, they correspond to the three orientation-preserving desingularizations of the associated double point.

Then we define a fourth elementary grid move, called \emph{rotation} (see Figure \ref{fig:Rotation_Moves}), and prove the following statement:
\begin{prop*}
  Every singular link admits a singular grid presentation.
  Moreover, any two grid diagrams for a singular link can be connected by a finite sequence of regular elementary grid moves and rotations.
\end{prop*}
\noi To prove this proposition, we use the diagrammatical approach of singular links given by L. Kauffman in \cite{Kauffman}.
Besides the three usual Reidemeister moves, this description adds two new ones which are illustrated in Figure \ref{fig:Singular_Reidemeister_Moves}.

In section \ref{sec:Homology}, we define a chain complex $\big(CV^-(G),\p_G^-\big)$ for a singular grid $G$ which generalizes the regular case.
Essentially, $CV^-(G)$ is, as a module, the grading-shifted direct sum of the chain complexes associated to all desingularizations of $G$.
Shifts increasing the Maslov and Alexander gradings by $l\in\Z$ are respectively denoted by $.[l]$ and $.\{l\}$.
The differential is defined by counting more general polygons on $G$.
For technical reason, the construction requires a choice of orientation for every singular RoC.
This choice can be reduces to a choice, on the underlying singular link $L$, of an orientation for all its double points \ie a choice of orientation for every plane spanned by the two transverse vectors which are tangent to $L$ at a double point.
\begin{figure}
  $$
  \begin{array}{ccc}
    \dessinH{1cm}{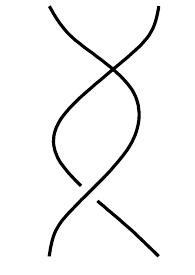} \leftrightarrow \dessinH{1cm}{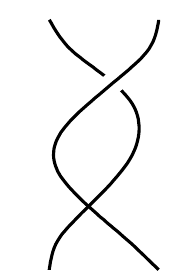} & \hspace{1.2cm} & \dessinH{1.4cm}{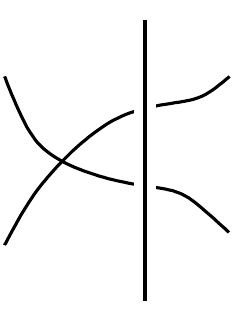} \leftrightarrow \dessinH{1.4cm}{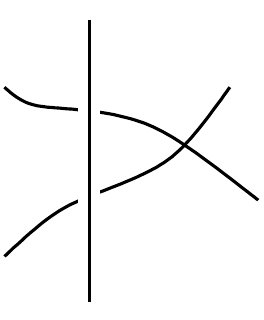}  \\[.5cm]
    \textrm{type IV} && \textrm{type V}
  \end{array}
  $$ 
  \caption{Singular Reidemeister moves}
  \label{fig:Singular_Reidemeister_Moves}
\end{figure}
Then we prove:
\begin{prop*}
  The homology $H_*\big(CV^-(G),\p_G^-\big)$ depends only on the underlying singular link $L$ with oriented double points.
\end{prop*}
\noi In section \ref{sec:Properties}, we discuss graded objects associated to the different filtrations and give the definition of singular link Floer homology $\widehat{HFV}$.
In this section we also prove a few symmetry properties which reflect some properties of Alexander polynomial.
Finally, we prove that $\widehat{HFV}$ is null for every link with a singular loop.

The paper is ended by computations of singular link homologies with $\fract\Z/{2\Z}$ coefficients, made with the help of a computer.
They lead to conjectures about conditions under which the singular link Floer homology should be null.

\saut

The present paper synthesizes some results of the author Ph.D thesis \cite{These}.
However, the reader should be warned that, for clarity reasons, notation and conventions may have been slightly modified.
Lastly, the author would like to sincerely thank the referee for all his or her comments, but also Peter Ozsv\'ath, Christian Blanchet, Thomas Fiedler and Etienne Gallais for interesting conversations and remarks.


%% file: Grids.tex
\section{Grid description for singular links}
\label{sec:Grids}

\input{Grids_Diagrams}

\input{Grids_Desingularization}


%% file: Grids_Diagrams.tex
\subsection{Singular grid diagrams}
\label{ssec:Grids_Diagrams}

A \emph{singular grid diagram} is a rectangular grid with some squares decorated by a $O$ or by an $X$ in such a way that each RoC contains exactly one or exactly two decoration(s) of each kind.
A RoC is called \emph{singular} if it contains four decorations and \emph{regular} otherwise.
Furthermore, in the former case, the two middle decorations are required to be surrounded by decorations of different kinds.
In other words, as we read the decorations from bottom to top (resp. from left to right) in a singular column (resp. singular row), we find one of the following: $OOXX$, $OXXO$, $XXOO$ or $XOOX$. 
We denote by, respectively, $\O$ and $\X$ the sets of $O$'s and $X$'s decorations of $G$.
The size of a grid is the cardinality of $\O$.
A singular grid diagram may have different numbers of lines and rows.

\begin{figure}[b]
  $$
  \dessin{2.3cm}{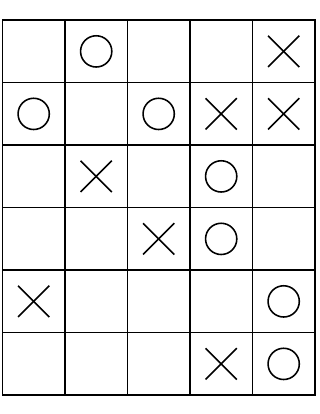} \hspace{.5cm} \leadsto \hspace{.5cm}\dessin{2.3cm}{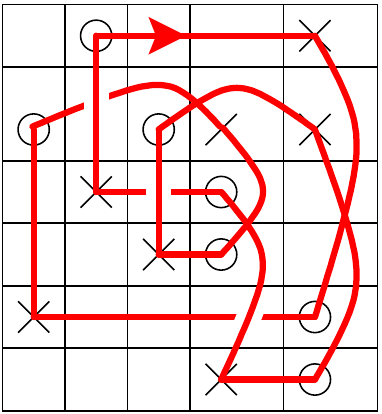} \hspace{.5cm} \leadsto \hspace{.5cm}\dessin{2.3cm}{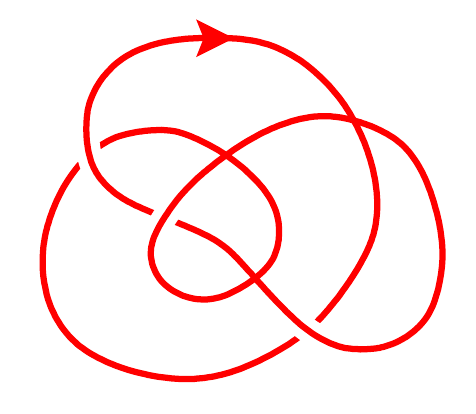}
  $$
  \caption*{From singular grid diagrams to singular links}
  \label{fig:SG->SK}
\end{figure}
\saut

As in the regular case, every singular grid diagram gives rise to an oriented singular link.
The process is almost identical.
First we join the decorations in regular columns.
For singular ones, we connect the uppermost decoration to the third one and the second to the lowermost by vertical lines slightly bended towards the right (or, equivalently, towards the left) in such a way that the two curves intersect in one singular double point.
Then we join again the decorations but inside the rows, taking care in underpassing vertical strands when necessary.
Pairs of singular horizontal strands are simultaneously bended upward or downward.
As a matter of fact, every singular RoC gives rise to a singular double point.

\begin{prop}
  \label{prop:SLink->SGrid}
  Every singular link can be described by a singular grid diagram.
\end{prop}
\begin{proof}
  Consider a planar diagram for a given singular link and choose a way to desingularize all the double points. 
  Now, consider a grid diagram which corresponds to this regular diagram.
  Singular points appear as regular crossings, \ie as four decorations arranged within a cross pattern.
  
  \saut
  
  {\parpic[r]{$\dessin{2cm}{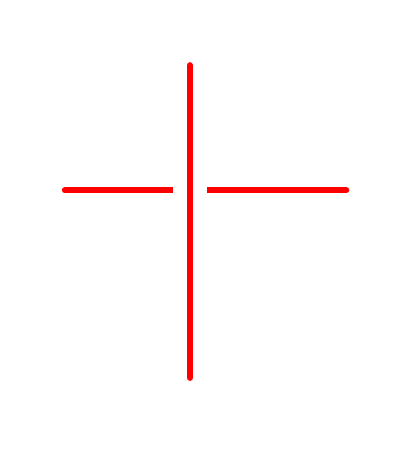}\leadsto\dessin{2cm}{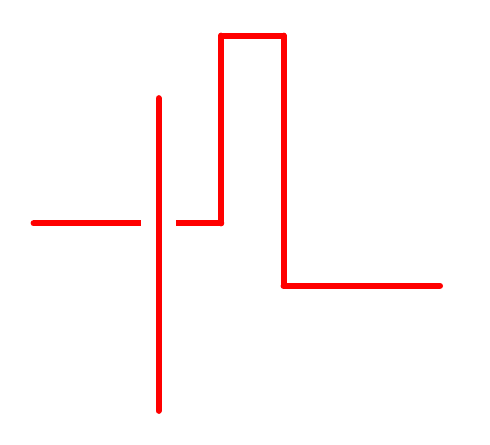}\leadsto\dessin{2cm}{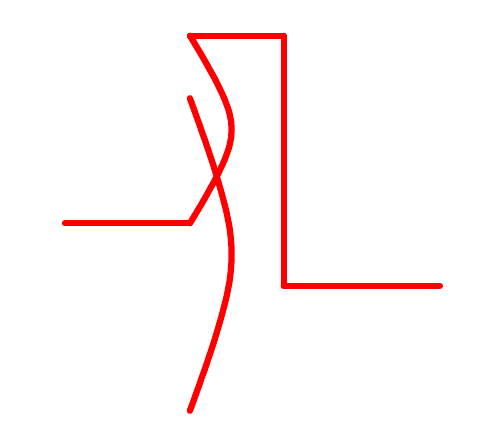}$}
    {If the vertical strand belongs to a regular column, then, by performing two stabilizations and a few commutations, one can obtain a configuration which enables the recovering of the double point by merging two adjacent columns.}
    
  }
  
  {\parpic[l]{$\dessin{2.15cm}{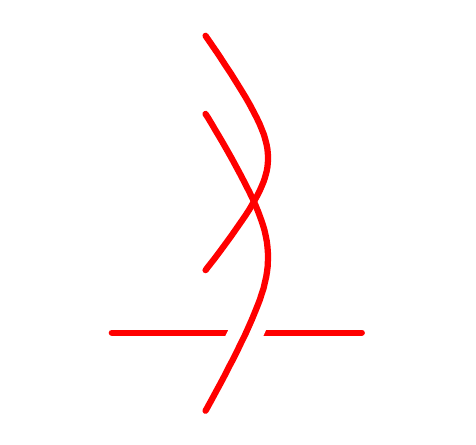}\leadsto\dessin{2.15cm}{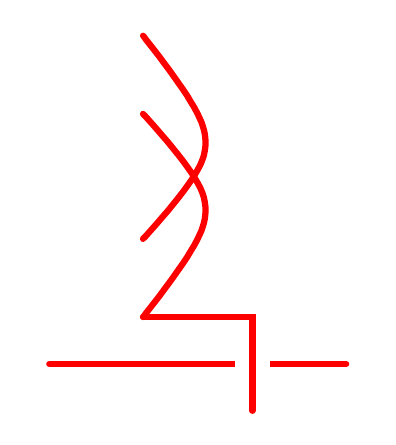}$}
    {\vspace{.5cm}If the vertical strand is already part of a singular column, then we can move the crossing in order to be back to the precedent case.}
    
  }
  
\end{proof}

Obviously, circular permutations and commutations of RoCs in a singular grid leave the associated singular link invariant.
As shown in Figure \ref{fig:Forbidden_Stabilization}, (de)stabilizations require a little more attention. 
\begin{figure}
  $$
  \xymatrix@C=1.5cm{\dessin{2.5cm}{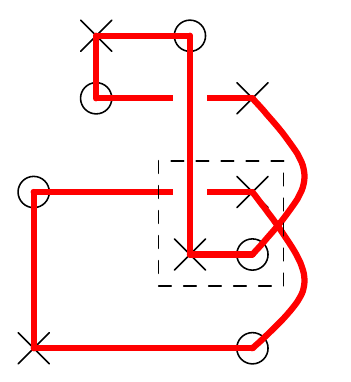}\ar@{<->}[r]|{\dessin{.5cm}{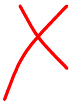}}&\dessin{2.5cm}{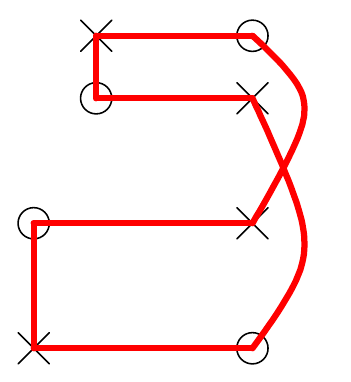}}
  $$ 
  \caption{Forbidden stabilization}
  \label{fig:Forbidden_Stabilization}
\end{figure}
In order to avoid forbidden phenomena, we require that the intersection of the $(2\times 2)$--square involved in a (de)stabilization with any given singular RoC contains at most one decoration.
In other words, the row and the column determined by adjacent decorations in the involved $(2\times 2)$--square are required to be regular.

\begin{figure}[b]
  $$
  \begin{array}{ccc}
    \xymatrix{\dessin{2.5cm}{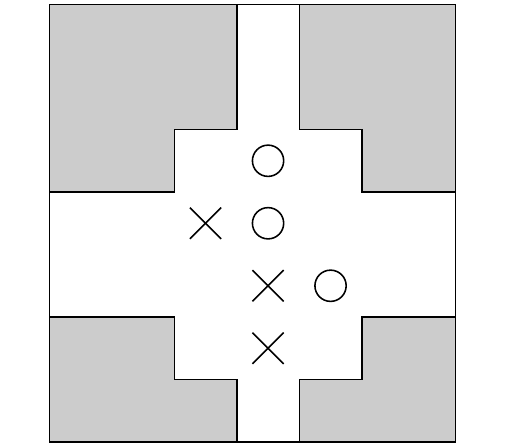} \ar@{<->}[r] & \dessin{2.5cm}{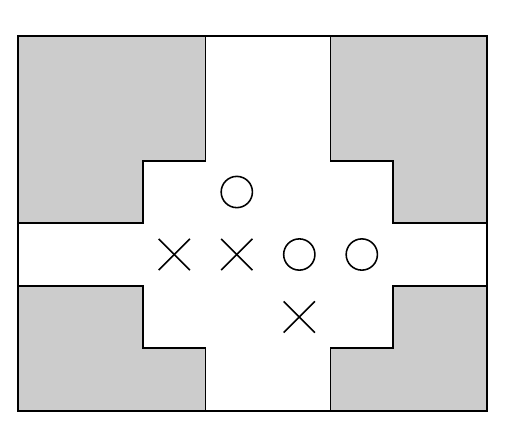}} & \hspace{.5cm} &  \xymatrix{\dessin{2.5cm}{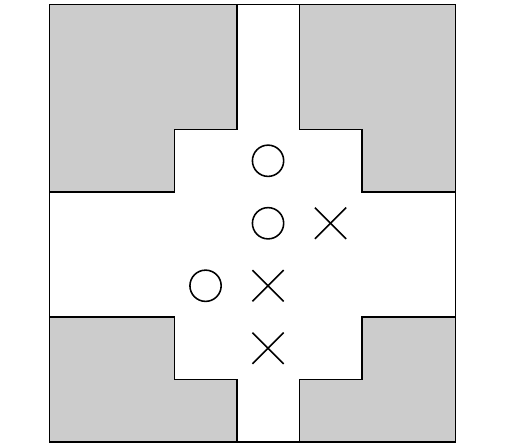} \ar@{<->}[r] & \dessin{2.5cm}{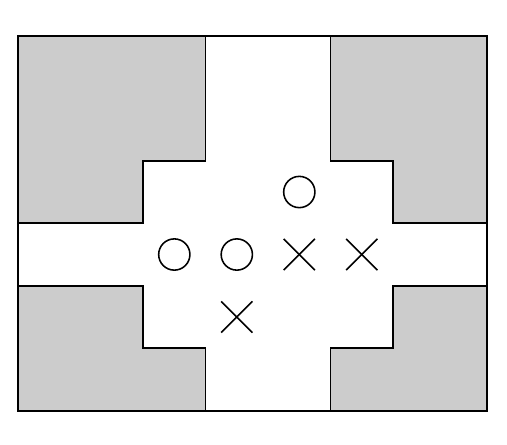}}\\
    \textrm{Left rotation} && \textrm{Right rotation}
  \end{array}
  $$
  \caption{Rotation moves}
  \label{fig:Rotation_Moves}
\end{figure}
\saut

Since a double point can be represented by mean of a singular column or by mean of a singular row, we introduce new grid moves, called \emph{rotations}, which relate these two possibilities.
Up to the addition and the removal of some empty pieces of rows or columns, rotations replace a specific $(3\times 4)$-subgrid by a $(4\times 3)$ one.
A picture of these specific subgrids is given in Figure \ref{fig:Rotation_Moves}.

\begin{theo}
  \label{theo:Elementary_Moves}
  Any two singular grid diagrams which describe the same singular link can be connected by a finite sequence of
  \begin{itemize}
  \item[-] regular cyclic permutations of RoCs, \ie cyclic permutations moving a regular RoC from one side to its opposite;
  \item[-] commutations of regular RoCs;
  \item[-] (de)stabilizations which involve at most one decoration in each singular RoC;
  \item[-] rotations.
  \end{itemize}
\end{theo}

These moves are called \emph{elementary moves}.
The first three ones are called \emph{regular elementary moves}.

\begin{figure}[b]
  $$
  \begin{array}{lc}
    \fbox{$\dessin{1.3cm}{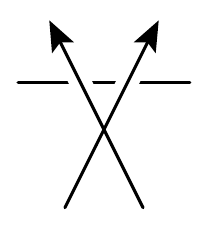}\sim\dessin{1.3cm}{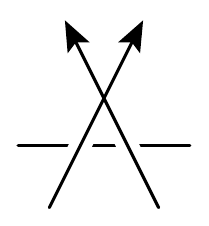}$}\ \ : \hspace{1cm} &
    \xymatrix{\dessin{2.45cm}{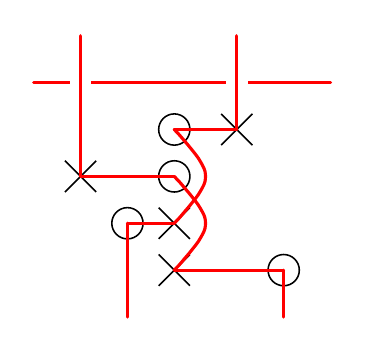} \ar@{<->}[r] & \dessin{2.45cm}{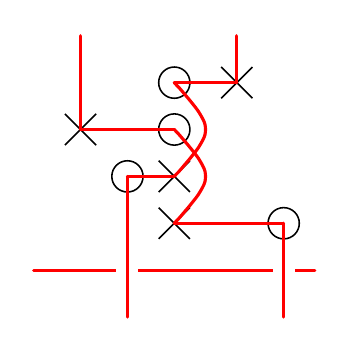}} \\[.4cm]
    \fbox{$\dessin{1.3cm}{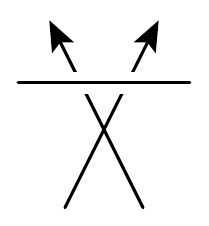}\sim\dessin{1.3cm}{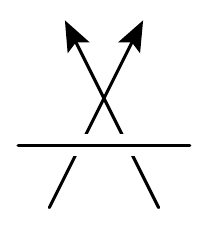}$}\ \ :  \hspace{1cm} &
    \xymatrix{\dessin{2.45cm}{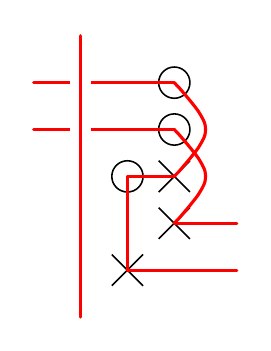} \ar@{<->}[r] & \dessin{2.45cm}{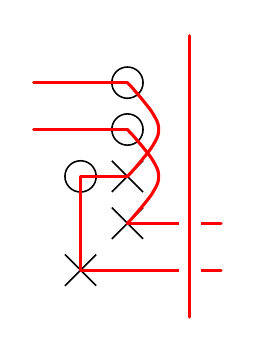}} \\[.4cm]
    \fbox{$\dessin{1.3cm}{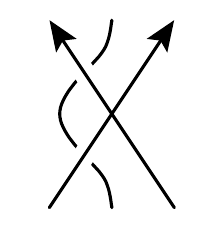}\sim\dessin{1.3cm}{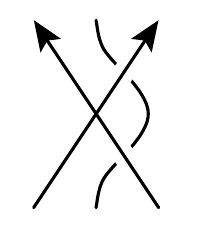}$}\ \ : \hspace{1cm} &
    \xymatrix{\dessin{2.45cm}{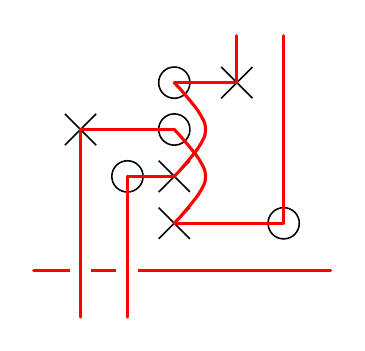} \ar@{<->}[r] & \dessin{2.45cm}{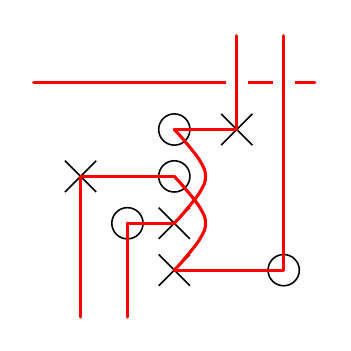}} \\[.4cm]
    \fbox{$\dessin{1.3cm}{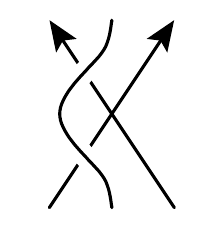}\sim\dessin{1.3cm}{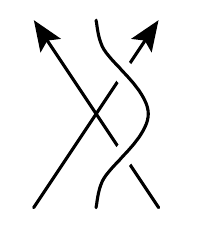}$}\ \ :  \hspace{1cm} &
    \xymatrix{\dessin{2.45cm}{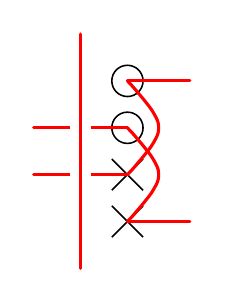} \ar@{<->}[r] & \dessin{2.45cm}{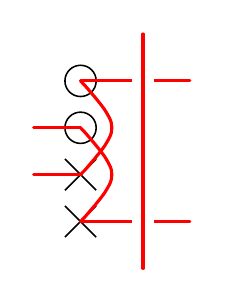}}
  \end{array}
  $$
  \caption{Realization of singular Reidemeister moves V: {\footnotesize Arrows stand for sequences of commutations.}}
  \label{fig:Reidemeister_V}
\end{figure}

\begin{figure}[p]
  $$
  \begin{array}{lc}
    \begin{array}{c}
      \fbox{$\dessin{1.3cm}{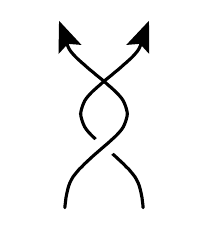}\sim\dessin{1.3cm}{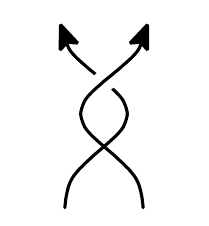}$}\ \ : \hspace{.2cm}\\[1.7cm]
      \\
    \end{array}
    &
    \vcenter{\hbox{\xymatrix@!0@C=1.4cm@R=.8cm{\dessin{2cm}{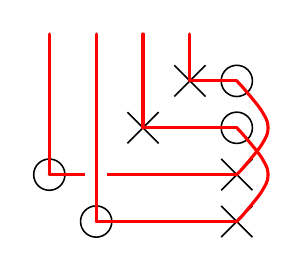} \ar@{<->}[dddrr] &&&&&& \dessin{2cm}{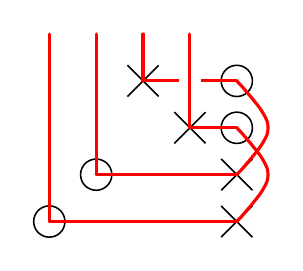}\\
          &&&&&&\\
          &&&&&&\\
          && \dessin{2cm}{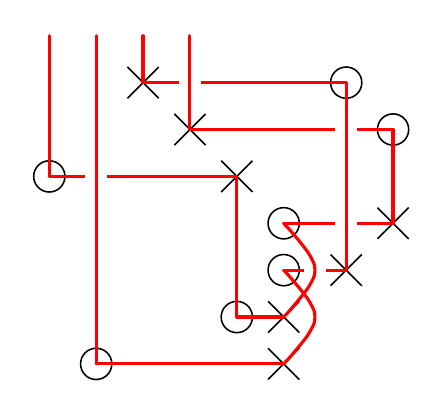} \ar@{<=>}[ddddll] && \dessin{2cm}{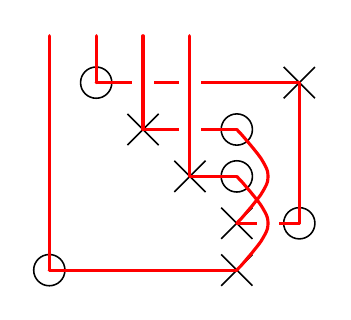} \ar@{<->}[uuurr] &&\\
          &&&&&&\\
          &&&&&&\\
          &&&&&&\\
          \dessin{2cm}{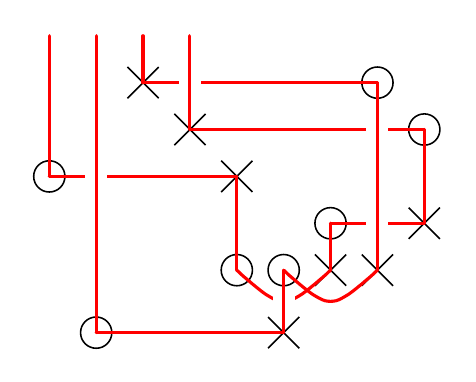} \ar@{<->}[ddrrr] &&&&&& \dessin{2cm}{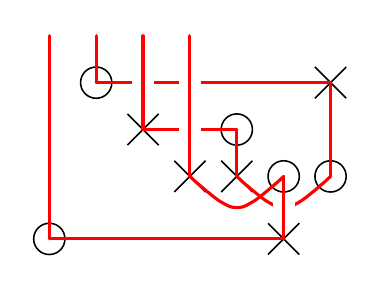} \ar@{<=>}[uuuull]\\
          &&&&&&\\
          &&& \dessin{1.5cm}{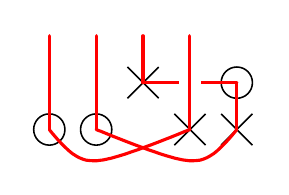} \ar@{<->}[uurrr] &&&}}}\\[5cm]
    
    \begin{array}{c}
      \fbox{$\dessin{1.3cm}{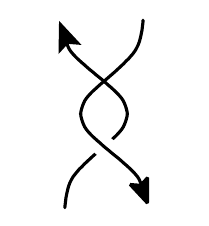}\sim\dessin{1.3cm}{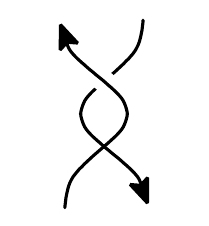}$}\ \ : \hspace{.2cm}\\[1.7cm]
      \\
    \end{array}
    &
    \vcenter{\hbox{\xymatrix@!0@C=1.4cm@R=.8cm{\dessin{2cm}{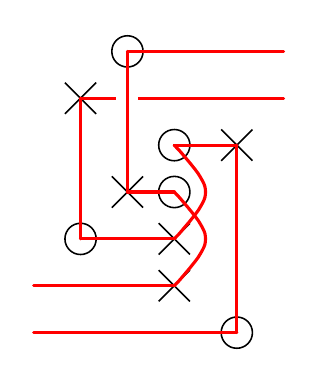} \ar@{<->}[dddrr] &&&&&& \dessin{2cm}{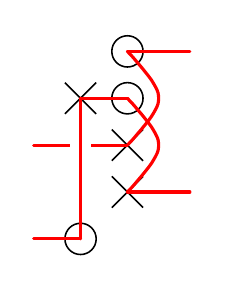}\\
          &&&&&&\\
          &&&&&&\\
          && \dessin{2cm}{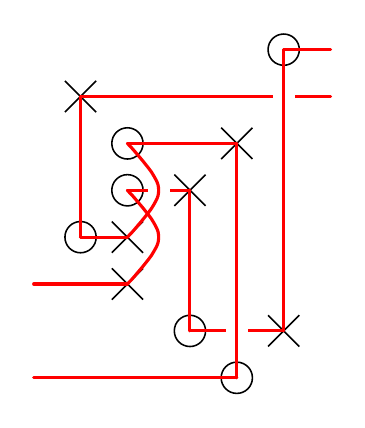} \ar@{<=>}[ddddll] && \dessin{2cm}{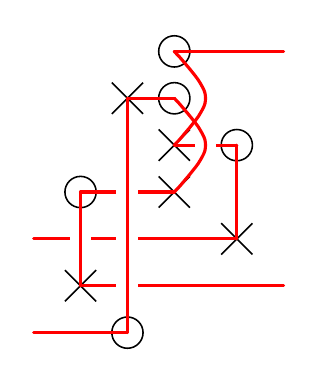} \ar@{<->}[uuurr] &&\\
          &&&&&&\\
          &&&&&&\\
          &&&&&&\\
          \dessin{2cm}{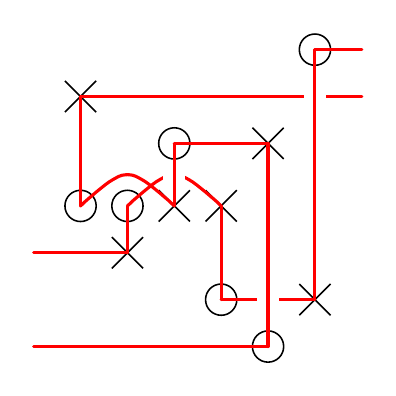} \ar@{<->}[ddrrr] &&&&&& \dessin{2cm}{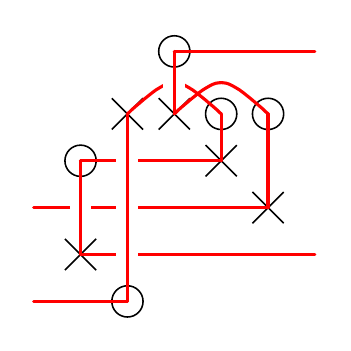} \ar@{<=>}[uuuull]\\
          &&&&&&\\
          &&& \dessin{2.2cm}{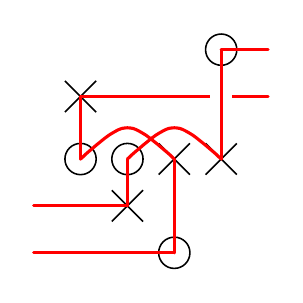} \ar@{<->}[uurrr] &&&}}}
  \end{array}
  $$
  \caption{Realization of singular Reidemeister moves IV: {\footnotesize Simple arrows stand for combinations of regular elementary moves. Double arrows stand for rotation moves. Mirror Reidemeister moves are obtained by reflecting vertically all the pictures and swapping the nature of the decorations.}}
  \label{fig:Reidemeister_IV}
\end{figure}

\begin{proof}
  Let $G_1$ and $G_2$ be two singular grid diagrams which describe the same link.
  
  \saut
  
  Essentially thanks to rotation moves, we can turn all singular rows in $G_1$ and in $G_2$ into singular columns.
  Furthermore, by performing a combination of regular elementary moves first, we can assume that the decorations in any singular column are in adjacent cases, the two $O$'s being above the two $X$'s, and that every decoration which shares a row with one of the two middle decorations of a singular column is located on the left of this column.
  
  We denote by $D_1$ and $D_2$ the diagrams associated to $G_1$ and $G_2$ by bending all the vertical singular strands towards the right.
  Double points are then in the following position:
  $$
  \dessin{2cm}{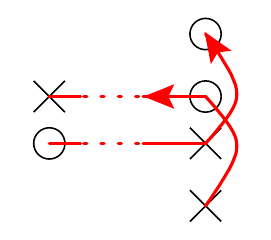}.
  $$
  
  First, we assume that $D_1$ and $D_2$ are isotopic as diagrams.
  Then we can choose an isotopy connecting $D_1$ to $D_2$ which rigidly preserves a neighbourhood of the crossings and of the double points, except for a finite number of times when a crossing is turned over or a double point fully rolled up around itself.
  When compared with the proof of Proposition 4 in \cite{Dynnikov}, only the latter case need some attention.
  However, it can also be easily realized as
  $$
  \begin{array}{lc}
    \begin{array}{c}
      \fbox{$\dessin{1.2cm}{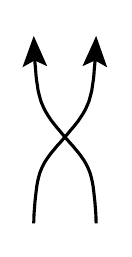}\sim\dessin{1.8cm}{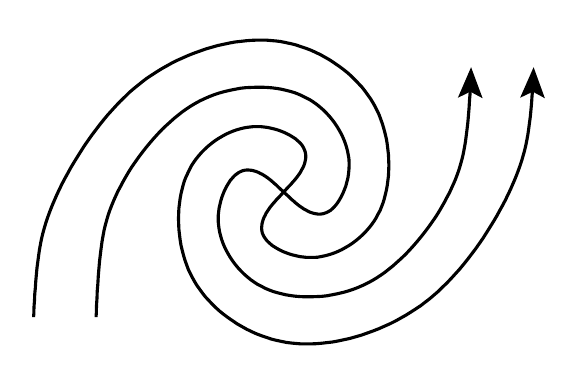}$}\ \ : \hspace{.2cm}
    \end{array}
    &
    \vcenter{\hbox{\xymatrix{\dessin{1.8cm}{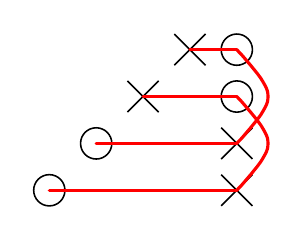} \ar@{<->}[r] & \dessin{2.6cm}{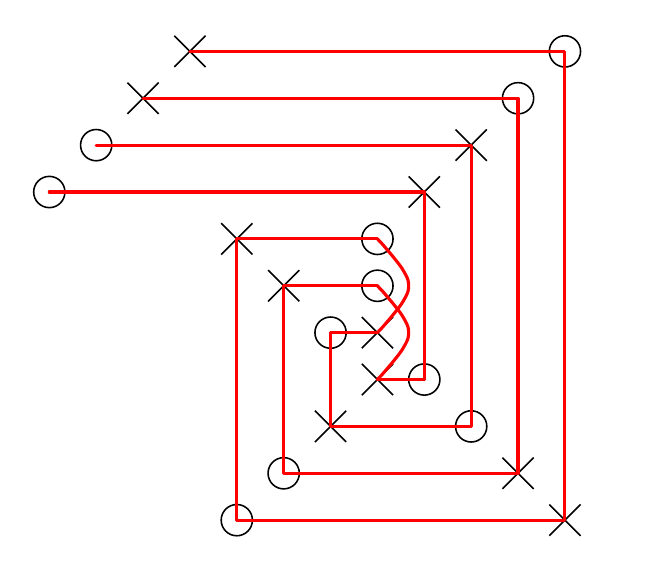}}}}
  \end{array}
  $$
  \noi using stabilizations and commutations.
  During the process, some cyclic permutations or commutations involving a singular column may occur.
  But then, with the help of rotations moves, they can be replaced by regular moves.
  The grids $G_1$ and $G_2$ can thus be connected by a sequence of rotations and regular elementary moves.
  
  \saut
  
  Now, it is sufficient to realize all the Reidemeister moves to complete the proof.
  For the regular ones, we refer to the proof of Proposition 4 in \cite{Dynnikov}.
  Because of the required rigidity condition on double points, the last two Reidemeister moves split into eight cases.
  Figures \ref{fig:Reidemeister_V} and \ref{fig:Reidemeister_IV} handle with all of them.
  
  \saut
  
  Conversely, every elementary move clearly preserves the underlying singular link.
\end{proof}

\begin{remarque}
  The description can be restricted to grids with no singular row. Then rotation moves must be replaced by another moves involving only singular columns. There are several equivalent ways to define them. Lemma $1.3$ in \cite{These} gives some of them.
\end{remarque}


%% file: Grids_Desingularization.tex
\subsection{RoC desingularization}
\label{ssec:Grids_Desingularization}

There are four ways to split a singular RoC into two regular ones.
However, as illustrated in Figure \ref{fig:RoC_Desingularizations}, only two of them preserve the connections between decorations induced by the associated link.
We are interested only in these resolutions.
\begin{figure}[b]
  \centering
  $
  \xymatrix@!0@R=1.2cm@C=2.5cm{\dessin{2cm}{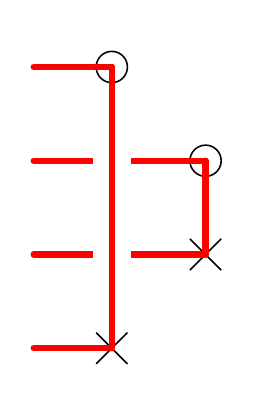}&& \dessin{2cm}{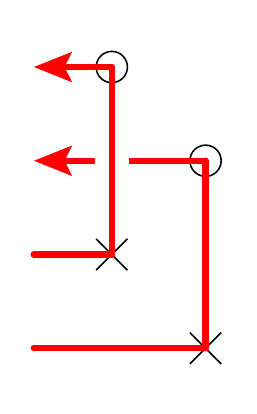}\\
    &\dessin{2cm}{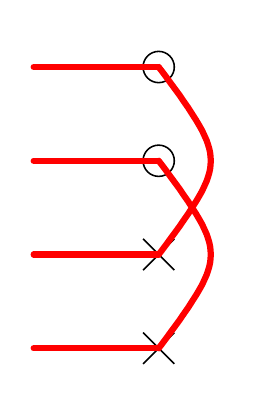} \ar[ur]|(.47){\textcolor{green}{\huge \checked}} \ar[ul]|(.47){\dessin{.5cm}{wrong}} \ar[dr]|(.47){\textcolor{green}{\huge \checked}} \ar[dl]|(.47){\dessin{.5cm}{wrong}} &\\
    \dessin{2cm}{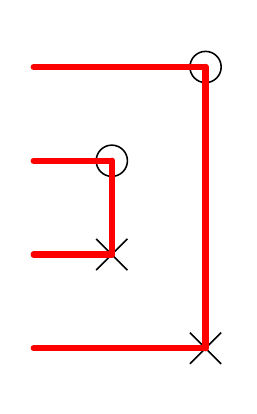}&& \dessin{2cm}{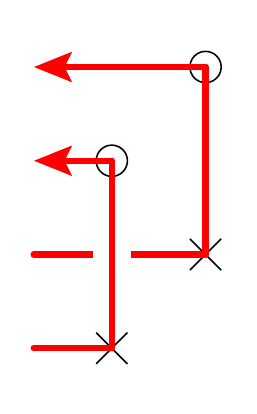}}
  $
  $
  \xymatrix@!0@R=1.2cm@C=2.3cm{\hspace{.1cm} 0\textrm{--resolution}\\ \\ \hspace{.1cm} 1\textrm{--resolution}}
  $
  $
  \xymatrix@!0@R=1.2cm@C=3.2cm{\dessinH{2cm}{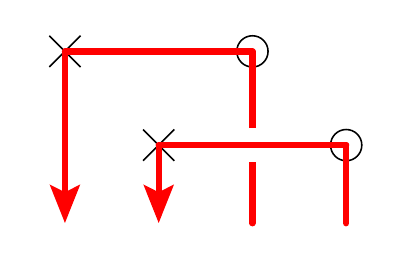} & \\
    &\dessinH{2cm}{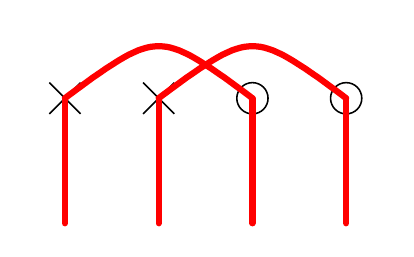} \ar[ul]|(.47){\textcolor{green}{\huge \checked}} \ar[dl]|(.47){\textcolor{green}{\huge \checked}}\\
    \dessinH{2cm}{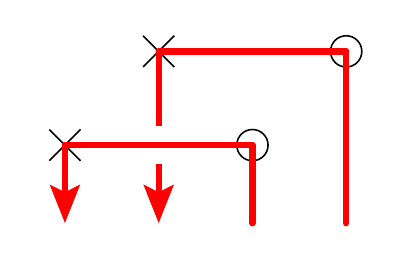} & }
  $
  \caption{Desingularizations of singular RoCs}
  \label{fig:RoC_Desingularizations}
\end{figure}
When dealing with a singular column, they can be distinguished by considering the slope of the line passing through its two topmost decorations:

-\hspace{.1cm}if the two decorations are of the same kind, then the \emph{$0$--resolution} corresponds to the negative slope.
We say the singular column is \emph{positively resolved} or \emph{$0$--resolved}.
The positive slope case is then called \emph{$1$--resolution} and the singular column is said to be \emph{negatively} or \emph{$1$--resolved};

-\hspace{.1cm}if the two decorations are of different kinds, then the $0$--resolution corresponds to the positive slope and the $1$--resolution to the negative one.

When dealing with a singular row, we give the same definitions but using the two rightmost decorations.\\

\begin{prop}
  \label{prop:Correspondance_GridResolution_LinkResolution}
  At the level of links, the $0$ and $1$--resolutions of a singular RoC correspond respectively to the positive and the negative resolutions of the double point associated to this RoC.
\end{prop}
\begin{proof}
  First, we deal with the desingularization of a singular column $c$.
  Up to cyclic permutation, the distinction between $0$ and $1$--resolutions is consistent.
  Actually, when a cyclic permutation of RoCs affects the topmost decoration of $c$, it changes its possible equalness of nature with the second topmost decoration of $c$, but it changes also the slope sign of the line passing through them.
  We can hence assume that $c$ is the rightmost column and that its two $\O$--decorations are above its two $\X$--ones.
  Then we are in the case illustrated in Figure \ref{fig:RoC_Desingularizations} and the proposition can be checked directly.

  Desingularization of singular rows can be treated in a similar way.
\end{proof}

Resolving a singular RoC can be seen as drawing an arc inside this singular RoC, which meanders between decorations and meets at most once other grid lines:
$$
\xymatrix@!0@R=1.5cm@C=3.7cm{& \dessinH{2.5cm}{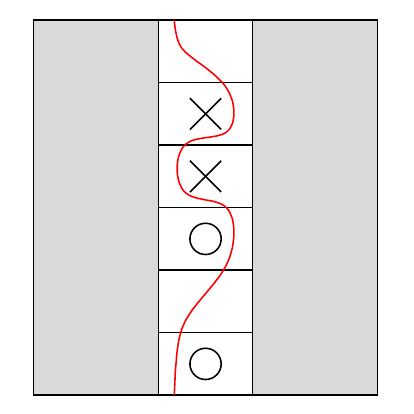} & **[l] 0\textrm{--resolution}\\
  \dessinH{2.5cm}{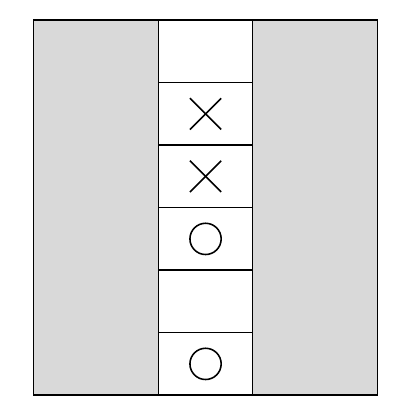} \ar[ur] \ar[dr] & & \\
  & \dessinH{2.5cm}{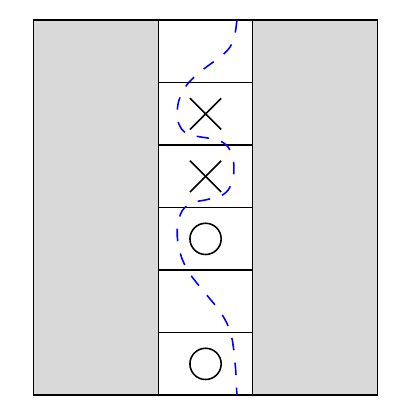} & **[l] 1\textrm{--resolution}}.
$$

\vspace{-.7cm}
{\parpic[r]{$\dessinH{2.8cm}{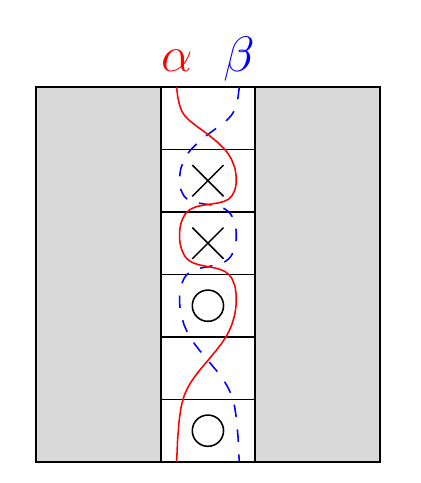}$}
  {\vspace{1cm}A set of \emph{winding arcs} is a set of two such arcs which meet each other in exactly four points disjoint from the grid lines.
    By convention, we denote by $\alpha$ the arc corresponding to the $0$--resolution and by $\beta$ the arc corresponding to the $1$--resolution.}
  
}
\vspace{.6cm}
If considering the torus obtained by gluing the opposite sides of the grid, the arcs $\alpha$ and $\beta$ bound four bigons and each of them contains exactly one decoration.
Every element of $\alpha\cap\beta$ is then at the intersection of two bigons which can be identified with the decorations they contain.
Furthermore, these two bigons can be ordered using the notion of being above or being on the right inherited from the grid.
It defines a \emph{type} for each element of $\alpha\cap\beta$.
In what follow, we will exclusively be interested in arcs intersections of types $\dessin{.4cm}{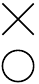}$, $\dessin{.4cm}{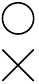}$, $\dessinH{.4cm}{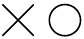}$ and $\dessinH{.4cm}{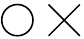}$.


%% file: Homology.tex
\section{Singular link Floer homology}
\label{sec:Homology}

\input{Homology_Definition}

\input{Homology_Statement}

\input{Homology_Consistency}


%% file: Homology_Definition.tex
\subsection{Definition}
\label{ssec:Homology_Definition}

The singular link Floer homology can be inductively defined as the mapping cone of a map associated to a given double point.
But in this paper, we will give a direct definition of the whole differential.
It is defined by counting grid polygons which generalize the regular rectangles.

\saut

Let $G$ be a singular grid of size $n$ with $k\in\N$ singular RoCs enhanced with winding arcs and $\T$ the torus obtained by identifying the opposite borders of $G$.
For convenience, we label the singular RoCs and all their associated objects with integers from $1$ to $k$ and the element of $\O$ with integers from $1$ to $n$.

For every $I=(i_1,\cdots,i_k)\in\{0,1\}^k$, we denote by $G_I$ the regular grid obtained by performing a $i_j$-resolution on the $j^\textrm{th}$ singular element for all $j\in \llbracket 1,k\rrbracket$.
At the level of bigraded $\Z[U_{O_1},\cdots,U_{O_{n}}]$--modules, we set
$$
CV^-(G)= \bigoplus_{I\in \{0,1\}^k}C^-(G_I)[-\#(I)],
$$
\noi where $\#(I)$ is the number of $1$ in $I$.
The value $\#(I)$ defines a third grading which is called \emph{desing grading}.

\begin{figure}[b]
 $$
  \begin{array}{ccccc}
    \dessin{3.7cm}{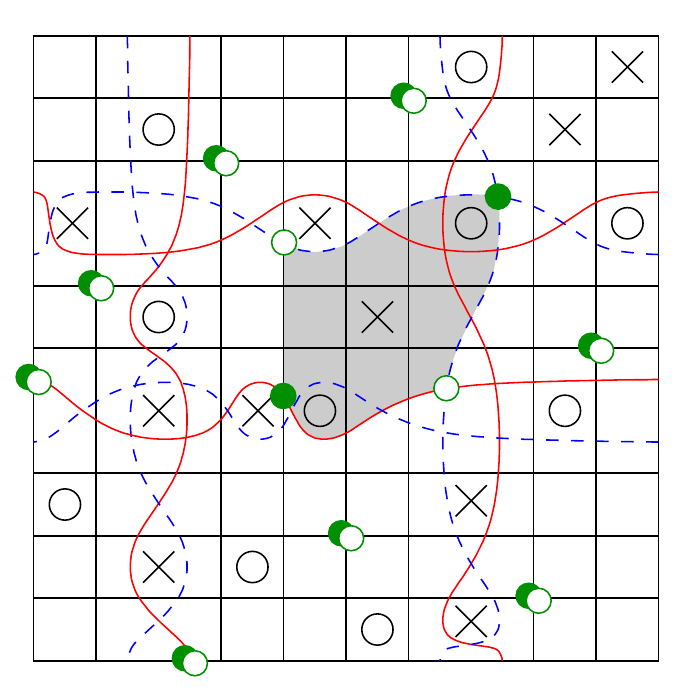} & \hspace{.8cm} & \dessin{3.7cm}{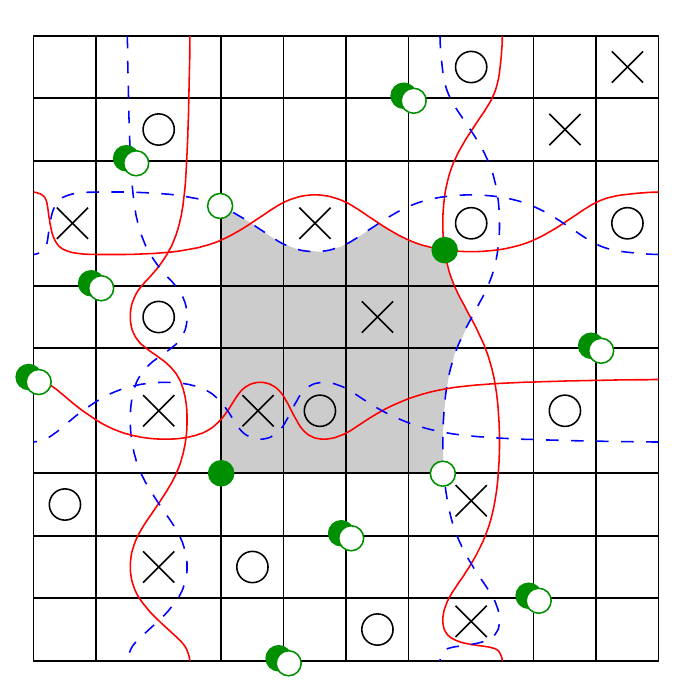} & \hspace{.8cm} & \dessin{3.7cm}{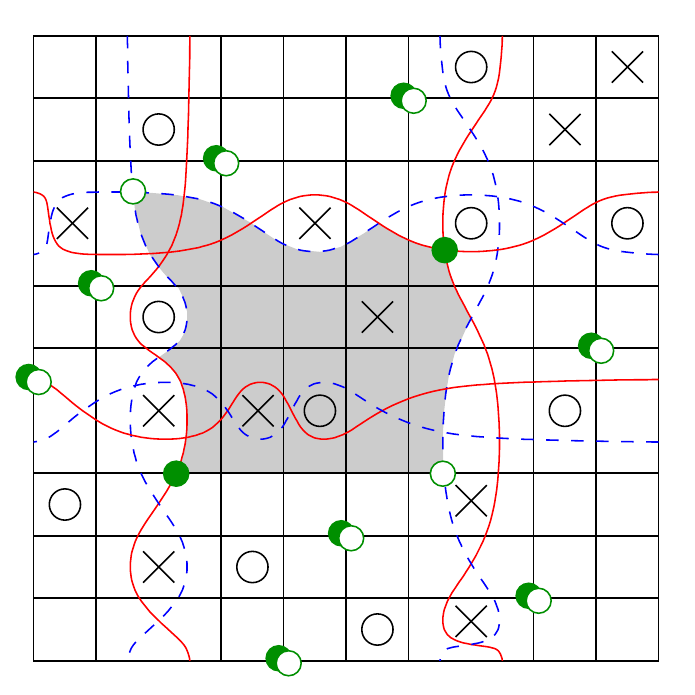}\\[1cm]
    \textrm{an empty rectangle} && \textrm{an empty hexagon} && \textrm{an empty heptagon}\\[1cm]
    \dessin{3.7cm}{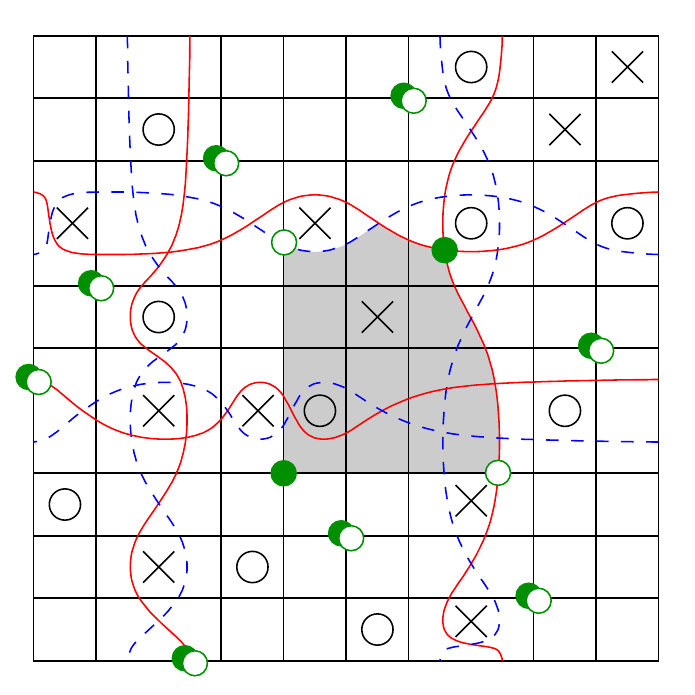} & \hspace{.8cm} & \dessin{3.7cm}{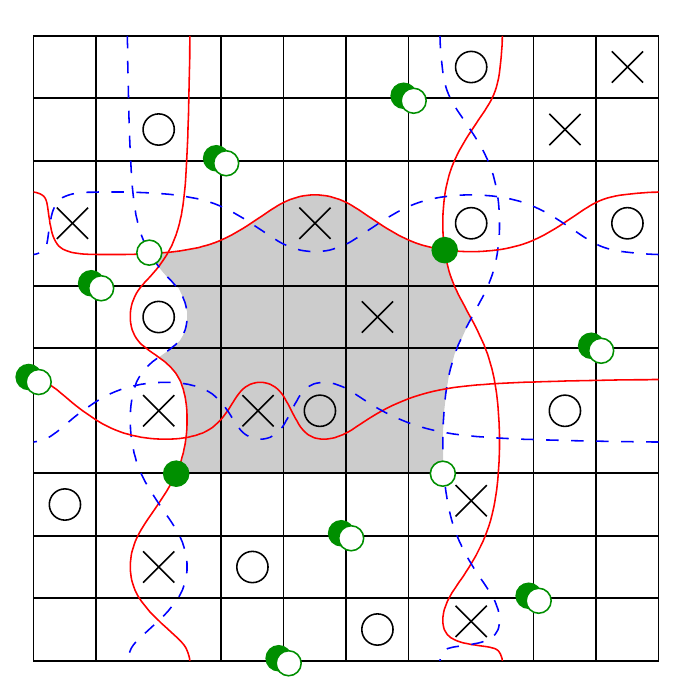} & \hspace{.8cm} & \dessin{3.7cm}{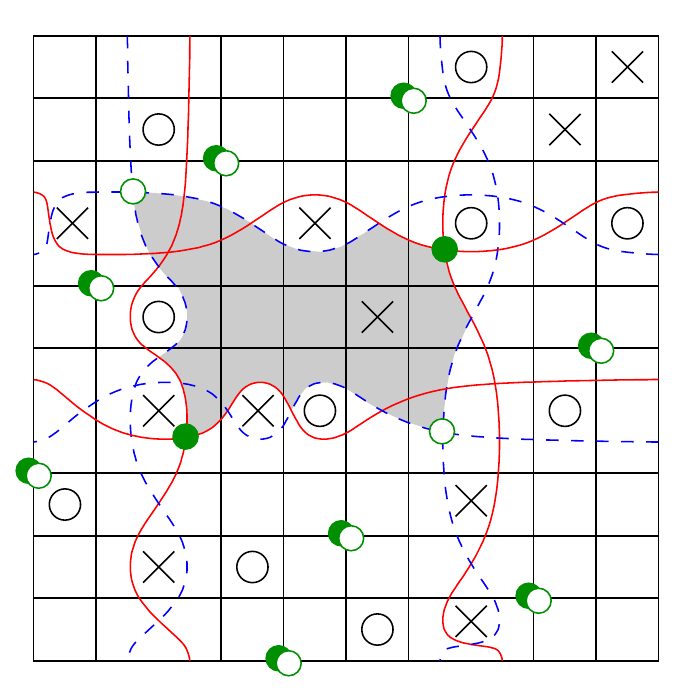}\\[1cm]
    \textrm{an empty pentagon} && \textrm{an empty hexagon} && \textrm{an empty octogon}
  \end{array}
$$ 
  \caption*{Examples of grid polygons connecting $x$ to $y$: {\footnotesize dark dots describe the generator $x$ while hollow ones describe $y$. Grid polygons are depicted by shading. With respect to the number of peaks, there are five kind of grid polygons with more than four corners.}}
    \label{fig:GridPolygons}
\end{figure}
\saut

Let $P$ be a set of peaks \ie a choice, for every $i\in\llbracket 1,k\rrbracket$, of an element in $\alpha_i\cap\beta_i$ of type $\dessin{.4cm}{Cstyle}$, $\dessin{.4cm}{Cpstyle}$, $\dessinH{.4cm}{rotCstyle}$ or $\dessinH{.4cm}{rotCpstyle}$.

\saut

Let $x$ and $y$ be generators of  $CV^-(G)$. A \emph{grid polygon connecting $x$ to $y$} is a polygon $\pi$ embedded in $\T$ which satisfies:
\begin{itemize}
\item[-] $\p \pi$ is embedded in the grid lines (including the winding arcs);
\item[-] $C(\pi)$, the set of corners of $\pi$, is a subset of $x\cup y \cup P$, with $C(\pi) \cap x \neq \emptyset$;
\item[-] starting at a element of $C(\pi)\cap x$ and running positively along $\p \pi$, according to the orientation of $\pi$ inherited from the one of $\T$, we first follow an horizontal arc and the corners of $\pi$ we meet are then successively and alternatively points of $y$ and $x$ with, possibly, an element of $P$ inserted between any two of them ;
\item[-] except on $\p \pi$, the sets $x$ and $y$ coincide;
\item[-] for any element $c\in C(\pi)$, $\Int (\pi)$ does not intersect the grid lines (including the winding arcs) in a neighbourhood of $c$.
\end{itemize}
The set $C(\pi)\cap P$ is called the \emph{set of peaks of $\pi$}.
The grid polygon is \emph{empty} if $\Int (\pi)\cap x=\emptyset$.
We denote by $\Pol (G,P)$ the set of all empty grid polygons on $G$ with at least four corners and by $\Pol (x,y,P)$ the set of those which connect $x$ to $y$.

\saut

Figure \ref{fig:Pol->Rect} illustrates how the peaks of a grid polygon can be embanked by adding or deleting some grid triangles and some bigons.
Then by moving, when necessary, dots from $\beta$ to $\alpha$--curves as shown in Figure \ref{fig:beta->alpha}, we define a map 
$$
\func{\phi}{\Pol (G,P)}{\Rect (G_{0\cdots 0})}.
$$

\begin{figure}[b]
  \subfigure[]{$\dessin{3.4cm}{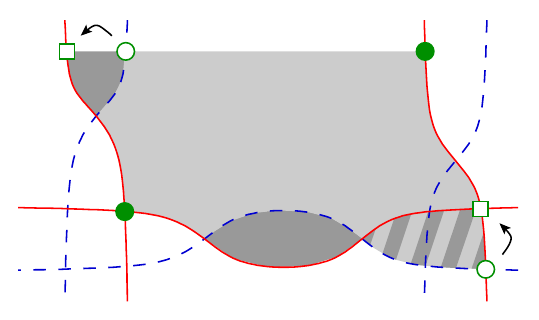}$ \label{fig:Pol->Rect}}
  \hspace{2cm}
  \subfigure[]{$\dessin{3.4cm}{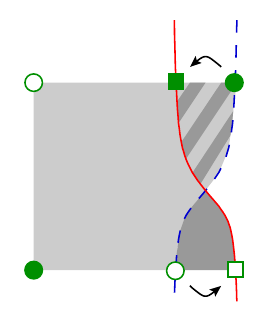}\hspace{1cm} \dessin{3.4cm}{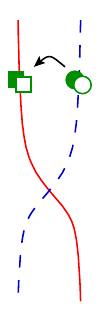}$ \label{fig:beta->alpha}}
  \caption{From grid polygons to rectangles: {\footnotesize dark dots describe the initial generator of the grid polygon while hollow ones describe the final one. The initial grid polygon is depicted by clear shading whereas the grid triangles and bigons are depicted by dark one.}}
\end{figure}

Conversely, a grid polygon with at least four corners can be obtained by adding peaks on a rectangle.
Depending on the relative position, in this rectangle, of the edge where a peak is added, we say the peak is of \emph{compass type North}, \emph{East}, \emph{South} or \emph{West}.

\saut

Now we state a lemma relating the Maslov and the Alexander degrees of generators connected by a grid polygon.

\begin{lemme}
  \label{lem:Grading_Respect}
  Let $x$ and $y$ be two generators of $CV^-(G)$ and $\pi$ a grid polygon in $\Pol (x,y,P)$. Then
  $$
  \begin{array}{l}
    M(x)-M(y) = 1 - 2\#(\pi\cap\O),\\[.2cm]
    A(x)-A(y) = \#(\pi\cap\X) - \#(\pi\cap\O).
  \end{array}
  $$
\end{lemme}
\begin{proof}
  It is sufficient to prove the statement for $M$ since replacing $\O$ by $\X$ and substracting the result to the $\O$--case prove it for $A$.
  We will do it by recurrence on $s$, the number of edges of $\pi$.

  The case $s=4$ corresponds to Lemma $2.5$ in \cite{MOST}.
  Now we suppose that the proposition is true for $s\geq 4$ and that $\pi$ has $s+1$ edges.
  Then $\pi$ has at least one peak $p$ and we suppose that it is of compass type East and type $\dessin{.4cm}{Cstyle}$.
  The other cases can be treated in a similar way or by using formulas for the behaviour of $M$ under rotation, reflection and swap of $\O$ and $\X$, formulas given in the proof of Proposition \ref{prop:Symetries}.
  
  According to Lemma $2.4$ in \cite{MOST}, the Maslov grading is invariant under cyclic permutation of the rows or of the columns.
  Hence, we can assume, without loss of generality, that, when read from bottom to the top, the decorations of the singular column where $p$ is located are ordered as follows: $XXOO$.
  Now, as in Figure \ref{fig:Pol->Rect}, we embank $p$ by adding or deleting a triangle $\tau_0$ and, possibly, some bigons $\tau_1,\dots, \tau_r$ in order to get a grid polygon $\pi'$ with $s$ edges connecting $x$ to a new generator $z$.
  Depending on the cardinality of $\O\cap(\bigcup_{i=0}^r\tau_i)$ and whether the triangle and the bigons are ripped or not, four cases have to be checked.
  In each case, it is straightforward to compute $M(z)-M(y)$, paying attention not to omit the shift in degree occuring in the definition of $CV^-(G)$, and to conclude using the recurrence assumption on $\pi'$.

  For instance, if the cardinality is one, then the triangle or one of the bigons is ripped and we have $M(z)-M(y)=2$.
  But, since $\pi'$ has $s$ edges, $M(x)-M(z)=1-2\#(\pi'\cap\O)=1-2\#(\pi\cap\O)-2$ and $M(x)-M(y)=1-2\#(\pi\cap\O)$.
  $$
  \dessin{3.2cm}{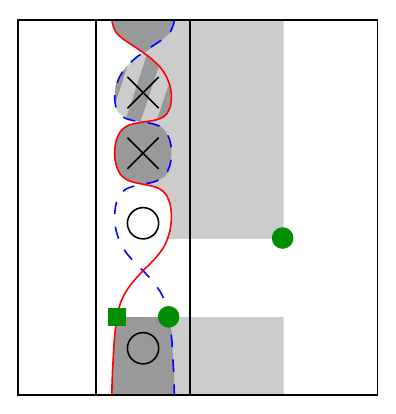} \hspace{1cm} \leadsto \hspace{1cm} y:\ \dessin{3.2cm}{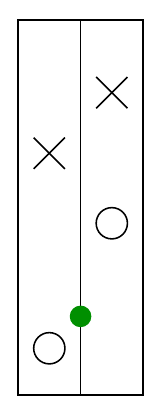} \hspace{1cm} z:\ \dessin{3.2cm}{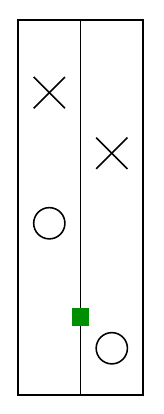}
  $$
\end{proof}

\saut

Now, we choose a \emph{compass convention} $\func{C}{P}{\{\textrm{North},\textrm{South},\textrm{East},\textrm{West}\}}$ with $C(p)\in \{\textrm{North},\textrm{South}\}$ if $p$ is an intersection of winding arcs belonging to a singular row and $C(p)\in \{\textrm{East},\textrm{West}\}$ otherwise.
Then we define a sign map
$$
\function{\e}{\Pol (G,P)}{\pi}{\{\pm 1\}}{(-1)^{\mu(\pi,C)}\e\big(\phi(\pi)\big)}
$$
\noi where $\mu(\pi,C)$ counts the number of peaks $p$ of $\pi$ such that the compass type of $p$ is $C(p)$.

Finally, we define the map $\func{\p^-_{G,P,C}}{CV^-(G)}{CV^-(G)}$, or $\p^-_G$ for short, as the morphism of $\Z[U_{O_1},\cdots,U_{O_{n}}]$--modules defined on the generators by
$$
\p^-_G(x) = \sum_{\substack{y \textrm{ generator}\textcolor{white}{R}\\ \textrm{of }CV^-(G)}\textcolor{white}{I}} \sum_{\pi \in \Pol (x,y,P)} \e(\pi)U_{O_1}^{O_1(\pi)}\cdots U_{O_{n}}^{O_{n}(\pi)}\cdot y.
$$

\begin{prop}
  \label{prop:Chain_Complex}
  The couple $\big(CV^-(G),\p^-_G\big)$ is a filtered chain complex.
\end{prop}
\begin{proof}
  The map $\p^-_G\circ\p^-_G$ counts juxtapositions of two grid polygons.
  We will prove that they can be gathered in canceling pairs.

  Let $\pi_1\in \Pol(x,y,P)$ and $\pi_2\in \Pol(y,z,P)$ for $x$, $y$ and $z$ three generators of $CV^-(G)$.
  If $\p \pi_1 \cap \p \pi_2$ contains a peak of $\pi_1$ or $\pi_2$, then, up to rotation, it locally looks like one of the two pictures below, where grid polygons are partially represented by shading with different intensities and where the square is part of the intermediary generator $y$.
  $$
  \dessin{3.5cm}{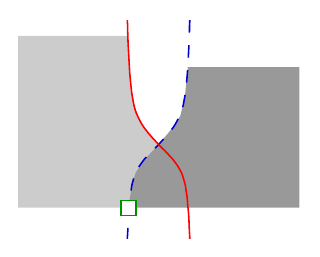} \hspace{1.5cm} \dessin{3.5cm}{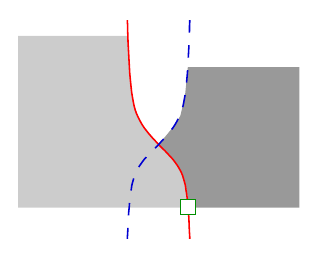}
  $$
  The other picture provides then an alternative decomposition into grid polygons $\pi'_1\in \Pol(x,y',P)$ and $\pi'_2\in \Pol(y',z,P)$, where $y'$ is a fourth generator of $CV^-(G)$, which involves a peak with opposite compass type.
  Moreover, $\phi(\pi_1)=\phi(\pi'_1)$ and $\phi(\pi_2)=\phi(\pi'_2)$, so the contributions cancel each other out.

  If $\p \pi_1 \cap \p \pi_2$ does not contain such a peak, then $\pi_1$, $\pi_2$ and their peaks induce a right inverse for $\phi$.
  The problem can be translated to the regular grid $G_{0\cdots 0}$ and the proofs of Propositions $2.8$ and $4.20$ in \cite{MOST} provide an alternative and canceling decomposition.

  \saut
  
  The fact that $\p^-_G$ preserves the Alexander filtration and decreases the Maslov grading by one is a consequence of Lemma \ref{lem:Grading_Respect}.
\end{proof}

\begin{cor_proof}
  \label{cor:Mapping_Cone}
  Let $r$ be a singular RoC in $G$ and $p$ the element of $P$ which belongs to $r$.
  Let $G_+$ and $G_-$ be the two grids obtained by desingularizing $r$.

  Then the $\Z[U_{O_1},\cdots,U_{O_{n}}]$-linear map $\func{f_p}{CV^-(G_+)}{CV^-(G_-)}$ defined on the generators by
  $$
  f_p(x) = \sum_{\substack{y \textrm{ generator}\textcolor{white}{R}\\ \textrm{of }CV^-(G_-)}\textcolor{white}{I}} \sum_{\pi \in \Pol (x,y,P,p)} \e(\pi)U_{O_1}^{O_1(\pi)}\cdots U_{O_{n}}^{O_{n}(\pi)}\cdot y
  $$
  \noi where $\Pol (x,y,P,p)$ is the set of grid polygons in $\Pol (x,y,P)$ which contain $p$ as a peak, is a chain map which anticommutes with the differentials defined for the set of peaks $P'=P\setminus \{p\}$ and the compass convention $C'=C_{|P'}$.

  The chain complex $\big(CV^-(G),\p^-_G\big)$ is the mapping cone of $f_p$.
\end{cor_proof}


%% file: Homology_Statement.tex
\subsection{Statement}
\label{ssec:Homology_Statement}

Now, we deal with $L$ a singular link enhanced with an orientation for every double point \ie an orientation for each plane spanned by two transverse vectors which are tangent to $L$ at a double point.
Let $G$ be a singular grid for $L$ with an arbitrary set of winding arcs.

We define a set of peaks $P$ as follows.
Let $D$ be the link diagram obtained from $G$ by bending upward the strands in singular rows and towards the right in singular columns.
Then $P$ is defined by choosing, for a singular row (resp. singular column), the arcs intersection of type $\dessin{.4cm}{Cstyle}$ (resp. $\dessinH{.4cm}{rotCstyle}$) if the orientation for the associated double point coincides with the orientation inherited from the plane on which $D$ is drawn.
Otherwise, we choose the arcs intersection of type $\dessin{.4cm}{Cpstyle}$ (resp. $\dessinH{.4cm}{rotCpstyle}$).

Finally, we choose an arbitrary compass convention $C$.

\begin{theo}
  \label{theo:Invariance_Statement}
  The homology $H_*(CV^-(G),\p^-_{G,P,C})$, denoted by $HV^-$, depends only on the oriented singular link $L$ and on the orientation of its double points.
\end{theo}


%% file: Homology_Consistency.tex
\subsection{Consistency}
\label{ssec:Homology_Consistency}

Theorem \ref{theo:Invariance_Statement} can be divided in six invariance propositions:
\begin{itemize}
\item[i)] under compass convention choice;
\item[ii)] under isotopy of winding arcs;
\item[iii)] under regular cyclic permutation of RoCs;
\item[iv)] under commutation of regular RoCs;
\item[v)] under (de)stabilization;
\item[vi)] under rotation.
\end{itemize}

\subsubsection{Compass convention}
\label{sssec:Compass_Convention}

The first point follows from Corollary \ref{cor:Mapping_Cone}.
Actually, changing the value of the compass convention for a given peak $p\in P$ corresponds to changing $f_p$ for $-f_p$.

\subsubsection{Isotopy of winding arcs}
\label{sssec:Winding_Arcs_Isotopy}

An isotopy of winding arcs modifies the differential each time an element of $P$ crosses a grid line.
So let  $(\alpha,\beta)$ and  $(\alpha',\beta')$ be two sets of winding arcs which are identical except near an horizontal grid line, denoted by $l$.
We assume that the arcs $\alpha$ and $\beta$ intersect in $p\in P$ on a side of $l$, whereas $\alpha'$ and $\beta'$ intersect in $p'$ on the other side. 
$$
\dessin{2.5cm}{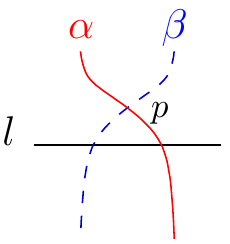}  \hspace{2cm}  \dessin{2.5cm}{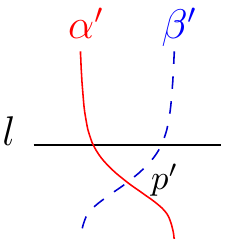}
$$
The case of an element of $P$ crossing a vertical grid line can be treated in a similar way.
With its terminology and applied to $p$ and $p'$, Corollary \ref{cor:Mapping_Cone} claims that the homologies associated to each set of winding arcs can both be described as the mapping cones of maps
$$
\func{f_p,f_{p'}}{CV^-(G_+)}{CV^-(G_-)}.
$$
We will prove that $f_p$ and $f_{p'}$ are homotopic.

For that purpose, we say that a generator $x$ of $CV^-(G_+)$ is \emph{$l$--linked} to a generator $y$ of $CV^-(G_-)$ if and only if $x\setminus(\alpha\cap l)=y\setminus(\beta\cap l)$.
Then, for all generators $x$ of $CV^-(G)$, we set
$$
h(x)=\left\{\begin{array}{cl}
    y & \textrm{if }x\textrm{ and }y\textrm{ are }l\textrm{--linked}\\
    0 & \textrm{otherwise}
  \end{array}\right.
$$
\noi and we extend $h$ to $CV^-(G_+)$ by $\Z[U_{O_1},\cdots,U_{O_n}]$--linearity.

\saut

\begin{lemme}
  \label{lem:Winding_Arcs_Isotopy}
  The map $h$ is an homotopy map between $f_p$ and $f_{p'}$ \ie it satisfies
  $$
  f_p - f_{p'}=h\circ\p^-_{G_+} - \p^-_{G_-}\circ h
  $$
  \noi where $\p^-_{G_+}$ and $\p^-_{G_-}$ are defined using compass conventions which coincide on common peaks and send $p$ and $p'$ to $\textrm{West}$.
\end{lemme}
\begin{proof}
  The surviving terms in $f_p-f_{p'}$ correspond to grid polygons with $p$ or $p'$ and $\alpha\cap l$ or $\beta\cap l$ as corners.
  They are terms of $h\circ \p^-_{G^+}$ when $\beta\cap l$ is a corner of the grid polygon and terms of $\p^-_{G^-}\circ h$ when $\alpha\cap l$ is.
  There is no difficulty in checking that the signs coincide.

  The remaining terms in $h\circ \p^-_{G_+}$ cancel with terms in  $\p^-_{G_-}\circ h$.
\end{proof}

\subsubsection{Cyclic permutation}
\label{sssec:Cyclic_Permutation}

Since grid polygons are embedded in the torus, they are clearly preserved by cyclic permutations.
Hence, we only need to pay attention to signs.

It is proven in \cite{MOST} and \cite{Gallais} that the sign assignment for rectangles is essentially unique in the sense that any two sign assignments define isomorphic chain complexes.
Since a cyclic permutation does not modify the compass type of grid polygon peaks and since a sign assignment for grid polygons depends only on its definition for rectangles and on the compass types, the invariance still holds in the singular case.

\subsubsection{Regular commutation}
\label{sssec:Commutation}

In this section, we consider a commutation of regular columns.
Commutation of regular rows can be treated similarly.

\vspace{-\baselineskip}

{\parpic[r]{$\dessin{3.5cm}{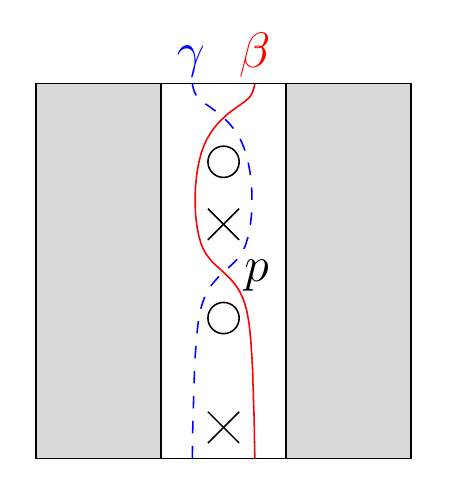}$}
  {\vspace{.6cm}Since the decorations of one of the two commuting columns are strictly above the decorations of the other one, the elementary move can be seen as replacing a distinguished vertical grid line $\beta$ by a different one $\gamma$, like in the picture on the right.

    \saut
    
    We denote by $G_\beta$ and $G_\gamma$ the corresponding grids.}
  
  \saut
  
}

By $P'$ we denote a common set of peaks considered for defining $\p^-_{G_\beta}$ and $\p^-_{G_\gamma}$.
Then, by considering the definition of grid polygons for a set of peaks $P=P'\cup(\beta\cap\gamma)$ and by including $\beta$ and $\gamma$ among the grid lines,  we define \emph{commuting polygons} connecting a generator $x$ of $CV^-(G_\beta)$ to a generator $y$ of $CV^-(G_\gamma)$.
One can observe that only one of the two elements of $\beta\cap\gamma$ can actually be realized as a peak.
This intersection is denoted by $p$ in the picture above.

For any pair of generators $x$ and $y$ of, respectively, $CV^-(G_\beta)$ and $CV^-(G_\gamma)$, we denote by $\PolC(x,y)$ the set of empty commuting polygons, with at least five corners, connecting $x$ to $y$.
Then, we can set the map $\func{\phi_{\beta\gamma}}{CV^-(G_\beta)}{CV^-(G_\gamma)}$ as the morphism of $\Z[U_{O_1},\cdots,U_{O_n}]$--modules defined on generators by
$$
\phi_{\beta\gamma}(x)=\sum_{\substack{y \textrm{ generator}\textcolor{white}{R}\\\textrm{of }CV^-(G_\gamma)\textcolor{white}{I}}} \sum_{\pi\in\PolC(x,y) } \e(\pi)U_{O_1}^{O_1(\pi)}\cdots U_{O_n}^{O_n(\pi)}\cdot y.
$$

\begin{lemme}
  \label{lem:Commutation_Isomorphism}
  The map $\phi_{\beta\gamma}$ is a chain quasi-isomorphism.
\end{lemme}
\begin{proof}
  The proof that $\phi_{\beta\gamma}$ anticommutes with the differentials is totally similar to the proof of Proposition \ref{prop:Chain_Complex}.

  Now we consider the filtration induced by the desing grading.
  The differentials and the map $\phi_{\beta\gamma}$ clearly preserve it.
  The associated graded chain complexes are the direct sums of the chain complexes associated to every desingularized grid, and, restricted to any of them, the graded map associated to $\phi_{\beta\gamma}$ is the eponyme morphism defined in Section $3.1$ of \cite{MOST}.
  In this paper, the authors prove that it is a quasi-isomorphism.
  It follows from standard homological algebra that $\phi_{\beta\gamma}$ is a quasi-isomorphism.
\end{proof}

\subsubsection{(De)Stabilization}
\label{sssec:Stabilization}

A stabilization replaces a $(1\times 1)$--subgrid $g_1$ by a $(2\times 2)$--subgrid $g_2$.  
Depending on the nature of the decoration in $g_1$ and on the square which is empty in $g_2$, there are eight different kinds of (de)stabilization. 
We prove invariance for the following one:

\begin{eqnarray}
  \label{thedot}
  \begin{array}{ccc}
    \dessin{1.7cm}{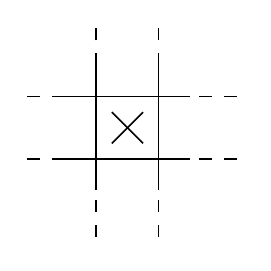} & \longrightarrow & \dessin{1.7cm}{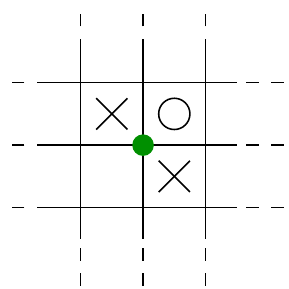},\\[1cm]
    G && G^s\\
  \end{array}
\end{eqnarray}

\noi but, {\it mutatis mutandis}, the proof can be adapted to the seven others cases.

We denote by $O_1$ the new $\O$--decoration and by $O_2$ the one which is lying on the same row than $g_1$.
According to the definition of (de)stabilizations, $O_1$ cannot belong to a singular RoC.

\saut

The broad outlines of the proof are:
\begin{description}
\item[i) Description of $CV^-(G)$ using $G^s$]
  Every generator of $CV^-(G)$ can be seen as drawn on $G^s$ by adding $x_0$, the dot located at the south-west corner of $O_1$ (see the dot in (\ref{thedot})).
  The gradings are the same and the differential is given by ignoring the conditions involving $O_1$ or $x_0$ \ie grid polygons may contain $x_0$ in their interior and there is no multiplication by $U_{O_1}$.
\item[ii) Description of $CV^-(G)$ involving $U_{O_1}$]
  The chain map $CV^-(G)$ is quasi-isomorphic to the mapping cone $C=C_1\oplus C_2[-1]$ of the map
  $$
  \xymatrix@C=2cm{C_1:= \left(CV^-(G)\otimes\Z[U_{O_1}]\right)\{-1\}[-1]\ar[r]^(.53){\times (U_{O_2}-U_{O_1})} & (CV^-(G)\otimes\Z[U_{O_1}])[1]=: C_2}.
  $$
  Hence, it is sufficient to define a quasi-isomorphism from $CV^-(G^s)$ to $C$.
\item[iii) Crushing filtration]
  There is a filtration on $CV^-(G^s)$ such that the associated graded differential is the sum over thin rectangles which are contained in the row or in the column through $O_1$ and which do not contain $O_1$ or any $X$.
\item[iv) Graded quasi-isomorphism]
  The associated graded chain complex splits into the three following subcomplexes\footnote{though intuitive, details about the notation are given in the Appendix A of \cite{These}}:
  $$
  \xymatrix@C=2cm{
    \left\{\dessin{1.5cm}{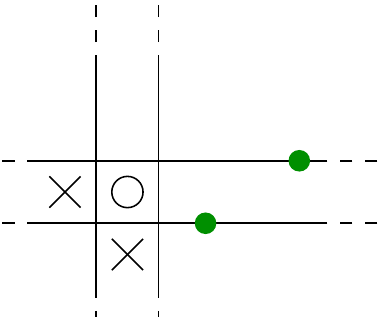}\right\} \ar[r]^{\dessin{.9cm}{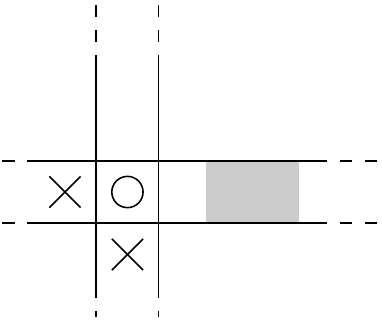}}_\sim \ar@{}|{}="Nya"  \ar@(dl,ul)"Nya"!<-1.4cm,-.2cm>;"Nya"!<-1.4cm,.4cm>^{\dessin{1.1cm}{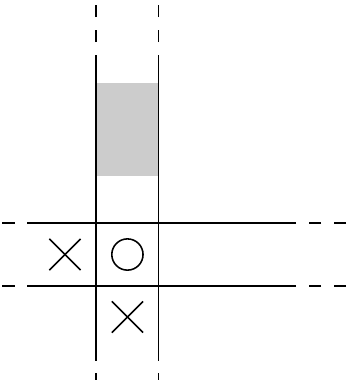}} & \left\{\dessin{1.5cm}{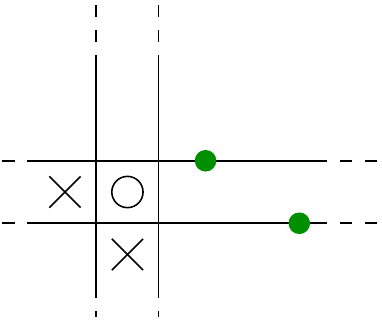}\right\}\ar@{}|{}="Nya2"  \ar@(dr,ur)"Nya2"!<1.3cm,-.2cm>;"Nya2"!<1.3cm,.4cm>_{\dessin{1.1cm}{DiffVert}}},
  $$
  $$
  \xymatrix{\left\{\dessin{1.5cm}{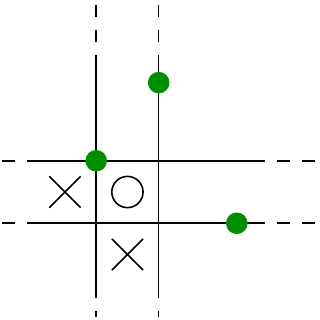}\ -\ \dessin{1.5cm}{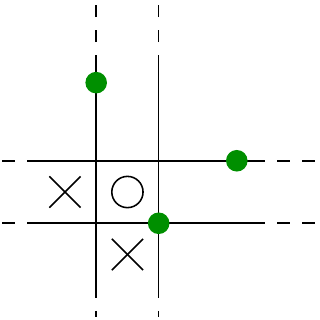}\right\} \ar@{}|{}="Nya" \ar@(dl,ul)"Nya"!<-2.3cm,-.2cm>;"Nya"!<-2.3cm,.4cm>^0}
  \hspace{1cm} \textrm{and} \hspace{.8cm}
  \xymatrix{\left\{\dessin{1.5cm}{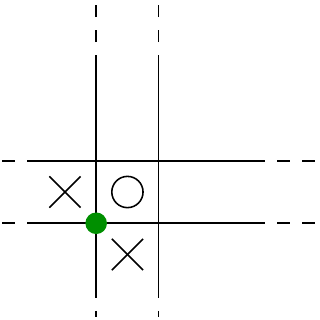}\right\}. \ar@{}|{}="Nya" \ar@(dl,ul)"Nya"!<-1.3cm,-0.2cm>;"Nya"!<-1.3cm,.4cm>^0}
  $$
  Then we can define a graded chain map $\func{F_{gr}}{CV^-(G^s)}{C}$ as
  $$
  F_{gr}:=\left\{\begin{array}{l}
      \xymatrix@C=2cm@R=0.5cm{\dessin{1.3cm}{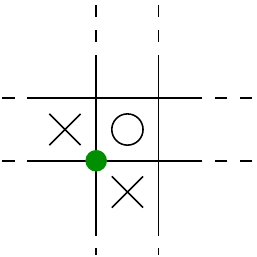} \ar@{|->}[r]^(.45){Id}& \dessin{1.3cm}{Gen1}\in C_1\\
        \dessin{1.3cm}{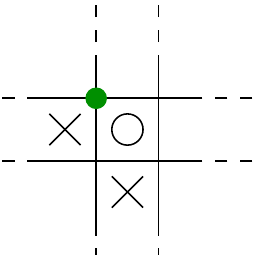} \ar@{|->}[r]^(.45){\dessin{.9cm}{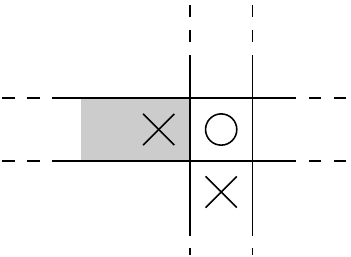}} & \dessin{1.3cm}{Gen1}\in C_2[-1] \hspace{.2cm}\\
        \dessin{1.3cm}{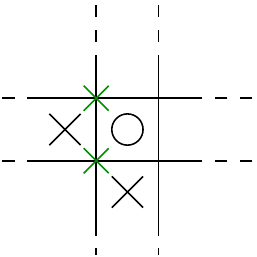} \ar@{|->}[r] & \hspace{.5cm} zero \hspace{1.2cm}}\end{array} \right..
  $$
  It is obviously a quasi-isomorphism for the graded chain complexes.
\item[v) Filtered extension of $F_{gr}$]
  The map $F_{gr}$ can be extended to a map $\func{F}{CV^-(G^s)}{C}$ of filtered chain complexes.
  Essentially, commutativity of $F$ with rectangles which are empty except $x_0$ which is actually contained in their interior, \ie rectangles embedded in the grid $G^s$ which are involved in the differential associated to $G$ as described in i) but not in the differential associated to $G^s$, generates inductively additional terms in the definition of $F$.
  Since its graded part is a quasi-isomorphism, $F$ is a quasi-isomorphism.
\end{description}

Details for points i), ii), iv) and v) are strictly similar than those of Section $3.2$ in \cite{MOST}.
Concerning the filtration of the point iii), we give an alternative construction.

First, we consider the filtration induced by the Alexander grading, the total polynomial degree in variables $U_{O_1}, \cdots, U_{O_n}$ and the desing grading.
Then the associated graded differential, denoted by $\widetilde{\p}$, counts only empty rectangles containing no decoration.
Now, we define a fourth grading on $\big(CV^-(G^s),\widetilde{\p}\big)$.

Let $x$ be a generator of $CV^-(G^s)$.
By construction, the extra RoCs in $G^s$ can be crushed in order to get back to $G$.
When crushing with $x$ drawn on the grid, it provides a set $\widetilde{x}$ of dots which is not a generator of $CV^-(G)$ since two singular grid lines, one horizontal and one vertical, have two dots on it.
It may happen that two dots merge, the resulting dot is then counted with multiplicity.

We perform a few cyclic permutations in such a way that the singular grid lines are the leftmost and the topmost ones.
The upper right corner of the grid is then filled with an $\X$--decoration which we denote by $X^*$.

\saut

Now, we can consider
$$
M_G(x):= M_{\O_G}(\widetilde{x})
$$
\noi where $\O_G$ is the set of $\O$--decorations of $G$ and $M_{.}(\ .\ )$ is the map defined in the introduction, which compares the relative positions of two planar sets of points.

\begin{figure}[ht]
  $$
  \dessinH{3.2cm}{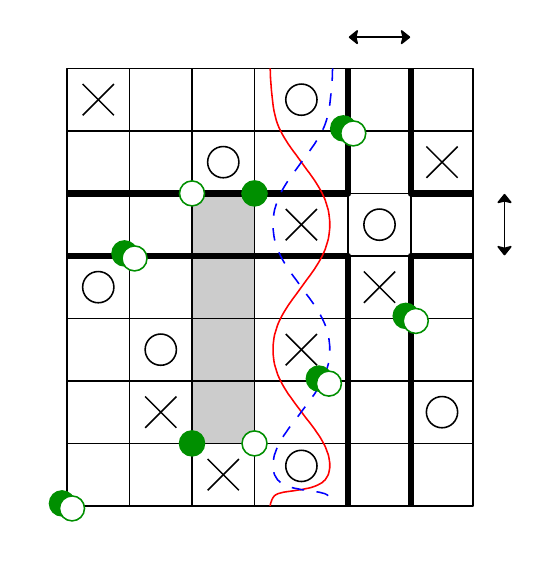}\ \begin{array}{c}\textrm{\scriptsize crushing}\\[-.1cm] \leadsto\\[-.1cm] \textrm{\scriptsize process} \end{array} \dessinH{3.2cm}{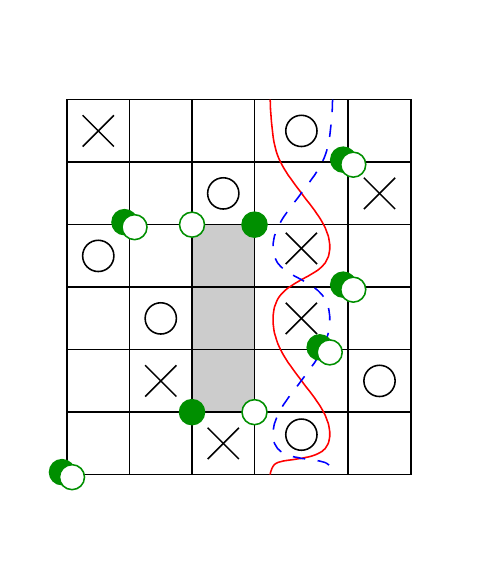} \begin{array}{c}\textrm{\scriptsize cyclic}\\[-.1cm] \leadsto\\[-.1cm] \textrm{\scriptsize perm.} \end{array} \dessinH{3.2cm}{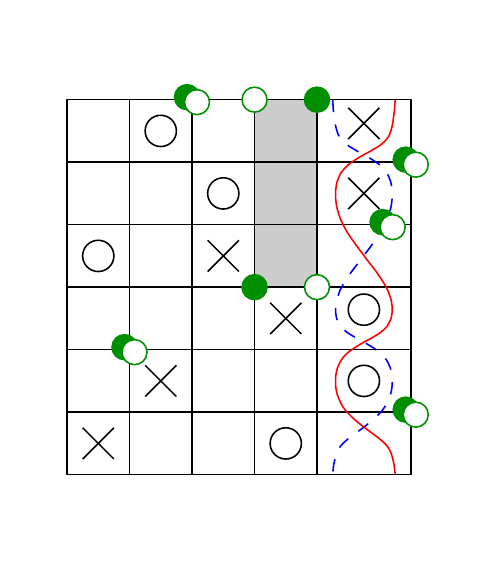}  \leadsto \ \ \begin{array}{l} M_G(x)=-3 \\[.3cm] M_G(y)=-4 \end{array}
  $$
  $$
  \dessinH{3.2cm}{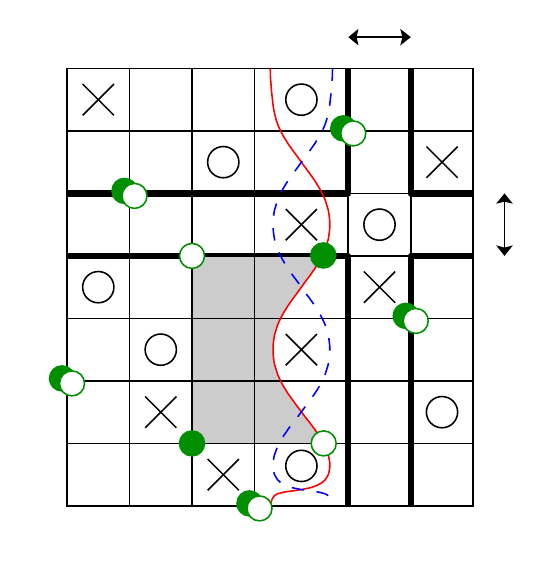}\ \begin{array}{c}\textrm{\scriptsize crushing}\\[-.1cm] \leadsto\\[-.1cm] \textrm{\scriptsize process} \end{array} \dessinH{3.2cm}{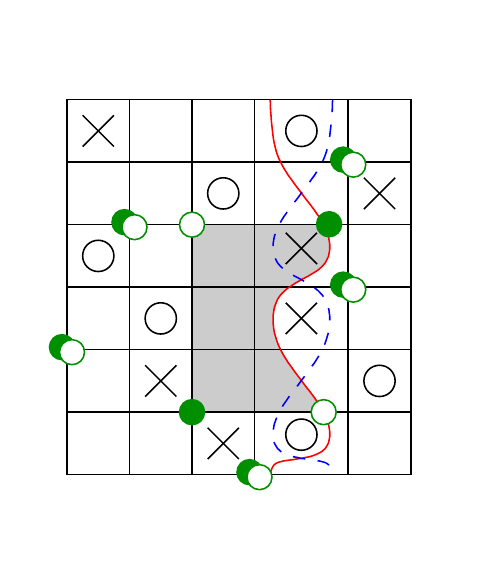} \begin{array}{c}\textrm{\scriptsize cyclic}\\[-.1cm] \leadsto\\[-.1cm] \textrm{\scriptsize perm.} \end{array} \dessinH{3.2cm}{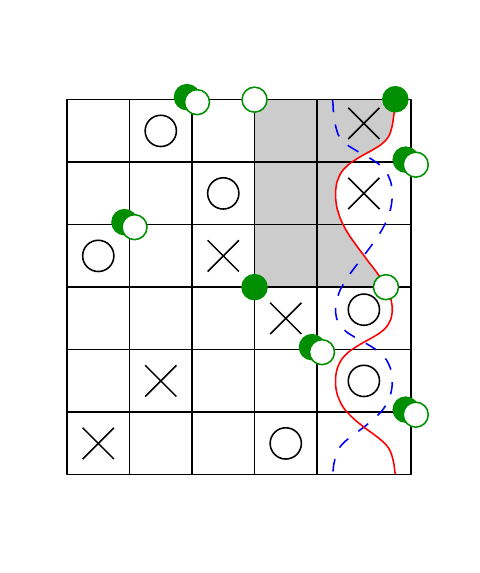}  \leadsto \ \ \begin{array}{l} M_G(x')=-6 \\[.3cm] M_G(y')=-7 \end{array}
  $$
  \caption*{Crushing row and column:  {\footnotesize dark dots describe the initial generators $x$ and $x'$ while hollow ones describe the final ones $y$ and $y'$. Polygons are depicted by shading.}}
  \label{fig:Crushing}
\end{figure}

\begin{lemme}
  \label{lem:Crushing_Filtration}
  $M_G$ defines a filtration on $\big(CV^-(G^s),\widetilde{\p}\big)$.
  The associated graded differential corresponds to the sum over rectangles which are contained in the row or in the column through $O_1$ and which do not contain $O_1$ or any $X$.
\end{lemme}
\begin{proof}
  Let $x$ and $y$ be two generators of $CV^-(G^s)$.
  
  In the crushing process, empty rectangles $\rho$ on $G^s$ containing no decoration give rise to empty rectangles $\widetilde{\rho}$ on $G$ which are also empty of decoration except, possibly, $X^*$.
  Actually, it may happen that a dot is pushed into $\p \rho$, but not into $\Int(\rho)$.
  If $\rho$ is connecting $x$ to $y$, then $\widetilde{\rho}$ is connecting $\widetilde{x}$ to $\widetilde{y}$.

  \saut
  
  If $\rho$ is totally flattened during the crushing process then $\widetilde{x}=\widetilde{y}$ and $M_G(x)=M_G(y)$.
  
  Now we assume that $\widetilde{\rho}$ is not flat.
  Since $X^*$ does not interfere with $M_G$, the fact that $\widetilde{\rho}$ contains it or not, does not matter.

  \vspace{-\baselineskip}

  {\parpic[r]{$\dessin{2.2cm}{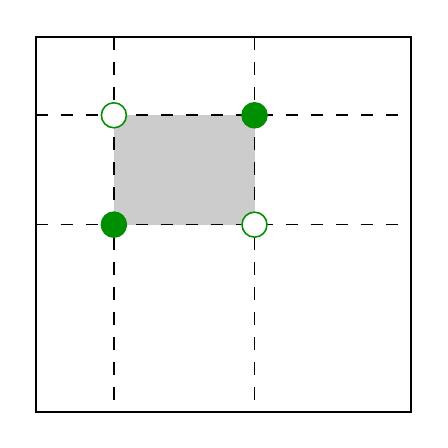}\ $}
    \vspace{.6cm}
    If $\widetilde{\rho}$ is not ripped and does not have any extra dot on its border, then the computation is identical than in the proof of Lemma $2.5$ of \cite{MOST}.
     
    Hence, we have $M_G(y)=M_G(x)-1$.

  }
  
  \vspace{.5\baselineskip}

  {\parpic[l]{$\ \dessin{2.2cm}{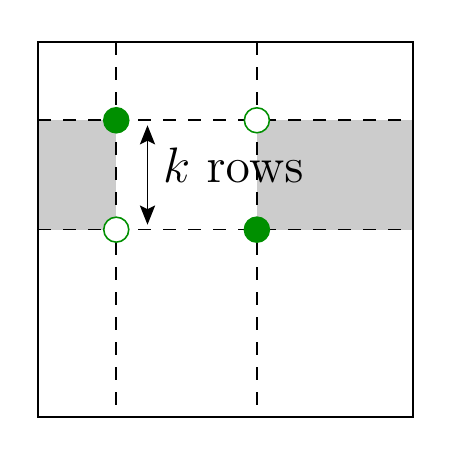}$}
    If $\widetilde{\rho}$ is horizontally ripped in two then we can compute using its horizontal complement which is not ripped and corresponds to a rectangle connecting $y$ to $x$.
    Since $\widetilde{\rho}$ is empty and contains no decoration, its complement must contain $k$ $\O$--decorations and $k-1$ dots, where $k$ is the height of $\widetilde{\rho}$.
      
    Then, we obtain $M_G(y)=M_G(x)+1+2(k-1)-2k=M_G(x)-1$.
  
  }
  

  {\parpic[r]{$\dessin{2.2cm}{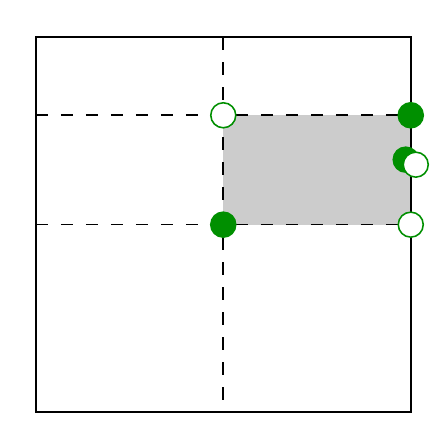}\ $}
    \vspace{.5cm}
    If $\widetilde{\rho}$ has an extra dot on its right border, then there is an extra term in $\I(\widetilde{x},\widetilde{x})$ which does not appear in $\I(\widetilde{y},\widetilde{y})$.
    
    As a result, $M_G(y)=M_G(x)-2$.
    
  }
  

  {\parpic[l]{\ $\dessin{2.2cm}{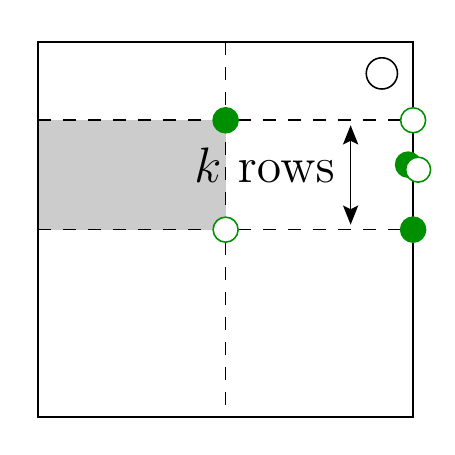}$}
    \vspace{.3cm}
    If $\widetilde{\rho}$ has an extra dot on its left border, we consider its horizontal complement $\widetilde{\rho}'$.
    Now, the extra term appears in $\I(\widetilde{y},\widetilde{y})$ but, on the other hand, $\widetilde{\rho}$ contains only $k-2$ dots in its interior.
       
    Finally, $M_G(y)=M_G(x)+2+2(k-2)-2k=M_G(x)-2$.

  }
  
  \saut

  Vertical ripping can be treated in the same way.
  Because of $O_1$, a rectangle cannot be ripped in four pieces.

  \saut
  
  Hence the grading $M_G$ induces a filtration on $\big(CV^-(G^s),\widetilde{\p}\big)$.
  Furthermore, it is clear that $M_G$ is only preserved by rectangles which are flattened during the crushing process.
  It corresponds exactly to rectangles contained in the row or in the column through $O_1$.
\end{proof}

\subsubsection{Rotation}
\label{sssec:Rotation}

Let $L$ be a singular link enhanced with an orientation for all its double points.
Let $G_h$ and $G_v$ be two grid diagrams for $L$ which differ from the following rotation move:
$$
\begin{array}{ccc}
  \dessin{2.2cm}{RotD1} & \ \longleftrightarrow \  & \dessin{2.2cm}{RotD2}.\\[1.3cm]
  G_h && G_v
\end{array}
$$
Furthermore, according to the construction of a set of peaks given in Section \ref{ssec:Homology_Statement}, we assume that the orientation of the double point involved in this move corresponds to the orientation inherited from the plane on which $G_h$ and $G_v$ are drawn.

We consider the four desingularizations of the involved RoCs.
All of them can be obtained by removing all winding arcs but two in the following grid $G$ (see (\ref{eq:RotInvDiag})):

\begin{eqnarray}
  \label{eq:Total_Grid}
  \dessin{6cm}{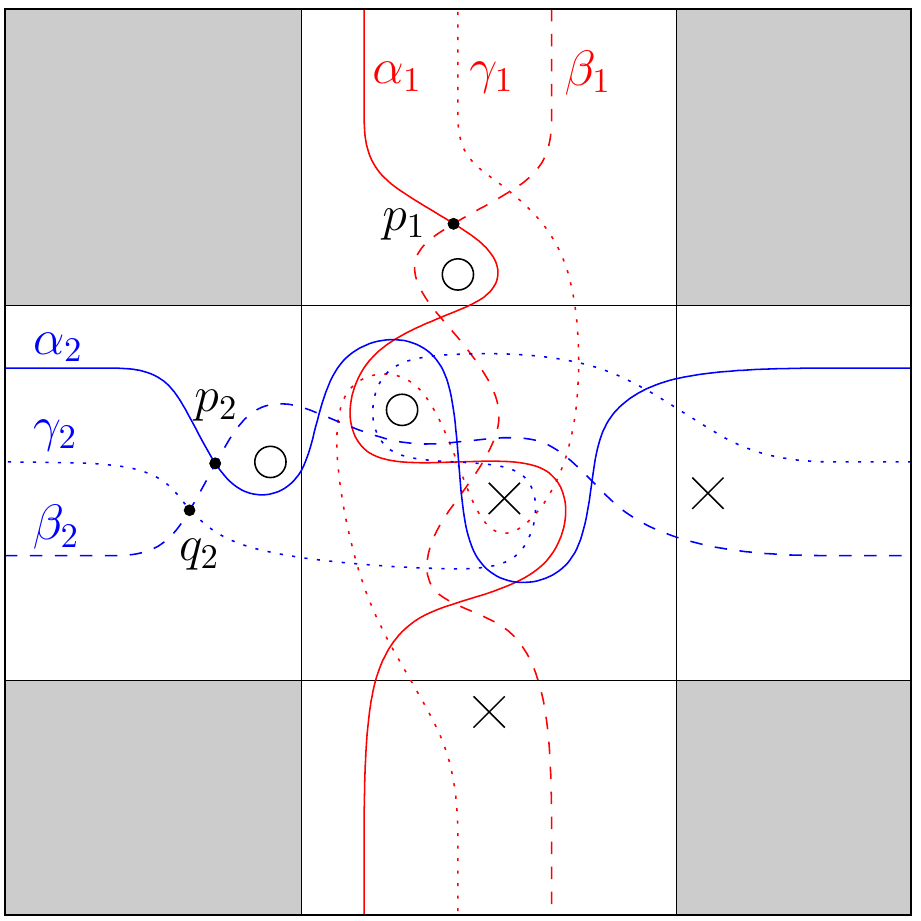}.
\end{eqnarray}

We denote by $G_{X_1X_2}$ where $(X_1,X_2)\in\{\alpha_1,\beta_1,\gamma_1\}\times\{\alpha_2,\beta_2,\gamma_2\}$ the grid obtained by removing $\disp{\bigcup_{i=1}^2}\left(\{\alpha_i,\beta_i,\gamma_i\}\setminus X_i\right)$ in $G$.
According to Corollary \ref{cor:Mapping_Cone}, the chain complexes $CV^-(G_h)$ and $CV^-(G_v)$ can be seen as the mapping cones of
$$
\func{f_{p_1}}{CV^-(G_{\alpha_1\beta_2})}{CV^-(G_{\beta_1\beta_2})}
$$
and
$$
\func{f_{p_2}}{CV^-(G_{\beta_1\alpha_2})}{CV^-(G_{\beta_1\beta_2})}
$$
where $p_1$ and $p_2$ are, respectively, the element of $\alpha_1\cup\beta_1$ of type $\dessin{.4cm}{Cstyle}$ and the element of $\alpha_2\cup\beta_2$ of type $\dessinH{.4cm}{rotCstyle}$ (see (\ref{eq:Total_Grid})).

The positive desingularizations of $G_h$ and $G_v$ can be linked by a sequence of two regular commutations.
In Section \ref{sssec:Commutation}, we have already defined quasi-isomorphisms
$$
\func{\phi_{\beta_2\gamma_2}}{CV^-(G_{\alpha_1\beta_2})}{CV^-(G_{\alpha_1\gamma_2})},
$$
$$
\func{\phi_{\gamma_2\beta_2}}{CV^-(G_{\alpha_1\gamma_2})}{CV^-(G_{\alpha_1\beta_2})}
$$
\noi and 
$$
\func{\phi_{\gamma_1\beta_1}}{CV^-(G_{\gamma_1\alpha_2})}{CV^-(G_{\beta_1\alpha_2})}.
$$

The following diagram synthesizes all the grids and the chain maps:
\begin{eqnarray}
  \xymatrix@!0@C=3.5cm@R=3.5cm{
    \dessin{2cm}{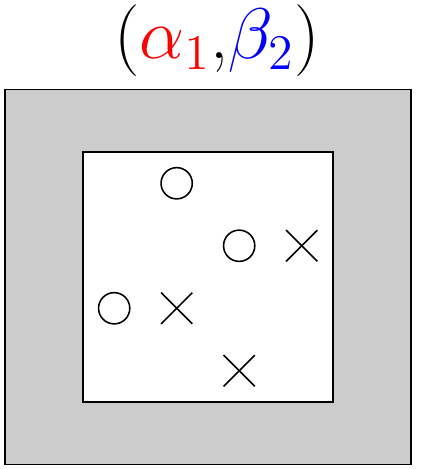} \ar@<-.1cm>[r]_{\phi_{\beta_2\gamma_2}} \ar[d]_{f_{p_1}} \ar@/_.55cm/@{-->}[drrr]^\phi & \dessin{2cm}{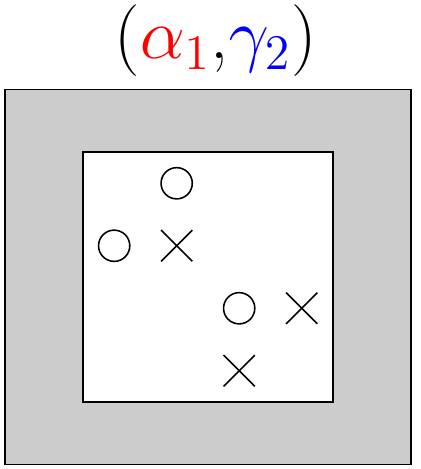} \ar@<-.1cm>[l]_{\phi_{\gamma_2\beta_2}} \ar@{-->}[r]^\psi & \dessin{2cm}{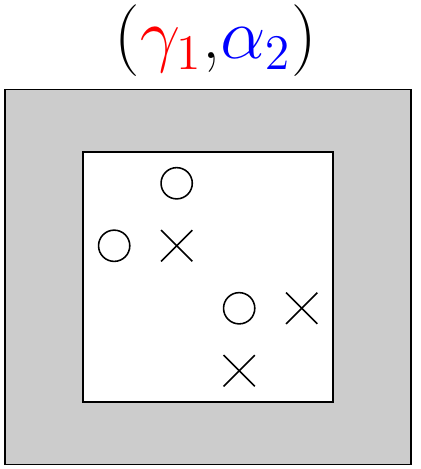} \ar[r]^{\phi_{\gamma_1\beta_1}} & \dessin{2cm}{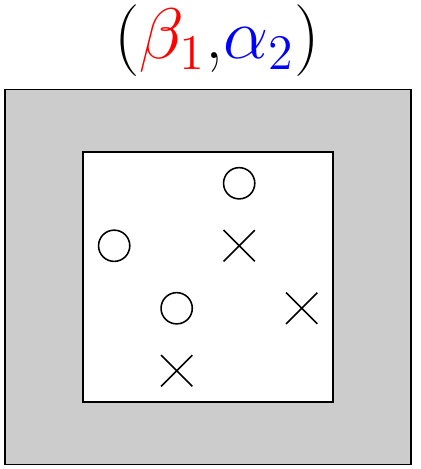} \ar[d]_{f_{p_2}}\\
    \dessin{2cm}{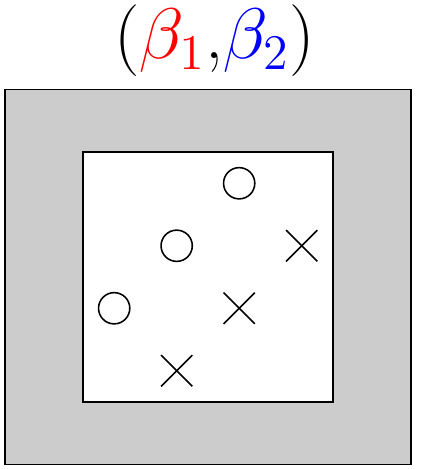} \ar[rrr]^\sim_{\Id} &&& \dessin{2cm}{MIT5}}.
  \label{eq:RotInvDiag}
\end{eqnarray}

Now we define two maps
\begin{gather*}
  \func{\psi}{CV^-(G_{\alpha_1\gamma_2})}{CV^-(G_{\gamma_1\alpha_2})},\\
  \func{\phi}{CV^-(G_{\alpha_1\beta_2})}{CV^-(G_{\beta_1\beta_2})}.
\end{gather*}
The map $\psi$ send a generator $x\in CV^-(G_{\alpha_1\gamma_2})$ to the unique generator $y\in CV^-(G_{\gamma_1\alpha_2})$ connected to $x$ by a pair of disjoint grid triangles which are empty and contain no decoration.
When $\alpha_1\cap\gamma_2\in x$, the triangles are degenerated and $x=y$ as sets of dots on $G$.

\begin{figure}[h]
  $$
  \begin{array}{ccc}
    \dessin{4cm}{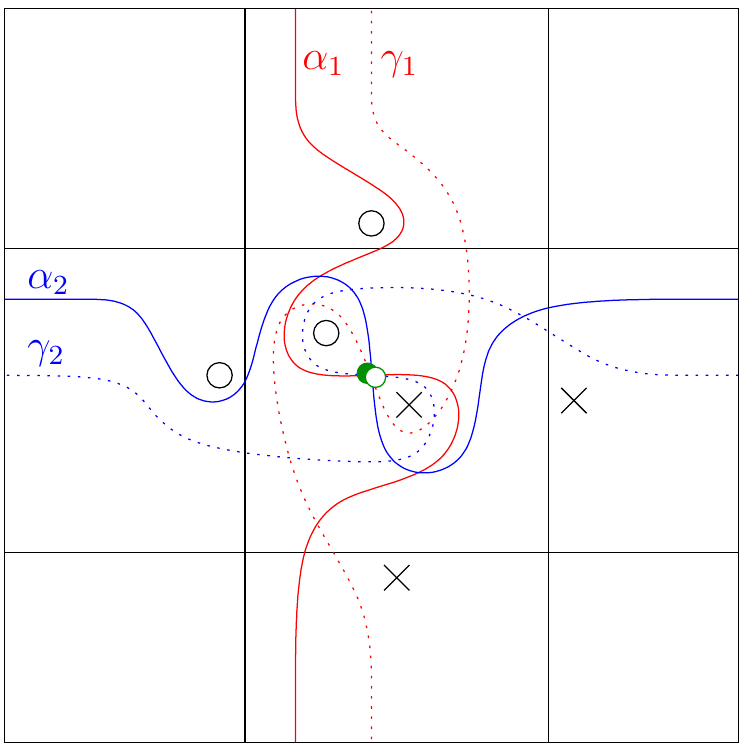} & \hspace{1cm} & \dessin{4cm}{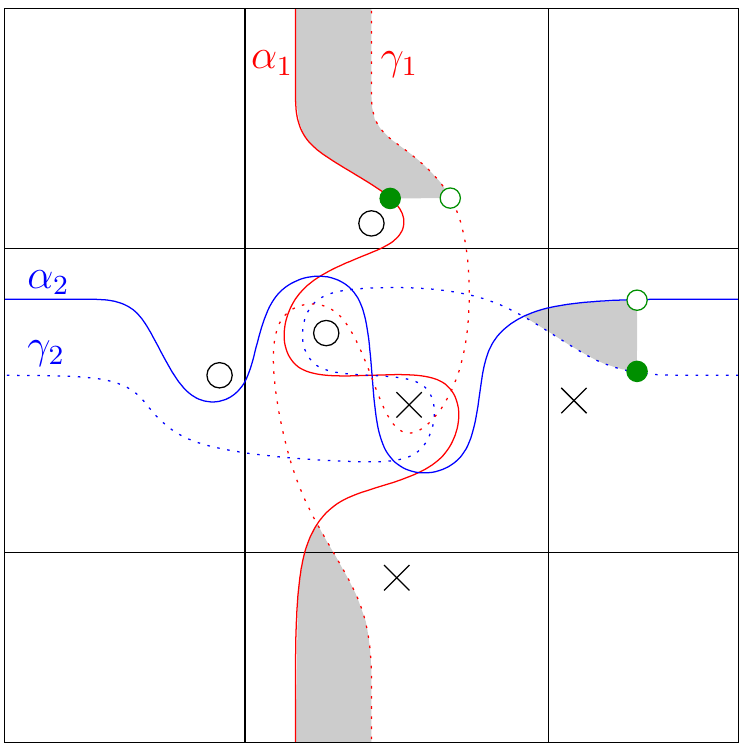} \\[2cm]
    \textrm{when }\alpha_1\cap\gamma_2\in x &&  \textrm{when }\alpha_1\cap\gamma_2\notin x
  \end{array}
  $$
  \caption*{A picture for $\psi$: {\footnotesize the generators $x$ and $y$ coincide except concerning the dark dots which belong to $x$ and the hollow ones which belong to $y$. Grid triangles are depicted by shading.}}
  \label{fig:Psi}
\end{figure}

For the definition of $\phi$, we need to introduce a new kind of polygons.
For any generators $x$ and $y$ in $CV^-(G_{\alpha_1\beta_2})$, a polygon $\Pi=\overline{\pi \setminus B}$ is a \emph{$\beta\gamma\beta$-polygon} connecting $x$ to $y$ if
\begin{itemize}
\item[-] $\pi$ is a commuting polygon in $\PolC (G_{\alpha_1\beta_2})$ connecting $x$ to $y$;
\item[-] $B$ is one of the two bigons delimited by the arcs $\beta_2$ and $\gamma_2$;
\item[-] $B\subset\pi$.
\end{itemize}
The $\beta\gamma\beta$-polygon $\Pi$ is \emph{empty} if $\Int (\Pi)\cap x=\emptyset$.
We denote by $\beta\gamma\beta$-$\Pol (x,y)$ the set of all empty $\beta\gamma\beta$-polygons, with at least six corners, connecting $x$ to $y$.

\begin{figure}[ht]
  $$
  \begin{array}{ccc}
    \dessin{4cm}{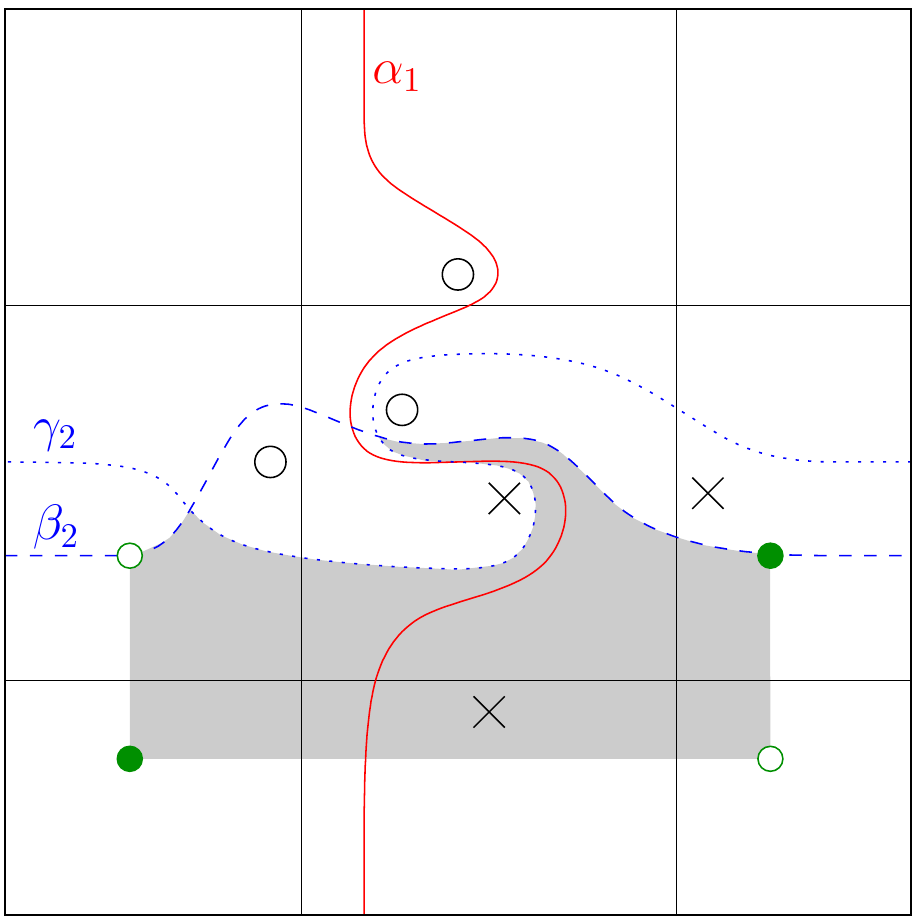} & \hspace{1cm} & \dessin{4cm}{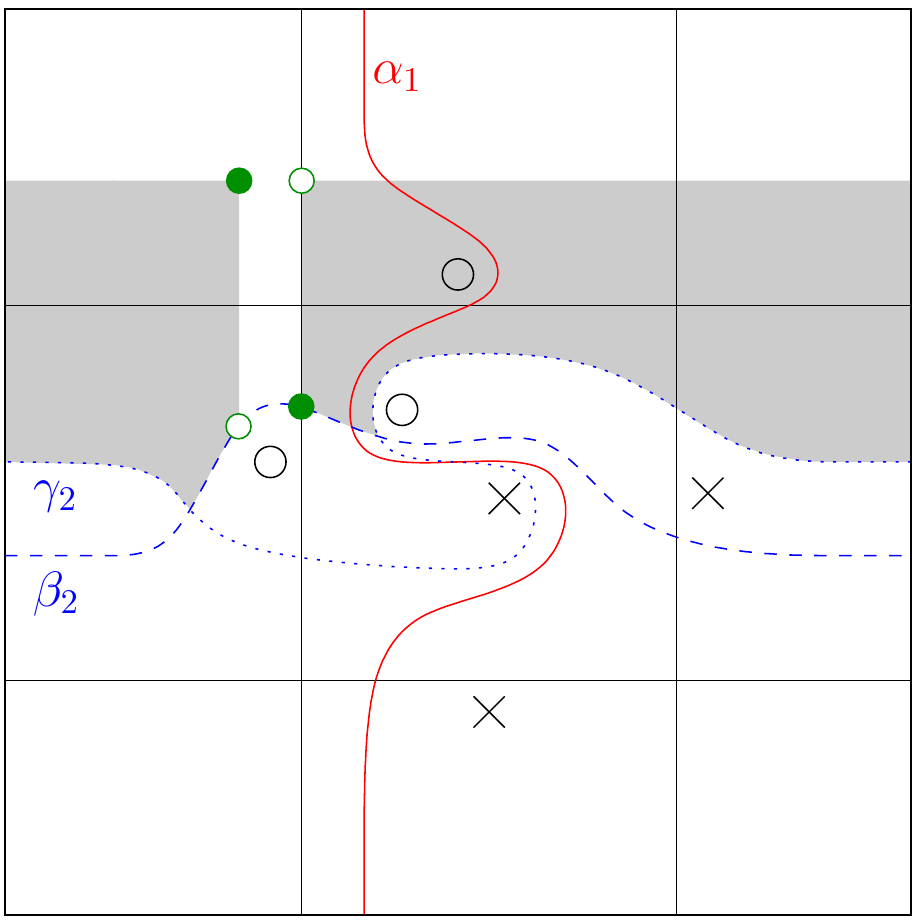}
  \end{array}
  $$ 
  \caption*{Examples of $\beta\gamma\beta$-polygons connecting $x$ to $y$: {\footnotesize dark dots describe the generator $x$ while hollow ones describe $y$. Grid polygons are depicted by shading.}}
  \label{fig:bgbPolygons}
\end{figure}

\saut

Now, we set the map $\func{\varphi_1}{CV^-(G_{\alpha_1\beta_2})}{CV^-(G_{\alpha_1\beta_2})}$ as the morphism of $\Z[U_{O_1},\cdots,U_{O_{n}}]$--modules defined on the generators by
$$
\varphi_1(x) = \sum_{\substack{y \textrm{ generator}\textcolor{white}{R}\\ \textrm{of }CV^-(G_{\alpha_1\beta_2})}\textcolor{white}{I}} \sum_{\Pi \in \beta\gamma\beta\textrm{-}\Pol (x,y)} \e(\pi)U_{O_1}^{O_1(\Pi)}\cdots U_{O_{n}}^{O_{n}(\Pi)}\cdot y
$$
\noi where $\pi$ is the element of $\Pol (G_{\alpha_1\beta_2})$ such that $\Pi=\overline{\pi \setminus B}$ for a bigon $B$.

\saut

We also define $\func{\varphi_2}{CV^-(G_{\alpha_1\gamma_2})}{CV^-(G_{\beta_1\beta_2})}$ by
$$
\varphi_2(x) = \sum_{\substack{y \textrm{ generator}\textcolor{white}{R}\\ \textrm{of }CV^-(G_{\beta_1\beta_2})}\textcolor{white}{I}} \sum_{\pi \in \Pol (x,y)} \e(\pi)U_{O_1}^{O_1(\pi)}\cdots U_{O_{n}}^{O_{n}(\pi)}\cdot y
$$
\noi where $x$ is a generator of $CV^-(G_{\alpha_1\gamma_2})$ and $\Pol (x,y)$ is the set of empty grid polygons containing $p_1$ and $q_2$, the leftmost element of $\beta_2\cap\gamma_2$ (see \ref{eq:Total_Grid}), as peaks and connecting $x$ to $y$.

\saut

Then we set $\phi=f_{p_1}\circ \varphi_1+\varphi_2\circ\phi_{\beta_2\gamma_2}$.
Since $CV^-(G_v)$ and $CV^-(G_h)$ can, respectively, be seen as the mapping cones of $f_{p_1}$ and $f_{p_2}$, we finally define $\func{\Phi}{CV^-(G_v)}{CV^-(G_h)}$ by
$$
\Phi(x)=\left\{
  \begin{array}{ll}
    \phi_{\gamma_1\beta_1}\circ\psi\circ\phi_{\beta_2\gamma_2}(x) + \phi(x) & \textrm{if }x\textrm{ is a generator of }CV^-(G_{\alpha_1\beta_2})\\
    x & \textrm{if }x\textrm{ is a generator of }CV^-(G_{\beta_1\beta_2})
  \end{array}
\right..
$$

\begin{lemme}
  \label{lem:Phi_Quasi_Isomorphism}
  The map $\Phi$ is a quasi-isomorphism.
\end{lemme}
\begin{figure}[p]
  \centering
  \subfigure[]{\fbox{\xymatrix@R=.6cm@C=1.45cm{
        \dessin{5.5cm}{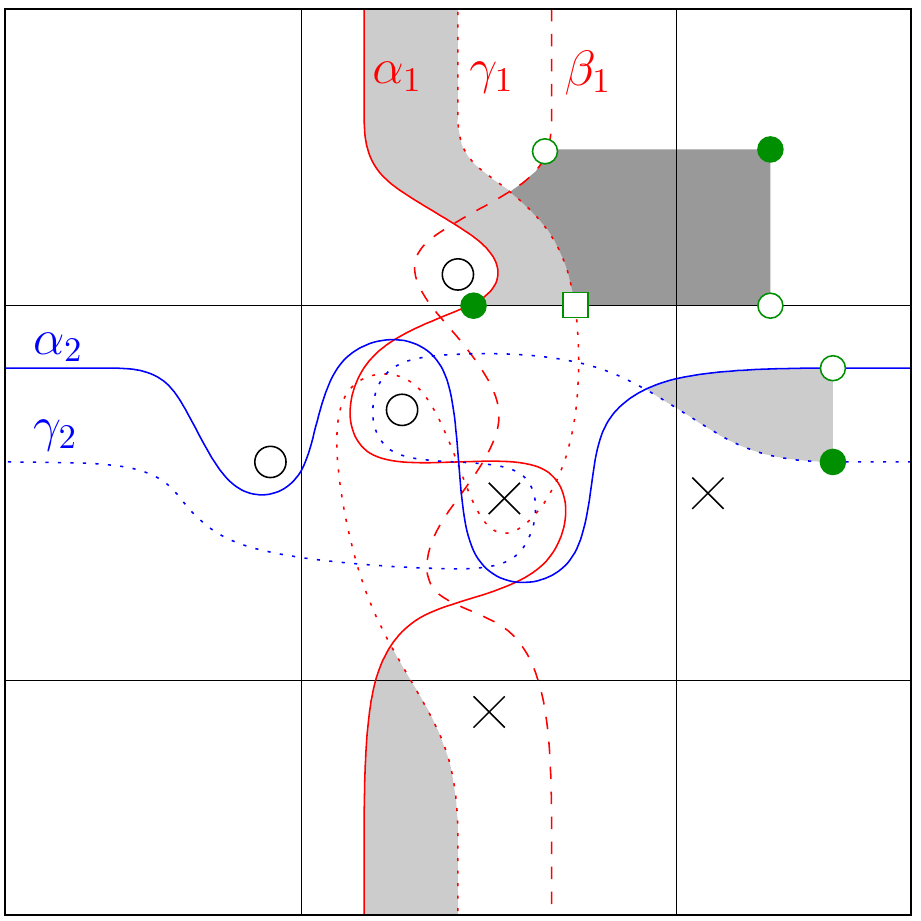} \ar@{^<-_>}[r]_(.48){\phi_{\gamma_1\beta_1}\circ \psi}^(.57){f'_{p_1}} & \dessin{5.5cm}{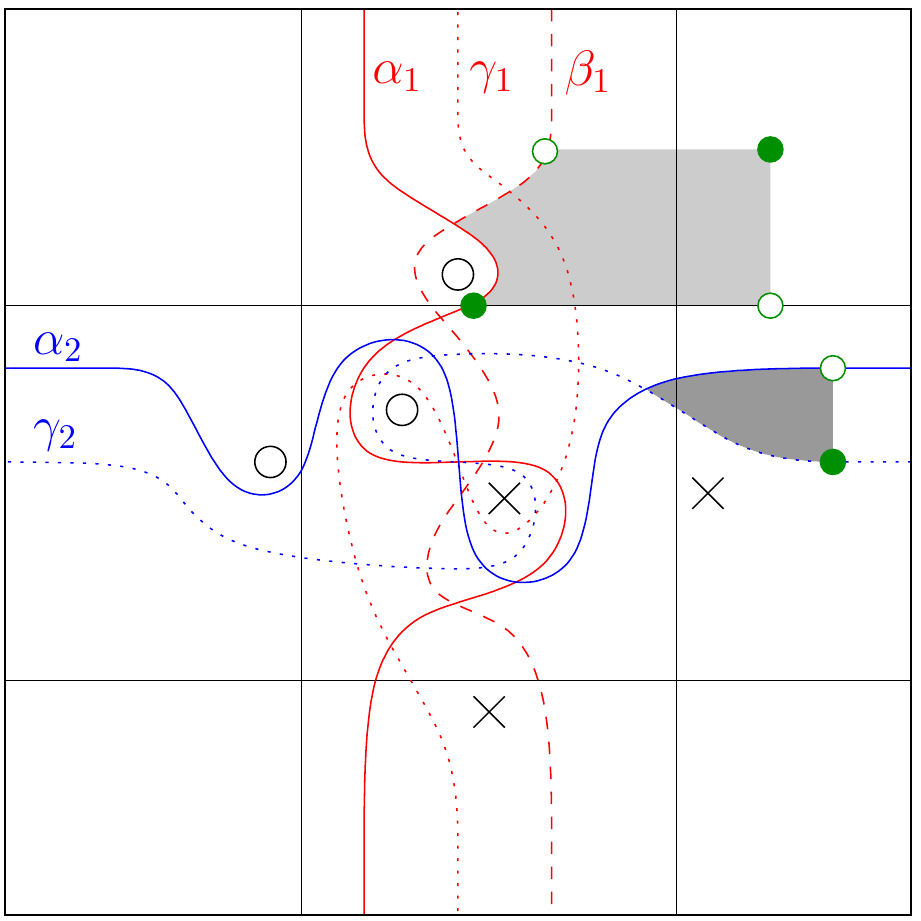}}}}\\
  \subfigure[]{\fbox{\xymatrix@R=.6cm@C=1.45cm{
        \dessin{5.5cm}{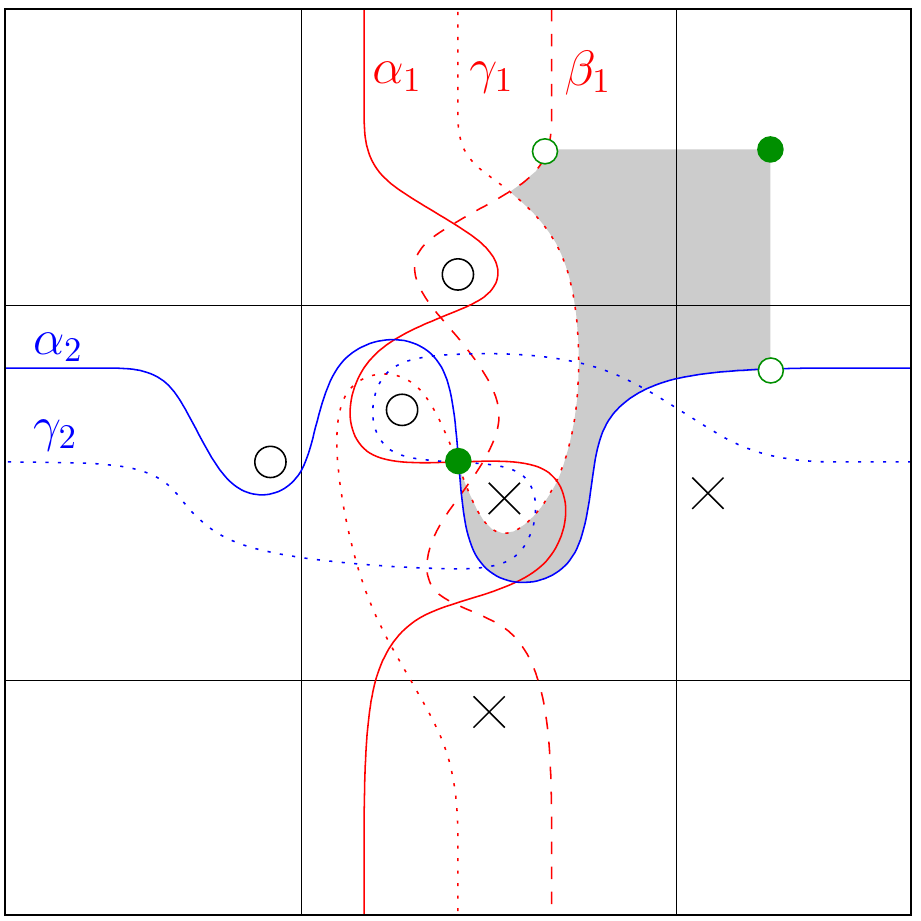} \ar@{^<-_>}[r]_(.48){\phi_{\gamma_1\beta_1}\circ \psi}^(.57){f'_{p_1}} & \dessin{5.5cm}{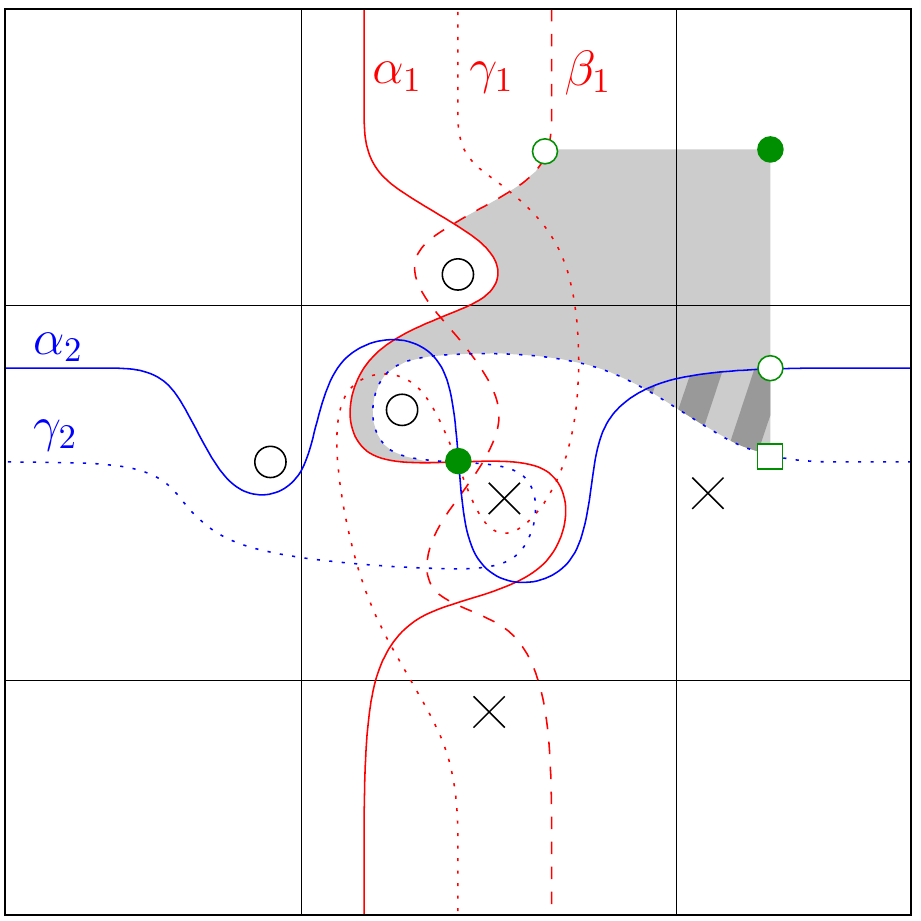}}}}\\
  \subfigure[]{\fbox{\xymatrix@R=.6cm@C=1.45cm{
        \dessin{5.5cm}{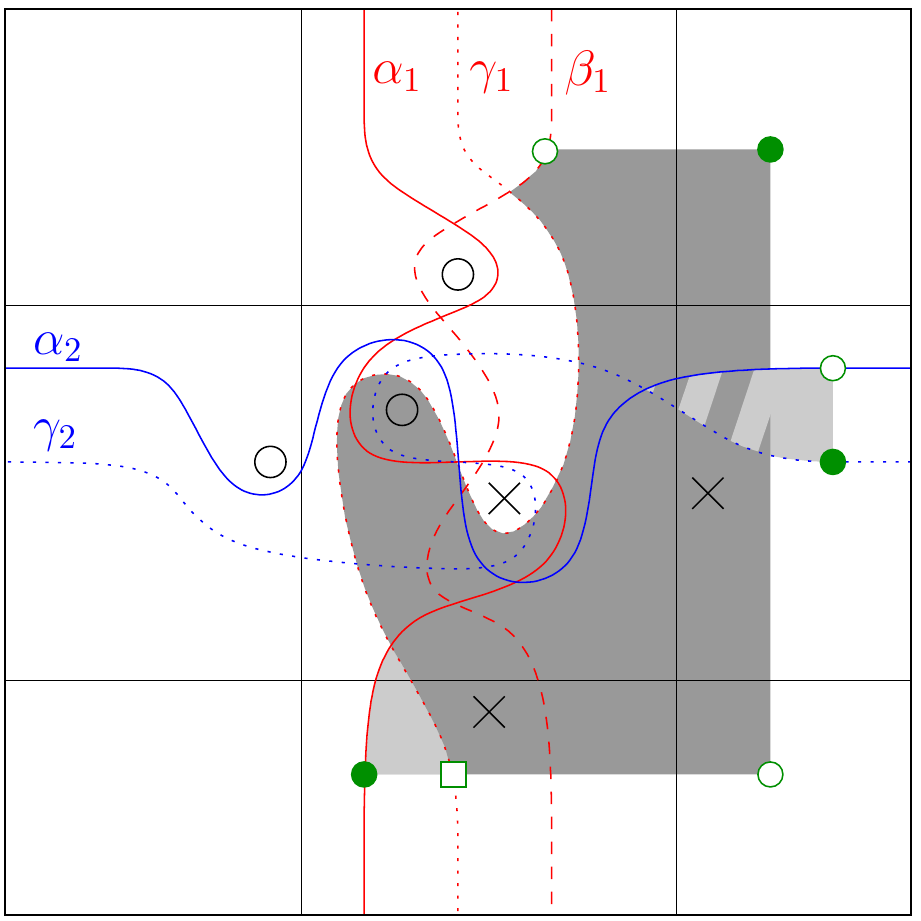} \ar@{^<-_>}[r]_(.48){\phi_{\gamma_1\beta_1}\circ \psi}^(.57){f'_{p_1}} & \dessin{5.5cm}{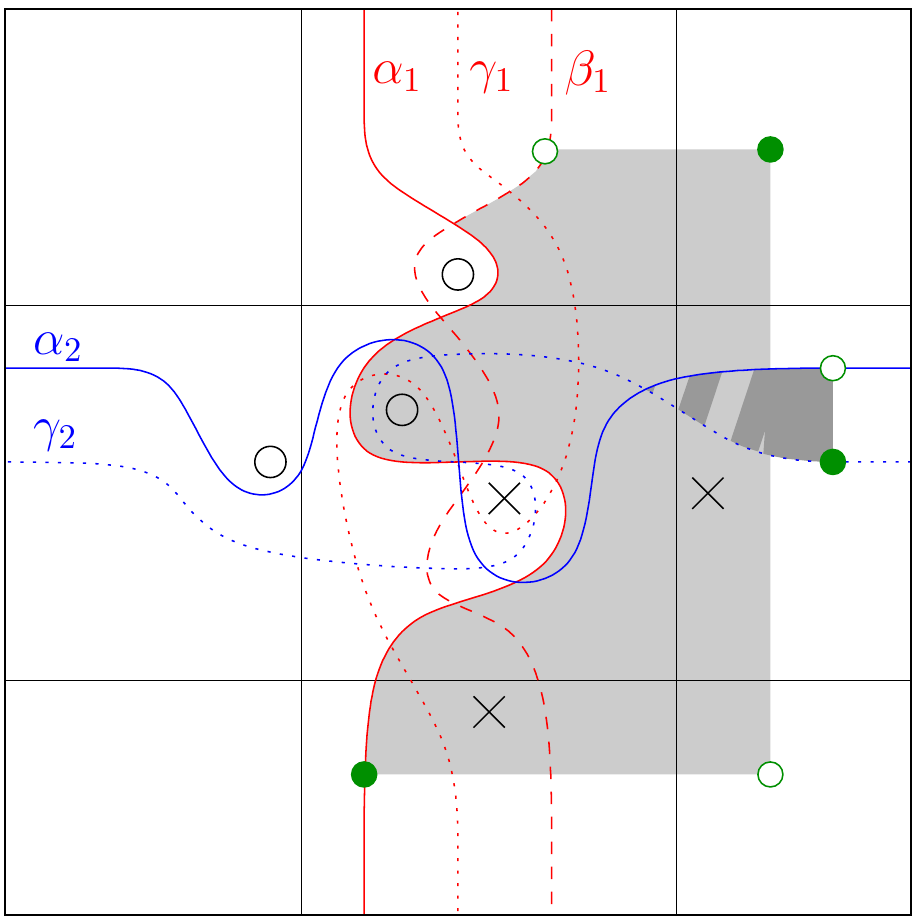}}}}
  \caption{Correspondence between grid polygons: {\footnotesize Dark dots describe the initial generator while hollow ones describe the final one. Squares describe intermediate states. Polygons are depicted by shading. The light gray one is considered first, then the dark one. For each polygon(s), we indicate to which map it belongs. Note that associated polygons share the same initial and final generators and the same sign rule. Moreover, they contain the same decorations. Other cases are similar.}}
  \label{fig:Correspondence}
\end{figure}

\addtocounter{subfigure}{3}

\begin{figure}[p]
  \vspace{1cm}
  \centering
  \subfigure[]{\fbox{\xymatrix@R=.6cm@C=1.45cm{
        \dessin{5.5cm}{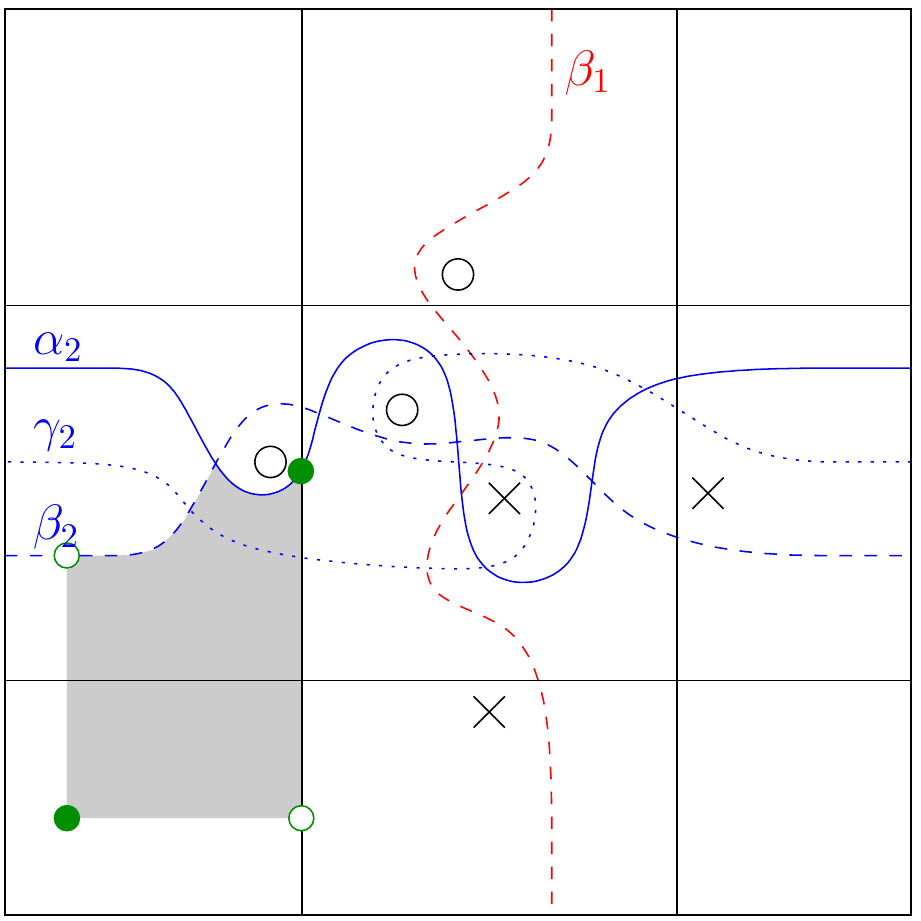} \ar@{^<-_>}[r]_(.43){f_{p_2}}^(.55){\phi'_{\gamma_2\beta_2}} & \dessin{5.5cm}{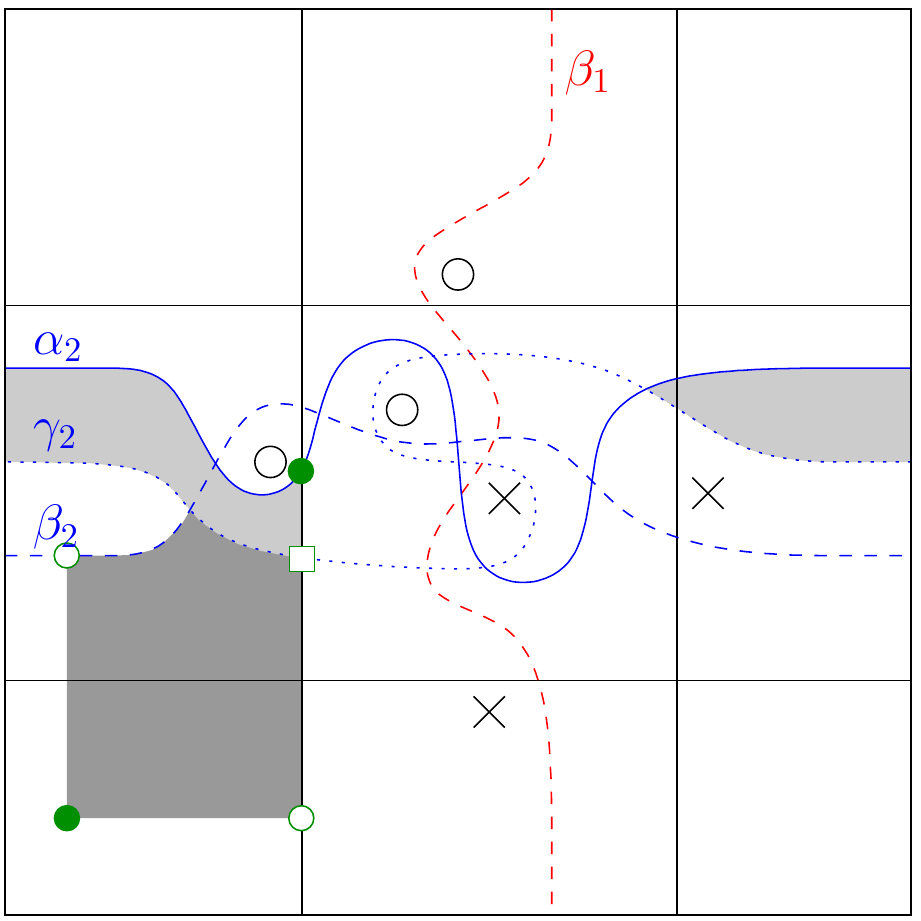}}}}\\
  \subfigure[]{\fbox{\xymatrix@R=.6cm@C=1.45cm{
        \dessin{5.5cm}{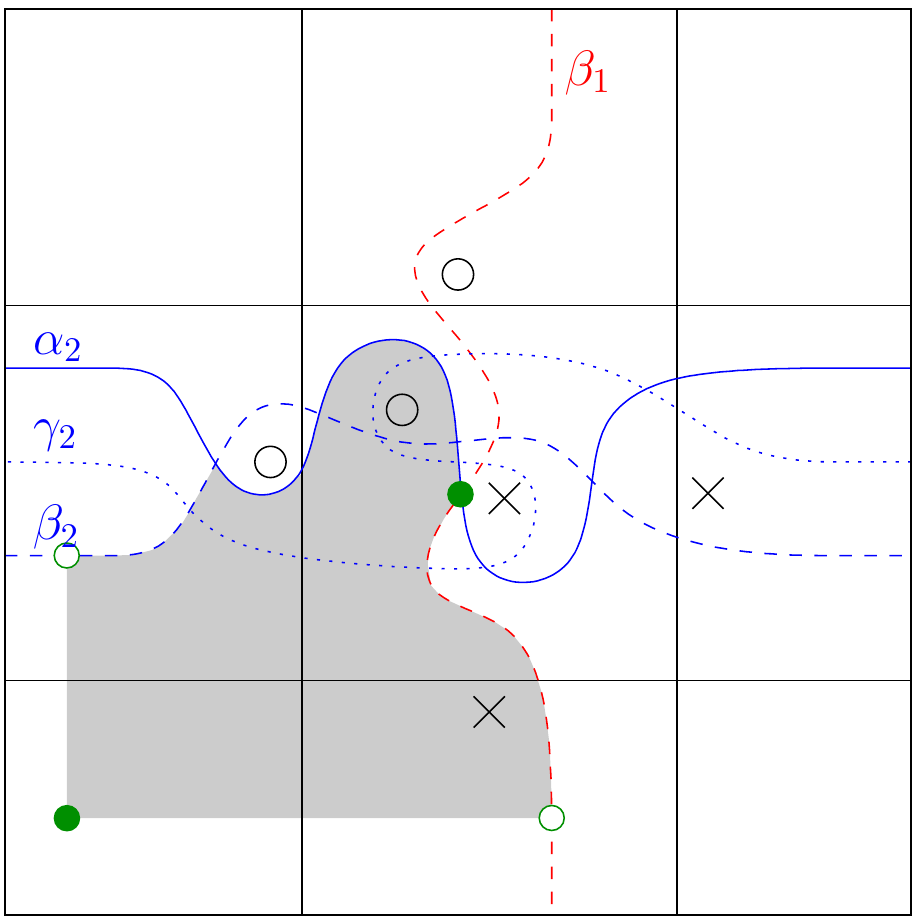} \ar@{^<-_>}[r]_(.43){f_{p_2}}^(.55){\phi'_{\gamma_2\beta_2}} & \dessin{5.5cm}{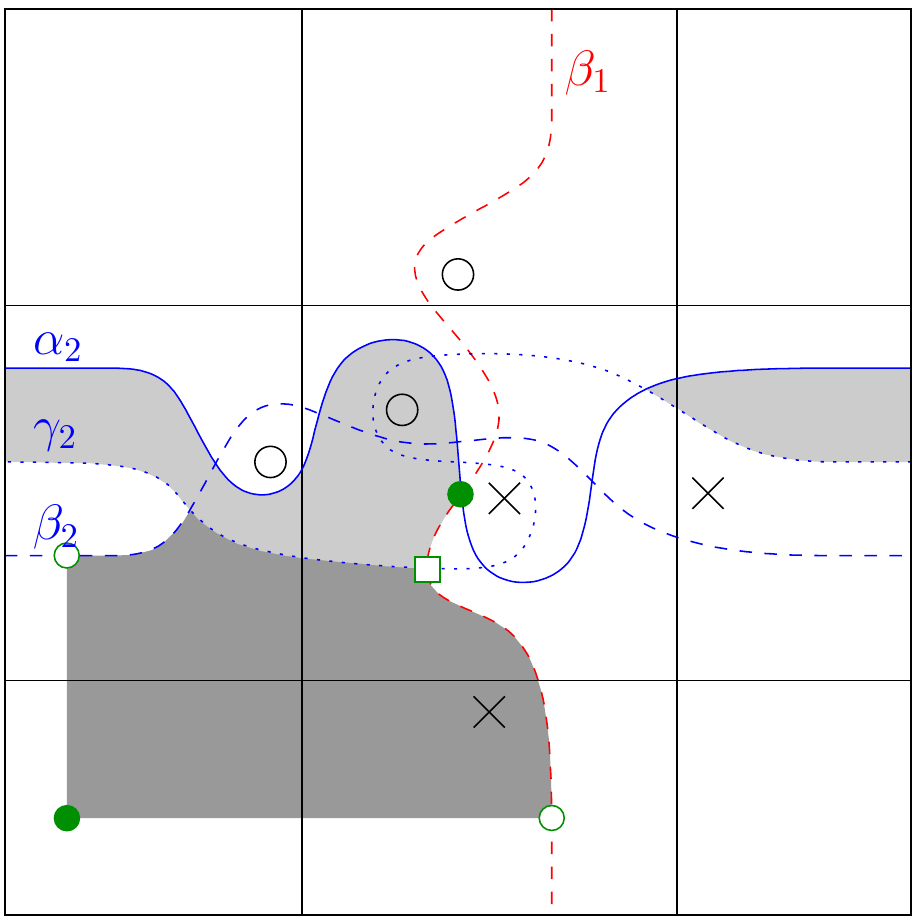}}}}\\
  \subfigure[]{\fbox{\xymatrix@R=.6cm@C=1.45cm{
        \dessin{5.5cm}{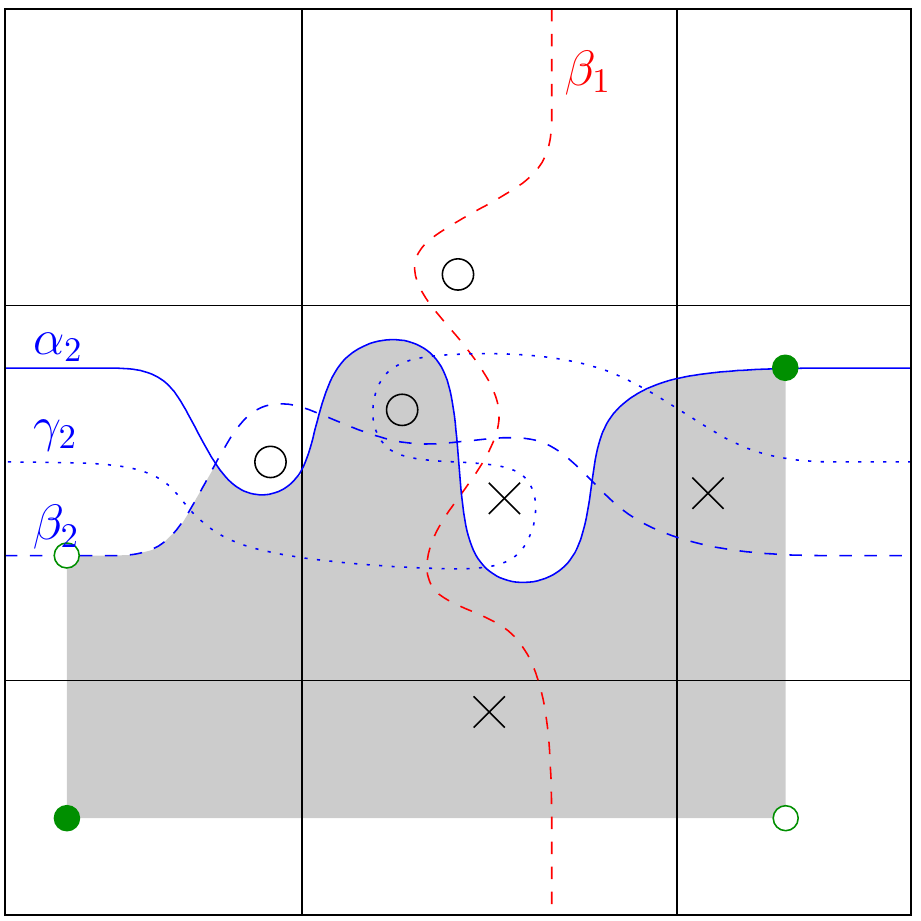} \ar@{^<-_>}[r]_(.43){f_{p_2}}^(.55){\phi'_{\gamma_2\beta_2}} & \dessin{5.5cm}{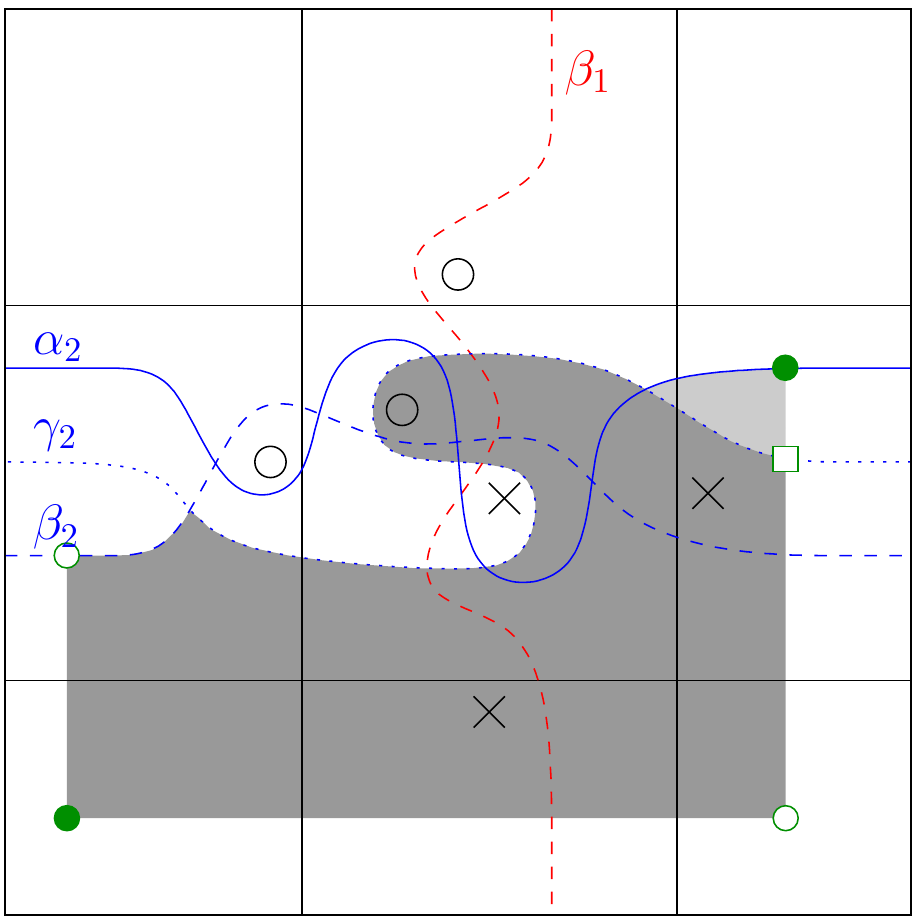}}}}
  \caption*{\vspace{\baselineskip}}
\end{figure}
\begin{proof}
  To prove that $\Phi$ commutes with the differentials, it is sufficient to prove that
  \begin{eqnarray}
    \label{eq:Formule}
    f_{p_1}+\Phi\circ\p^-_{G_{\alpha_1\beta_2}} \equiv f_{p_2}\circ\phi_{\gamma_1\beta_1}\circ\psi\circ\phi_{\beta_2\gamma_2} + \p^-_{G_{\beta_1\beta_2}}\circ\Phi.
  \end{eqnarray}
  
  Let $x$ be a generator of $CV^-(G_{\alpha_1\gamma_2})$. 
  Figure \ref{fig:Correspondence} illustrates a correspondence between the polygons involved in $\phi_{\gamma_1\beta_1}\circ\psi$ (resp. $f_{p_2}$) and some polygons involving $p_1$ (resp. $q_2$) as a peak.
  A chain map $f'_{p_1}$ (resp. $\phi'_{\gamma_2\beta_2}$) can be defined by summing over the latter.
  Moreover, in this correspondence, associated polygons contain the same decorations.
  Then their combinatorics are identical and we can refer to the proof of Proposition \ref{prop:Chain_Complex} to claim that
  \begin{eqnarray}
    f_{p_2}\circ\phi_{\gamma_1\beta_1}\circ\psi(x) +  f_{p_1}\circ\phi_{\gamma_2\beta_2}(x)+\p^-_{G_{\beta_1\beta_2}}\circ\varphi_2(x) + \varphi_2\circ\p^-_{G_{\alpha_1\gamma_2}}(x)=0.
    \label{eq:Magie}
  \end{eqnarray}  
  Moreover, it follows from the proof of propositions 3.2 and 4.24 in \cite{MOST} that
  $$
  \Id + \phi_{\gamma_2\beta_2}\circ\phi_{\beta_2\gamma_2} + \p^-_{G_{\alpha_1\beta_2}}\circ\varphi_1 + \varphi_1\circ\p^-_{G_{\alpha_1\beta_2}}\equiv 0.
  $$
  Then
  \begin{eqnarray*}
    0 & \equiv & f_{p_1} + f_{p_1}\circ\phi_{\gamma_2\beta_2}\circ\phi_{\beta_2\gamma_2} + f_{p_1}\circ\p^-_{G_{\alpha_1\beta_2}}\circ\varphi_1 + f_{p_1}\circ\varphi_1\circ\p^-_{G_{\alpha_1\beta_2}}\\[.3cm]
    & \equiv & f_{p_1} + f_{p_1}\circ\phi_{\gamma_2\beta_2}\circ\phi_{\beta_2\gamma_2} - \p^-_{G_{\beta_1\beta_2}}\circ f_{p_1}\circ\varphi_1 + f_{p_1}\circ\varphi_1\circ\p^-_{G_{\alpha_1\beta_2}},\\
  \end{eqnarray*}
  \noi and using (\ref{eq:Magie}), we obtain
  \begin{eqnarray*}
    0 & \equiv & f_{p_1} - f_{p_2}\circ\phi_{\gamma_1\beta_1}\circ\psi\circ\phi_{\beta_2\gamma_2} - \p^-_{G_{\beta_1\beta_2}}\circ\varphi_2\circ\phi_{\beta_2\gamma_2} - \varphi_2\circ\p^-_{G_{\alpha_1\gamma_2}}\circ\phi_{\beta_2\gamma_2}\\
    && \hspace{6cm} - \p^-_{G_{\beta_1\beta_2}}\circ f_{p_1}\circ\varphi_1 + f_{p_1}\circ\varphi_1\circ\p^-_{G_{\alpha_1\beta_2}}\\[.3cm]
    & \equiv & f_{p_1} + f_{p_1}\circ\varphi_1\circ\p^-_{G_{\alpha_1\beta_2}} + \varphi_2\circ\phi_{\beta_2\gamma_2}\circ\p^-_{G_{\alpha_1\beta_2}}  - f_{p_2}\circ\phi_{\gamma_1\beta_1}\circ\psi\circ\phi_{\beta_2\gamma_2}\\
    && \hspace{6cm}  - \p^-_{G_{\beta_1\beta_2}}\circ\varphi_2\circ\phi_{\beta_2\gamma_2} - \p^-_{G_{\beta_1\beta_2}}\circ f_{p_1}\circ\varphi_1,
  \end{eqnarray*}
  what concludes the proof of (\ref{eq:Formule}) since the remaining terms cancel each other out.

  \saut
  
  Now we consider the filtration induced by the grading which sends a generator $x$ of $CV^-(G_v)$ (resp. $CV^-(G_h)$) to $0$ or $1$ depending on the resolution of the singular column (resp. singular row) involved in the rotation move.
  The graded part associated to $\Phi$ is a composition of quasi-isomorphisms.
  The whole map $\Phi$ is hence a quasi-isomorphism.
\end{proof}

\begin{figure}[t]
  $$
  \xymatrix{
    \dessin{2.5cm}{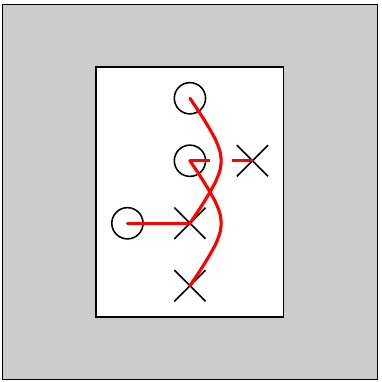} \ar@{^<-_>}[r]_(.42){\dessin{.3cm}{Cstyle}}^(.57){\dessinH{.3cm}{rotCstyle}} & \dessin{2.5cm}{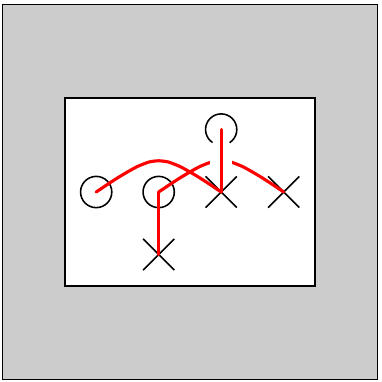} & & \dessin{2.5cm}{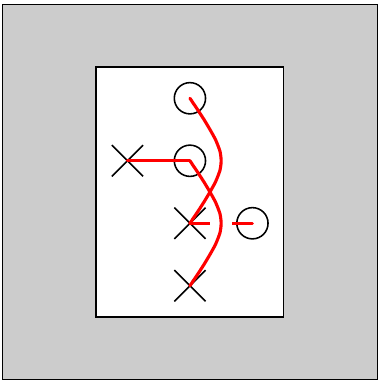} \ar@{^<-_>}[r]_(.42){\dessin{.3cm}{Cstyle}}^(.57){\dessinH{.3cm}{rotCpstyle}} & \dessin{2.5cm}{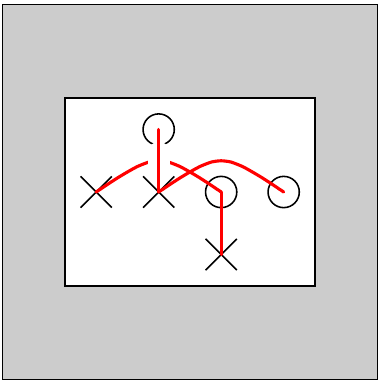}\\
    \dessin{2.5cm}{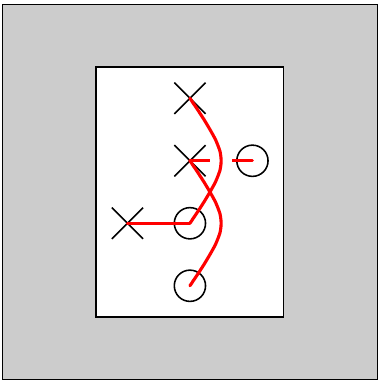} \ar@{^<-_>}[r]_(.42){\dessin{.3cm}{Cpstyle}}^(.57){\dessinH{.3cm}{rotCpstyle}} & \dessin{2.5cm}{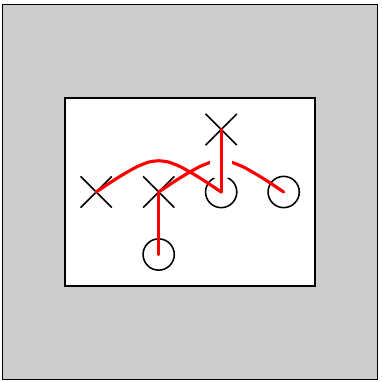} & &\dessin{2.5cm}{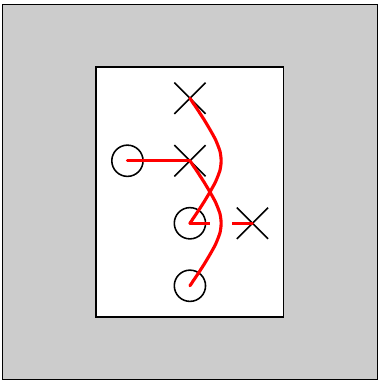} \ar@{^<-_>}[r]_(.42){\dessin{.3cm}{Cpstyle}}^(.57){\dessinH{.3cm}{rotCstyle}} & \dessin{2.5cm}{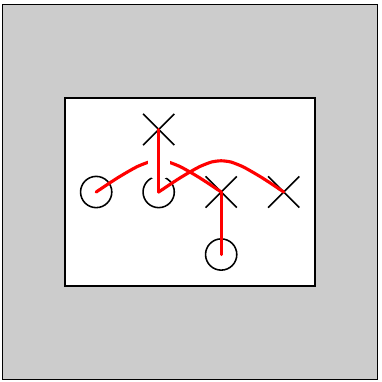}
  }
  $$
  \caption{Four rotation moves: {\footnotesize For each move, we indicate the convention for arcs intersection which is considered in the proof of invariance.}}
  \label{fig:FourRotations}
\end{figure}

By symmetry, the proof can be adapted to fit the four rotation moves given in Figure \ref{fig:FourRotations}.
They correspond, up to equivalence, to the two rotation moves for the two choices of arcs intersection.
It is straighforwardly checked that they simultaneously change or preserve, in one hand, the choice of arcs intersections considered in the proof and, on the other hand, the orientation induced by the planar diagrams for the double point involved in the rotation move .



%% file: Properties.tex
\section{Algebraic properties}
\label{sec:Properties}

In this section, we gather a few properties.

\input{Properties_GradedHomologies}

\input{Properties_Symetries}

\input{Properties_Acyclicity}


%% file: Properties_GradedHomologies.tex
\subsection{Graded homologies}
\label{ssec:Properties_GradedHomologies}

For any grid $G$, the chain complex $CV^-(G)$ is enhanced with several filtrations.
Now we consider the associated graded objects \ie the differentials obtained by removing terms which increase or decrease the gradings.

The graded differential associated to the Alexander grading, denoted by removing the minus exponent, corresponds to the sum over the grid polygons which do not contain any $\X$--decoration.
The associated homology is also denoted by removing the minus exponent.

For any $\O$--decoration $O$, there is a filtration induced by the polynomial degree in $U_O$.
The associated graded differential counts grid polygons which do not contain $O$.
It corresponds also to sending the variable $U_O$ to zero.
The differential and the homology obtained by sending all the variables to zero are denoted by adding a tilde.

In case of a link $L$, the set of $\O$--decorations can be partitioned according to the component of $L$, the decorations belong to.
Since $\O$--decorations induce a one-to-one correspondence between the rows and the columns of any desingularization of $G$, $\O$-decorations can be seen as a permutation.
The partition defined above corresponds to the decomposition of this permutation into disjoint cycles.
The homology obtained by sending one variable in each equivalence class to zero is denoted by adding a hat.
The following proposition proves that it does not depend on the choice of the representatives.

\begin{prop}
  \label{prop:Graded_Homologies}
  Let $G$ be a grid diagram of size $n$ for a singular link $L$ with $\ell$ components.
  The graded homologies $HV(G)$ and $\widehat{HV}(G)$ depend only on $L$ and on an orientation for its double points.
  Furthermore $\widetilde{HV}(G)\equiv\widehat{HV}(G)\otimes V^{\otimes(n-\ell)}$, where $V$ is a free bigraded $\Z$--module generated by two elements of bidegrees $(0,0)$ and $(-1,-1)$.
\end{prop}

The first statement is obtained by restricting the proofs given in Section \ref{ssec:Homology_Consistency} to the considered grid polygons.

The second is a straighforward adaptation of Proposition $2.13$ in \cite{MOST}.

\saut

\begin{defi}
  For any link $L$ with oriented double points, the homology $\widehat{HV}(L)$ is called the singular link Floer homology of $L$.
  It is also denoted by $\widehat{HFV}(L)$.
\end{defi}


%% file: Properties_Symetries.tex
\subsection{Symmetries}
\label{ssec:Properties_Symmetries}

Some Alexander polynomial properties have their counterpart in link Floer homology.

\saut

Let $L$ be an oriented link with $\ell$ components and $k$ oriented double points.
We denote by
\begin{itemize}
\item[-] $L^!$ its mirror image with reversed orientations for the double points;
\item[-] $-L$ the link obtained by reversing the orientation of $L$;
\item[-] $L^\#$ the link obtained by reversing the orientation of the double points of $L$. 
\end{itemize}

\begin{prop}\label{prop:Symetries}
  We denote by $\widehat{HFV}_j^\circ{}^i$ the Maslov $i^\textrm{th}$ and Alexander $j^\textrm{th}$ group of cohomology associated to $\widehat{HFV}$.
  Then
  \begin{itemize}
  \item[-] $\widehat{HFV}(-L)\simeq  \widehat{HFV}(L)$;
  \item[-] $\forall i,j\in\N,\widehat{HFV}_i^j(L^\#)\simeq  \widehat{HFV}_{i-2j}^{-j}(L)$;
  \item[-] $\forall i,j\in\N,\widehat{HFV}_i^j(L^!)\simeq\widehat{HFV}_{-j}^{\circ}{}^{\hspace{-.2cm}-i+1-\ell+k}(L)$.
  \end{itemize}
\end{prop}
\begin{proof}
  Let $G$ be a grid diagram of size $n$ for $L$.

  \saut
  
  Flipping $G$ along the line $y=x$ gives a grid $-G$ for $-L$.
  This operation induces a bijection $\psi$ between the generators of $CV^-(G)$ and those of $CV^-(-G)$.
  Since it sends the $0$--resolution (resp. $1$--resolution) of a singular RoC to the $0$--resolution (resp. $1$--resolution) of its image; and since it does not affect the relative positions of dots and decorations, this bijection preserves the Maslov and Alexander gradings.
  Hence, after having substituted $-U_{O_i}$ to $U_{O_i}$ in $CV^-(-G)$ for all $i\in\llbracket 1,n\rrbracket$ and because of the antisymmetry of the convention used for the choice of the orientation of the double points, $\psi$ commutes with the differentials.

  This proves the first statement.

  \saut
  
  The second is obtained by switching the role of $\O$ and $\X$ in $-G$.
  The new grid, denoted by $G'$, describe $L$.
  We denote by $M$, $A$, $M'$ and $A'$ the Maslov and Alexander gradings associated to, respectively, $-G$ and $G'$.
  Because of their definitions and since we have switched the $\O$ and the $\X$--decorations, one can check that, for a given generator $x$,
  $$
  M'(x)-2A'(x)=M(x) + n -\ell,
  $$
  $$
  -A'(x)=A(x) + n -\ell.
  $$
  Moreover, since they count only grid polygons which contain no decoration, the differentials $\widetilde{\p}_{-G}$ and $\widetilde{\p}_{G'}$ clearly coincide.
  But switching the decorations also switchs the conventions for the orientations of double points.
  Then, according to the Proposition \ref{prop:Graded_Homologies}, for all integers $i$ and $j$, we have defined an isomorphism
  
  \begin{eqnarray}
    \left(\widehat{HFV}(L^\#)\otimes V^{\otimes (n-\ell)}\right)_i^j \simeq \left(\widehat{HFV}(-L)\otimes V^{\otimes (n-\ell)}\right)_{i-2j-n+\ell}^{-j - n +\ell}
    \label{eq:Orient}
  \end{eqnarray}
  
  Now, we denote by $v_+$ and $v_-$ the generators of, respectively, highest and lowest Alexander degree in $V^{\otimes (n-\ell)}$.
  Then, by identifying $\widehat{HFV}$ in $\widehat{HFV}\otimes V^{\otimes (n-\ell)}$ with $\widehat{HFV}\otimes v_+$ in the left-hand side of (\ref{eq:Orient}), and with $\widehat{HFV}\otimes v_-$ in its right-hand side, we obtain
  $$
  \widehat{HFV}_i^j(L^\#) \simeq \widehat{HFV}_{i-2j}^{-j}(-L).
  $$
  Finally, we use the first statement to conclude.

  \saut
  
  For the last statement, we rotate $-G$ ninety degree and get a grid diagram $G^!$ for $L^!$.
  Resolutions of singular RoCs are then swapped.
  We denote by $M^!$ and $A^!$ the Maslov and Alexander degrees associated to $G^!$.
  Since all dots and decorations are on distincts vertical and horizontal lines, one can check that, for a given generator $x$,
  $$
  M(x)+M^!(x)= 1 - n + k
  $$
  $$
  A(x) + A^!(x) = \ell - n.
  $$
  The differential induced on $\widetilde{C}(-G)$ by the rotation of $G^!$ counts the preimages of a generator under the differential given by $-G$.
  Then, after having substituted $-U_{O_i}$ to $U_{O_i}$ for all $i\in\llbracket 1,n\rrbracket$, it corresponds to the codifferential defined by $G$ on the dual basis of the usual generators.
  Moreover, the changes on double points orientations and on the convention for the choice of intersections of the winding arcs correspond.
  Finally, we obtain
  $$
  \left(\widehat{HFV}(L^!)\otimes V^{\otimes (n-\ell)}\right)_i^j \simeq \left(\widehat{HFV}^\circ(-L)\otimes V^{\otimes (n-\ell)}\right)_{-j-n+\ell}^{-i+1-n+k}.
  $$
  As above, it induces then an isomorphism
  $$
  \widehat{HFV}_i^j(L^!) \simeq \widehat{HFV}_{-j}^\circ{}^{-i+1-\ell+k}(L).
  $$
\end{proof}


%% file: Properties_Acyclicity.tex
\subsection{Acyclicity}
\label{ssec:Properties_Acyclicity}

Let $G$ be a grid presentation for a singular link $L$.
The addition of a singular loop to $L$ can be seen as replacing a regular column of $G$ with adjacent decorations by the following pattern:
$$
\dessin{2.1cm}{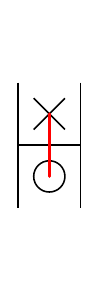} \ \leadsto \ \ \dessin{2.1cm}{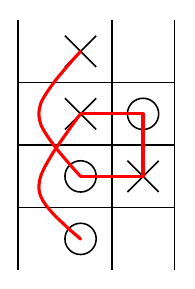}.
$$
We denote by $S$ the singular column which is added with the loop, by $(\alpha,\beta)$ a set of winding arcs for $S$, by $G_S$ the grid hence obtained and by $L_S$ the link described by $G_S$.
Up to global reversing of the double points orientations, we can assume that the double point associated to $S$ is given the orientation induced by the plane on which $G_S$ is drawn.

Now, we consider the filtration on $\big(CV^-(G_S),\widehat{\p}\big)$ which counts the number of singular RoCs different from $S$ which are positively resolved and we will prove that the associated graded complex $\big(CV^-(G_S),\widehat{\p}^*\big)$ is acyclic.

\saut

For that purpose, we define a filtration by crushing the column and the rows which have appeared with the singular loop (compare Section \ref{sssec:Stabilization}).
If $x$ is a generator of $CV^-(G_s)$ drawn on $G_s$, then, during the crushing process, we push the dot which belongs to $\alpha$ or $\beta$ to the first vertical grid line on its left.
We obtain a set of dots $\widetilde{x}$ on $G$ such that exactly one horizontal and two adjacent vertical grid lines have more than one dots.
We permute cyclically the rows and the columns in such a way that the singular horizontal grid line is the bottommost and the two vertical ones the leftmost and the second rightmost.
Then, the upper right corner of the grid is filled with an $\O$--decoration, which we denote by $O^*$, and the bottom right one by a $\X$--decoration, which we denote by $X^*$.

For all generators $x$ of $CV^-(G_s)$, we define:
$$
M_G(x):= M_{\O_G}(\widetilde{x})
$$
where $\O_G$ is the set of $O$--decorations of $G$ and $M_{.}(\ .\ )$ is the map defined in the introduction, which compares the relative positions of two planar sets of points.

\begin{figure}[!h]
  $$
  \dessinH{3.2cm}{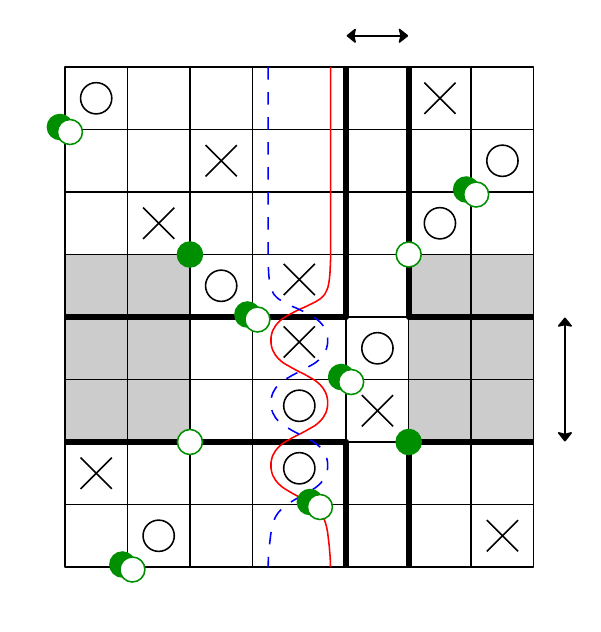}\ \begin{array}{c}\textrm{\scriptsize crushing}\\[-.1cm] \leadsto\\[-.1cm] \textrm{\scriptsize process} \end{array} \dessinH{3.2cm}{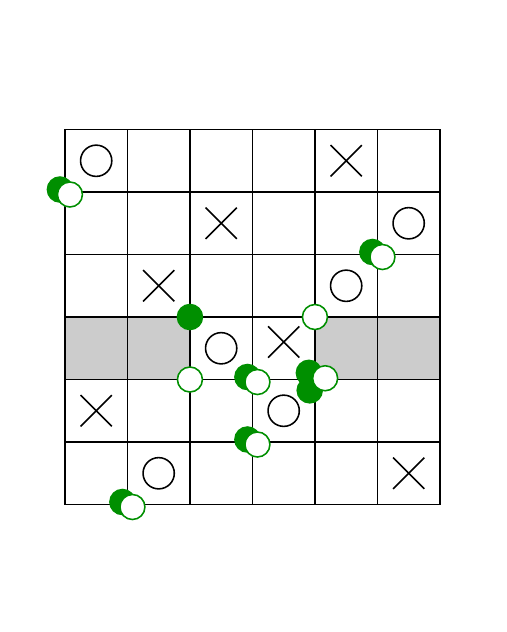} \begin{array}{c}\textrm{\scriptsize cyclic}\\[-.1cm] \leadsto\\[-.1cm] \textrm{\scriptsize perm.} \end{array} \dessinH{3.2cm}{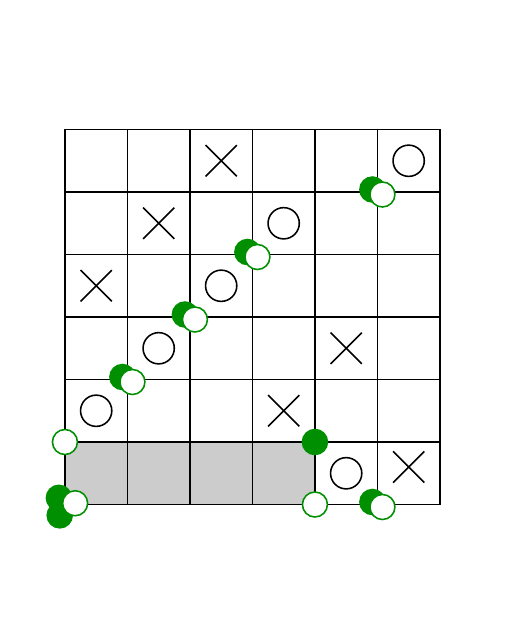}  \leadsto \ \ \begin{array}{l} M_G(x)=-6 \\[.3cm] M_G(y)=-8 \end{array}
  $$
  \caption*{Crushing rows and column:  {\footnotesize dark dots describe the initial generator $x$ while hollow ones describe the final one $y$. Grid polygons are depicted by shading.}}
  \label{fig:CrushingLoop}
\end{figure}

\begin{lemme}
  \label{lem:Crushing_Filtration_2}
  The chain complex $\big(CV^-(G_S),\widetilde{\p}^*\big)$ is filtered by $M_G$.
  The associated graded differential, denoted by $\widetilde{\p}_\gr$, corresponds to the sum over grid polygons contained in the crushed RoCs and which do not contain any decoration.
\end{lemme}
\begin{proof}
  Arguments are essentially the same as in the proof of Lemma \ref{lem:Crushing_Filtration}.
  Nevertheless, a few details are different.
  
  An empty rectangle $\rho$ on $G_s$ containing no decoration gives rise to an empty rectangle $\widetilde{\rho}$ on $G$ which may contain $O^*$ or $X^*$.
  The latter is not involved in the computation of $M_G$ but the former is.
  However, the proof of Proposition \ref{lem:Crushing_Filtration} reduces the reasoning to the study of non ripped rectangles \ie rectangles which do not intersect the leftmost column nor the uppermost row.
  Consequently, $O^*$ is not involved in the computation.
  Moreover, such rectangles are not crossed by the singular grid lines and the number of dots they contain can hence be deduced as needed.
  
  Because of the four crushed decorations, no rectangle can be ripped in four pieces.

  \saut
  
  The only remaining grid polygons involved in $\widetilde{\p}^*$ are pentagons with a peak in $S$.
  But, because of the two extremal decorations in $S$, any such polygon $\pi$ containing no decoration must lie entirely in the two crushed rows.
  If $\pi$ is connecting $x$ to $y$, then $\widetilde{x}=\widetilde{y}$ and $M_G(x)=M_G(y)$.

  \saut
  
  The filtration induced by $M_G$ is then respected and the associated graded differential is as stated. 
\end{proof}

In order to simplify the proof, we need to define a last filtration on $\big(CV^-(G_S),\widetilde{\p}_\gr\big)$.
We denote by $\alpha'$ (resp. $\beta'$) the interior of the intersection of $\alpha$ (resp. $\beta$) with the complement of the two crushed rows.
Then we can define a grading $\kappa$ for all generators $x$ of $CV^-(G_S)$ by
$$
\kappa(x)=\#(x\cap\beta') - \#(x\cap\alpha').
$$
It is easy to check that the filtration associated to $\kappa$ is respected by $\widetilde{\p}_\gr$.
We denote by $\widetilde{d}_\gr$ the associated graded differential.
It differs from $\widetilde{\p}_\gr$ by forbidding the following two grid polygons:
$$
\dessin{2.85cm}{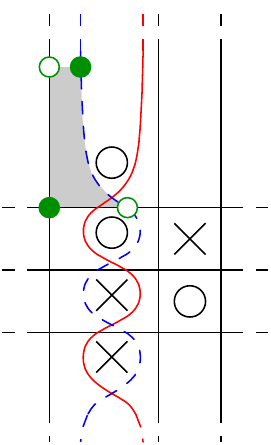} \hspace{2cm} \dessin{2.85cm}{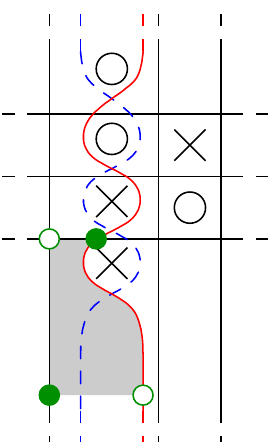}.
$$

\begin{lemme}
  \label{lem:Acyclic_Subcomplexes}
  The chain complex $\big(CV^-(G_s),\widetilde{d}_\gr\big)$ is acyclic.
\end{lemme}

Actually, the graded chain complex $(CV^-(G_s),\widetilde{d}_\gr)$ can be split into a direct sum of 35 acyclic subcomplexes.
They are listed in Appendix A of \cite{These}.


%% file: Computations.tex
\section{Computations}
\label{sec:Computations}

The author has written a program in OCaml which computes graded singular link Floer homologies $\widehat{HFV}$ with $\fract\Z/{2\Z}$--coefficients.
The results of some computations have been gathered in tables \ref{fig:Computations}(a)--\ref{fig:Computations}(c)\footnote{other computations can be found in \cite{These}}.
For each singular knot with oriented double points, we give the Poincar\'e polynomial of the associated homology.
The $t$ and $q$ variables correspond, respectively, to the Maslov and the Alexander gradings.

Orientations for double points are denoted by comparing them to the orientation induced by the plane where the diagrams are drawn.
A plus sign means that the two orientations coincide and a minus sign that they do not.

\begin{figure}[h]
  $$
  \hspace{-.3cm}
  \begin{array}{cc}
    \begin{minipage}{.5\linewidth}
      \begin{center}
        \begin{tabular}{|c|c|}
          \hline
          $\dessin{1.2cm}{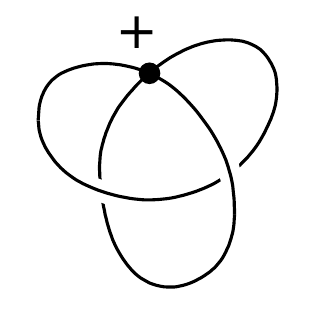}\ \dessin{1.2cm}{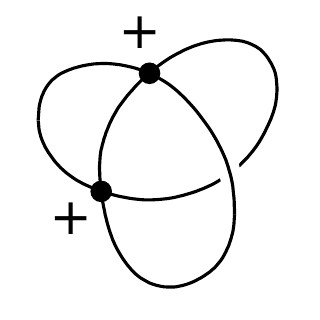}\ \dessin{1.2cm}{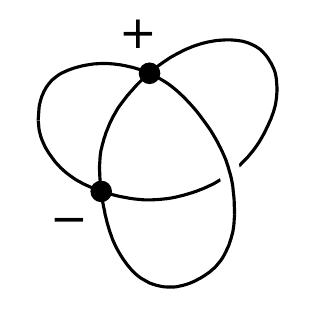}$
          &
          $q^{-1}(1+tq)^2$\\
          \hline
          $\dessin{1.2cm}{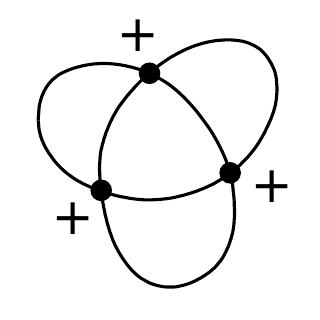}$
          &
          $0$\\
          \hline
          $\dessin{1.2cm}{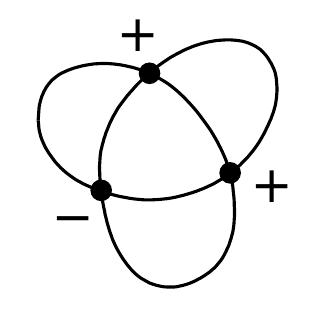}$
          &
          $q^{-1}(1+t)(1+tq)^2$\\
          \hline
        \end{tabular}
      \end{center}

      \begin{center}
        (a) Singularizations of the trefoil knot
      \end{center}
      
      \vspace{.5cm}
      
      \begin{center}
        \begin{tabular}{|c|c|}
          \hline
          $\dessin{1.68cm}{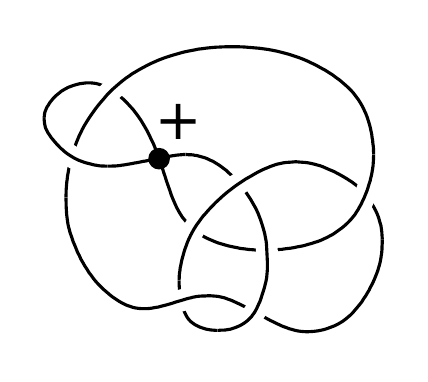}$
          &
          $t^{-2}q^{-2}(1+t+tq)(1+tq)^2$\\
          \hline
        \end{tabular}
      \end{center}

      \begin{center}
        (b) A singularization of the knot $9_{44}$
      \end{center}
      
    \end{minipage}
    
    &
    
    \begin{minipage}{.5\linewidth}
      \begin{center}
        \begin{tabular}{|c|}
          \hline
          $\dessin{1.2cm}{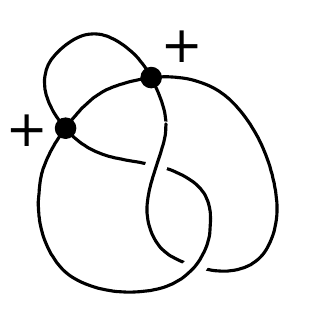}$\\
          \hline
          $\dessin{1.32cm}{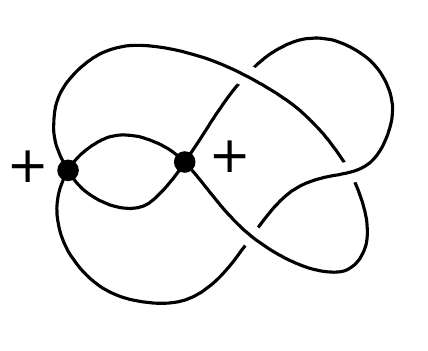}\ \dessin{1.32cm}{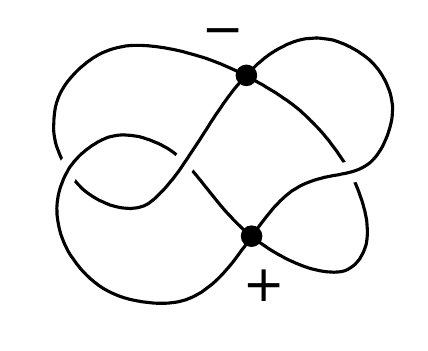}\ \dessin{1.32cm}{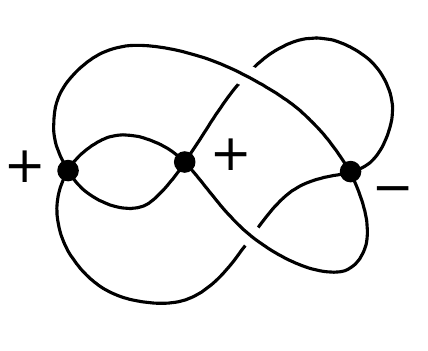}$\\
          \hline
          $\dessin{1.32cm}{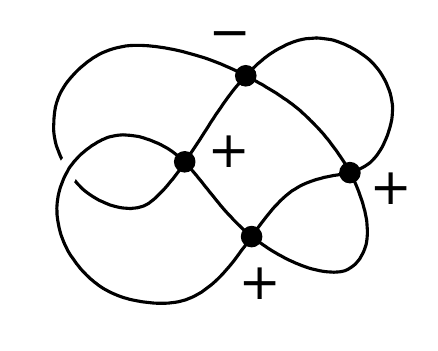}\ \dessin{1.32cm}{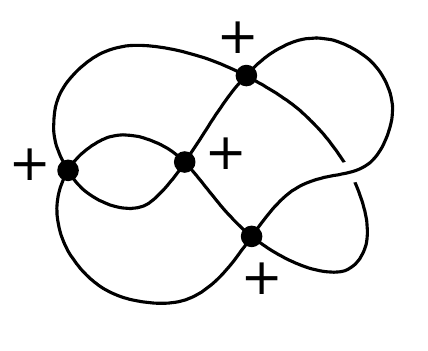}\ \dessin{1.32cm}{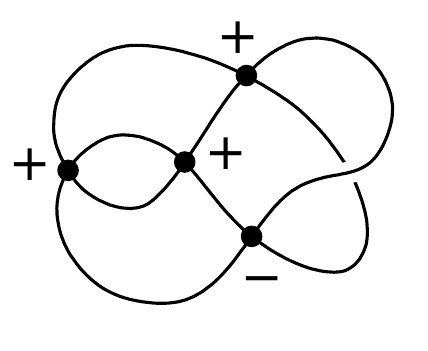}\ \dessin{1.32cm}{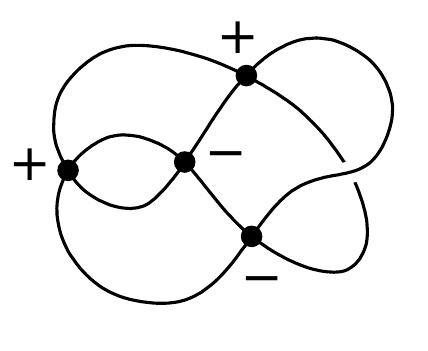}$\\
          \hline
          $\dessin{1.2cm}{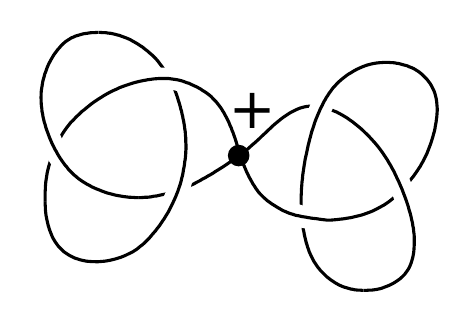}\ \dessin{1.2cm}{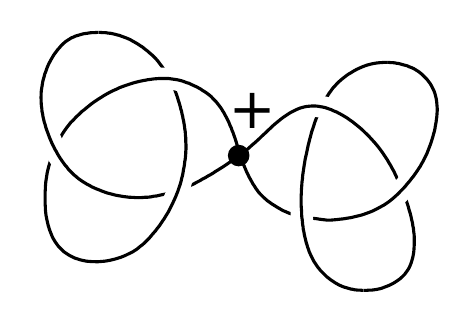}$\\
          \hline
        \end{tabular}
      \end{center}

      \begin{center}
        (c) Singular knots with null singular\\
        link Floer homology
      \end{center}
      
    \end{minipage}
  \end{array}    
  $$
  \caption{Tables of singular link Floer homologies}
  \label{fig:Computations}
\end{figure}
\saut

These tables confirm that the homology depends on the choice of an orientations for the double points.
Moreover, Figure \ref{fig:Computations}(b) provides a counterexample to the symmetry $\widehat{HF}_i^j(L)\simeq\widehat{HF}_{i-2j}^{-j}(L)$ which holds in the regular case.

\saut

Together with other computations, these lead to the following conjectures:

\begin{conjecture}
  The singular link Floer homology $\widehat{HFV}\left(\ \dessin{1cm}{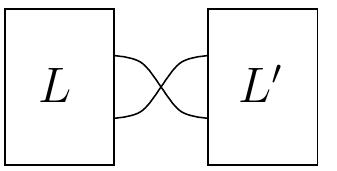}\right)$ of a singular connected sum of two links is null.
\end{conjecture}

\begin{conjecture}
  The singular link Floer homology of a purely singular link, \ie a link which admits a planar diagram with only singular crossings, is null as soon as the double points are oriented accordingly to the orientation of the plane where such a purely singular diagram is drawn.
\end{conjecture}

The first conjecture, which can be seen as a generalization of the acyclicity for links with a singular loop, is essentiel for the purpose of a categorification of Vassiliev theory.
The second one would be a first step towards some kind of finite type properties.
